\documentclass[11pt]{amsart}
\usepackage{etex}
\usepackage[english]{babel}
\usepackage[utf8]{inputenc}
\usepackage{multirow}
\usepackage{array}
\usepackage{mathtools,arydshln}\mathtoolsset{centercolon}
\usepackage{bigstrut}
\usepackage[%pagebackref,
pdftex,hyperfootnotes]{hyperref}
     \newcommand{\PARENS}[1]{\left(#1\right)}
          \newcommand{\ccases}[1]{\begin{cases}#1\end{cases}}
        \usepackage{amsmath}
        \usepackage{mathrsfs}
      \usepackage{amssymb,amsthm,amscd}
       \usepackage{upgreek}
\newcommand{\coloneq}{:=}
           \usepackage[scaled]{helvet} % ss
           \usepackage{courier} % tt
           \usepackage{eulervm} % a better implementation of the euler package (not in gwTeX)
           \normalfont
           \usepackage[T1]{fontenc}

\usepackage[toc,page]{appendix}
\usepackage[total={18cm,24.5cm},top=2cm, left=1.5cm]{geometry}
\newcommand{\ee}{\boldsymbol{e}}
\newcommand{\n}{\boldsymbol{n}}

\newcommand{\x}{\boldsymbol{x}}
\newcommand{\y}{\boldsymbol{y}}
\newcommand{\z}{\boldsymbol{z}}
\newcommand{\q}{{\boldsymbol{\alpha}}}
\newcommand{\kk}{\boldsymbol{k}}
\newcommand{\ele}{\boldsymbol{\ell}}

\renewcommand{\d}{\operatorname{d}}
\newcommand{\Exp}[1]{\operatorname{e}^{#1}}

\newcommand{\diag}{\operatorname{diag}}

\newcommand{\Ds}{\mathscr D}

\newcommand{\Z}{\mathbb{Z}}
\newcommand{\R}{\mathbb{R}}
\newcommand{\C}{\mathbb{C}}
\newcommand{\I}{\mathbb{I}}
\newcommand{\T}{\mathbb{T}}

\newtheorem{pro}{Proposition}[section]
\newtheorem{lemma}{Lemma}[section]
\newtheorem{definition}{Definition}[section]
\newtheorem{theorem}{Theorem}[section]
\newtheorem{cor}{Corollary}[section]
\numberwithin{equation}{section}

\begin{document}

	\title[Multivariate linear spectral transformations and multispectral Toda]{Linear spectral   transformations   for\\ multivariate orthogonal  polynomials and\\ multispectral Toda hierarchies}
	\author{Gerardo Ariznabarreta}
	\address{Departamento de Física Teórica II (Métodos Matemáticos de la Física), Universidad Complutense de Madrid, Ciudad Universitaria, Plaza de Ciencias nº 1, 28040-Madrid, Spain}
	\email{gariznab@ucm.es}
	\thanks{GA thanks economical support from the Universidad Complutense de Madrid  Program ``Ayudas para Becas y Contratos Complutenses Predoctorales en España 2011"}
	\author{Manuel Mañas}
	\email{manuel.manas@ucm.es}
	%\dedicatory{Corresponding author e-mail : manuel.manas@ucm.es}
	\thanks{MM thanks economical support from the Spanish ``Ministerio de Economía y Competitividad" research project MTM2012-36732-C03-01,  \emph{Ortogonalidad y aproximación; teoría y aplicaciones}; and GA and MM thanks economical support from the Spanish ``Ministerio de Economía y Competitividad" research project MTM2015-65888-C4-3,  \emph{Ortogonalidad, teoría de la aproximación y aplicaciones en física matemática}}
	\begin{abstract}
Linear spectral transformations of orthogonal polynomials in the real line, and  in particular Geronimus  transformations,  are extended to orthogonal polynomials depending on several real variables. Multivariate Christoffel--Geronimus--Uvarov  formul{\ae} for the perturbed orthogonal polynomials and their quasi-tau matrices  are found for each perturbation of the original linear functional. These  expressions are given  in terms of quasi-determinants of bordered truncated block matrices and the 1D Christoffel--Geronimus--Uvarov formul{\ae} in terms of quotient of determinants of combinations of the original orthogonal polynomials and their Cauchy transforms, are recovered.  A new multispectral Toda hierarchy of nonlinear partial differential equations,   for which the multivariate orthogonal polynomials are reductions, is proposed. This new integrable hierachy is associated with non-standard multivariate biorthogonality. Wave and Baker functions, linear equations, Lax and Zakharov--Shabat equations, KP type equations,  appropriate reductions,  Darboux/linear spectral transformations, and  bilinear equations involving linear spectral transformations are presented. Finally, the paper includes an Appendix devoted  to multivariate Uvarov transformations. Particular attention is paid to 0D-Uvarov perturbations and  also to the 1D-Uvarov perturbations, which require of the theory of Fredholm integral equations.
	\end{abstract}
	\maketitle%	\enlargethispage{1.2cm}
	
	\tableofcontents

\section{Introduction}

The aim of this paper is twofold,  in the first place we discuss an extension of the linear spectral transformation given in  \cite{Zhedanov1997Rational} for  orthogonal polynomials in the real line (OPRL) to several real variables; i. e., to complex  multivariate orthogonal polynomials in real variables (MVOPR). Secondly, to generalize the Toda hierarchy introduced in \cite{ariznabarreta2014multivariate} in the context of MVOPR, to a more general case,
that we have named multispectral Toda hierarchy. For this new integrable hierarchy, which has the MVOPR as a particular reduction, we find the multivariate  linear spectral transformations.

\subsection{Historical background and state of the art}
Elwin Christoffel, when discussing  Gaussian quadrature rules  in  \cite{Christoffel1858uber}, found explicit formul{\ae} relating   sequences of orthogonal polynomials corresponding to two measures  $\d x$ and $p(x) \d x$, with $p(x)=(x-q_1)\cdots(x-q_N)$. The so called Christoffel formula  is a classical result which can be found in a number of orthogonal polynomials textbooks, see for example \cite{Szego1939Orthogonal,Chihara1978introduction,Gautschi2004Orthogonal}.

Within a   linear functional approach to the theory of orthogonal polynomials, see \cite{Maroni1985espaces,Maroni1988calcul,Maroni1991theorie} and  \cite{Petronilho2004topological,Petronilho2006linear},   given a linear functional $u\in(\R[x])'$
its canonical or elementary Christoffel transformation is a new moment functional given by  $\hat{u}=(x-a)u$, $ a\in\mathbb{R}$,  \cite{Chihara1978introduction,Yoon2002Darboux,Bueno2004darboux}.
Its right inverse is called the Geronimus transformation, i.e., the elementary or canonical Geronimus transformation is a new moment linear functional $\check{u}$  such that $(x-a)\check{u}= u$. In this case  we can write $\check u=(x-a)^{-1}u+\xi\delta(x-a)$, where $\xi\in\R$ is a free parameter and  $\delta(x)$ is the Dirac functional supported at the point $x=a$ \cite{Geronimus1940polynomials,Maroni1990Surlasuite}. Multiple Geronimus transformations \cite{Derevyagin2014Multiple}  appear when one studies  general inner products $\langle \cdot,\cdot\rangle$ such that the multiplication by a polynomial operator $h$ is symmetric and satisfies  $\langle h(x)p(x), q(x)\rangle= \int p(x) q(x) \d\mu(x)$ for  a nontrivial probability measure  $\mu$.

In  \cite{Uvarov1969connection} Vasily Uvarov  considered the multiplication of the measure by a rational function with prescribed zeros and poles, and got determinantal formul\ae{} ---in terms of the original orthogonal polynomials and its Cauchy transformations--- for the perturbed polynomials.  That is, he worked  out  in  \S 1 the linear spectral transformation without masses. Moreover, he also introduced  in \S 2 the so called canonical Uvarov transformation  the moment linear functional $u$ is transformed into $\check u= u + \xi \delta(x-a)$ with $\xi\in\R$, and presented a determinantal formul\ae{} ---in terms of  kernel polynomials--- for several masses of this type, $ \check u= u + \xi_1 \delta(x-a_1)+\cdots+\xi_N\delta(x-a_N)$.

 The Stieljes  function
$F(x):=\sum_{n=0}^{\infty}\frac{\langle u,x^n\rangle}{x^{n+1}}$ of a linear functional  $u\in(\R[x])'$
is relevant in the theory of orthogonal  polynomials for several reasons, is in particular  remarkable  its close relation with  Padé approximation theory, see \cite{Brezinskin1980Pade,Karlberg1990Pade}. Alexei Zhedanov studied in \cite{Zhedanov1997Rational} the following  rational spectral transformations  of the Stieltjes function
\begin{align*}
F(x)\mapsto \tilde F(x) = \frac{A(x)F(x) + B(x)}{C(x)F(x) + D(x)},
\end{align*}
as a natural extension of the bove mentioned three canonical transformations. Here $A(x), B(x), C(x)$ and $D(x)$ are polynomials  such that
$\tilde F(x)=\sum_{n=0}^{\infty}\frac{\langle \tilde  u,x^n\rangle}{x^{n+1}}$ is a new Stieljes function.
Linear spectral transformations correspond to the  choice $c(x)=0$, of which particular cases are the canonical Christoffel transformations $\tilde F(x) = (x-a)F(x)-F_0 $ and the canonical Geronimus transformation of $\tilde F(x) = \frac{F(x)+\tilde F_0}{x-a}$. Every linear spectral transformation of a  moment functional is given as a composition of Christoffel and Geronimus transformations \cite{Zhedanov1997Rational}.

These transformations are refered generically as Darboux transformations, a name coined in the context of integrable systems in \cite{Matveev1978Differential}.  Gaston Darboux, when studying the  Sturm--Liouville theory in  \cite{Darboux1882proposition},  explicitly   treated these transformations, which he obtained by a simplification of a geometrical transformation founded previously by Théodore Moutard \cite{Moutard1878reciprocal}.
In the OPRL framework, such a factorization of Jacobi matrices has been studied in \cite{Bueno2004darboux,Yoon2002Darboux}, and also played a key role in the study of bispectrality  \cite{Grunbaum1996Orthogonal,Grunbaum2011Darboux}.
In the differential geometry context, see \cite{Eisenhart1943Transformations},  the Christoffel, Geronimus, Uvarov and linear spectral transformations are related to geometrical transformations like  the Laplace, Lévy, adjoint Lévy and the fundamental Jonas transformations.

Regarding orthogonal polynomials in several variables we refer the reader to the excellent monographs \cite{Dunkl2014Orthogonal,Xu1994multivariate}. Milch \cite{Milch1968growth} and Karlin and McGregor  \cite{Karlin1975growth} considered multivariate Hahn and Krawtchouk polynomials in relation with  growth birth and death processes. Since 1975 substantial developments have been achieved, let us mention the spectral properties of these multivariate Hahn and Krawtchouk polynomials, see \cite{Geronimo2006Factorization}.
 Orthogonal polynomials and cubature formul{\ae} on the unit ball, the standard simplex, and the unit sphere were studied in \cite{Xu2001Orthogonal} finding  a strong connections between both themes. The common zeros of multivariate orthogonal polynomials were discussed in \cite{Xu1994Common} where relations with  higher dimensional quadrature problems were found. A description of orthogonal polynomials  on the bicircle and polycircle and their relation to bounded analytic functions on the polydisk is given in \cite{Knese2008Polynomials}, here a Christoffel--Darboux like formula, related in this  bivariate case with  stable polynomials, and Bernstein--Szeg\H{o} measures are used, allowing for a new proof of Ando  theorem in operator theory. Bivariate orthogonal polynomials linked to a moment functional satisfying the two-variable  Pearson type differential equation and an extension of some of the characterizations of the classical orthogonal polynomials in one variable was discussed in \cite{Fernandez2005Weak}; in the paper \cite{Fernandez2010Krall}  an analysis of a bilinear form obtained by adding a Dirac mass to a positive definite moment functional  in several variables is given.

Darboux transformations for multivariate orthogonal polynomials were first studied in \cite{ariznabarreta2014multivariate,ariznabarreta2015darboux} in the context of a Toda hierarchy. These transformations are the multidimensional extensions of the Christoffel transformations. In \cite{ariznabarreta2015darboux} we presented for the first time a multivariate extension of the classical 1D Christoffel formula, in terms of quasi-determinants \cite{Gelfand1991Determinants,Gelfand1995Noncommutative,Olver2006Multivariate}, and poised sets \cite{Olver2006Multivariate,ariznabarreta2015darboux}. Also in this general multidimensional framework we have studied in \cite{ariznabarreta2015multivariate} multivariate Laurent polynomials orthogonal with respect to a measure supported in the unit torus, finding in this case the corresponding Christoffel formula.
In \cite{Alfaro2014linearly} linear relations between two families of multivariate orthogonal polynomials were studied. Despite that \cite{Alfaro2014linearly} does not deal with  Geronimus formul{\ae}, it deals with linear connections among two families of orthogonal polynomials, a first step towards a connection formul{\ae} for the multivariate  Geronimus transformation.

 Sato \cite{Sato1981Soliton,Sato1983Soliton} and Date, Jimbo, Kashiwara and Miwa  \cite{Date1981operator,Date1982transformation,Date1983transformation}
 introduced geometrical tools, like the infinite-dimensional Grasmannian and infinite dimensional Lie groups an Lie algebras, which have becomed essential,
 in the description of integrable hierarchies.  We also mention \cite{Mulase1984Complete}, were the  factorization problems, dressing procedure, and linear systems where shown to be  the keys for
integrability. Multicomponent versions of the integrable Toda equations \cite{Ueno1984TodaI,Ueno1984TodaII,Ueno1984TodaIII} played a prominent role in the connection with orthogonal polynomials and differential geometry. In \cite{Bergvelt1988tau,Bergvelt1995partitions,Kac2003KP,manas2000dressing,manas2000dressing2} multicomponent versions of the KP hierachy were analyzed,  while in \cite{manas2009multicomponent,manas2009multicomponent1} we can find a  study of  the multi-component Toda lattice hierarchy, block Hankel/Toeplitz reductions, discrete flows, additional symmetries and dispersionless limits.
In \cite{Adler2008Moment,alvarez2011multiple}  the relation of the multicomponent KP--Toda  with mixed multiple orthogonal polynomials was discussed.

Adler and  van Moerbeke  showed  the prominent role played by the  Gauss--Borel factorization problem for  understanding  the strong bonds between orthogonal polynomials and integrable systems. In particular, their studies on the 2D Toda hierarchy --what they called the discrete KP hierarchy--  neatly established  the deep connection among standard orthogonality of polynomials and integrability of nonlinear equations of Toda type, see \cite{Adler1997group,Adler1998Toda,Adler1999generalized,Adler1999vertex,Adler2001Integrals} and also \cite{Felipe2001Algebraic}.
Let us also mention that multicomponent Toda  systems or non-Abelian versions of Toda equations with matrix orthogonal polynomials was studied, for example,  in \cite{Miranian2005Matrix,alvarez2010multicomponent} (on the real line) and in \cite{Miranian2009Matrix,alvarez2013orthogonal} (on the unit circle).

The approach to  linear spectral  transformations and Toda hierarchies used  in this paper, which is based on the Gauss--Borel factorization problem, has been used before in different contexts. We have  connected integrable systems  with orthogonal polynomials of diverse types:
\begin{enumerate}
	\item  As already mentioned, mixed multiple orthogonal polynomials and multicomponent Toda was analyzed in \cite{alvarez2011multiple}.
	\item Matrix orthogonal Laurent polynomials on the circle and CMV orderings were considered \cite{ariznabarreta2014matrix}
	\item The Christoffel transformation has been recently discussed for matrix orthogonal polynomials in the real line \cite{alvarez2015Christoffel}.
\end{enumerate}

\subsection{Results and layout of the paper}

First, we complete this introduction with some background material from \cite{ariznabarreta2014multivariate}. Then,  in \S  \ref{S:2} we discuss  the   Geronimus type transformation for multivariate orthogonal polynomials. We introduce  the resolvents and find the connection formul{\ae}. The multivariate extension of the Geronimus determinantal formula depends on the introduction of a semi-infinite matrix $R$, that for the 1D case is encoded in the Cauchy transforms of the OPRL, the second kind functions. However, no such connection exists in this more general scenario, and the multivariate Cauchy transform of the MVOPR does not provide the necessary aid for  finding  the  multivariate  formula for Geronimus transformations (aid which is provided by  the semi-infinite matrix $R$). Then, we end the section by discussing the 1D reduction and recovering the Geronimus results  \cite{Geronimus1940polynomials}. A similar approach can be found in \S \ref{S:3}  for the linear spectral  for which we present a multivariate quasi-determinantal Christoffel--Geronimus--Uvarov formula \cite{Zhedanov1997Rational}, and we give a brief discussion of the existence of poised sets.

In \cite{ariznabarreta2014multivariate}  we  considered semi-infinite matrices having the adequate symmetries, that we call multi-Hankel, so that a multivariate  moment functional or moment semi-infinite matrix appeared. In section \ref{S:4} we are ready to abandon this more comfortable  MVOPR situation and explore different scenarios by  assuming that  $G$ could be arbitrary, as far it is Gaussian factorizable. We are dealing with perturbations of non-standard multivariate biorthogonality. We first give the general setting for this integrable hierarchy,  that we have named  multi-spectral Toda lattice hierarchy,  finding the corresponding Lax and Zakharov--Shabat equations and the role played by the Baker and adjoint Baker functions. Some reductions, like the multi-Hankel that leads to dynamic MVOPR, and extensions of it are presented. We also consider the action of the discussed multivariate linear spectral transformations and find the Christoffel--Geronimus--Uvarov formula in this broader scenario. To end the paper, we find  generalized bilinear equations that involve  linear spectral transformations.

We have also included an appendix to discuss multivariate Uvarov transformations.  For the 0D-Uvarov transformation,  which can be considered an immediate extension of the results of Uvarov \cite{Uvarov1969connection},  connection formulas are found. The general situation is discussed in terms of jets, we then particularize to mass perturbation, which for the OPRL case appears in \cite{Uvarov1969connection} and in the multivariate case in \cite{Teresa0}, and to a dipole perturbation.  The more appealing 1D-Uvarov perturbation is also discussed, and a connection formula is  given in terms of a solution of an integral Fredholm equation.

\subsection{Preliminary material}\label{MV}
Following \cite{ariznabarreta2015darboux},  a brief account of complex multivariate  orthogonal polynomials  in a  $D$-dimensional real space (MVOPR) is given.  Cholesky factorization of a semi-infinite moment matrix will be keystone to built such objects.  Consider $D$ independent real variables $\x=\left(x_1,x_2,\dots,x_D \right)^\top\in \Omega\subseteq\mathbb{R}^D$, and the corresponding ring of complex multivariate polynomials $\C[\x]\equiv \C[x_1,\dots,x_D]$.
Given a multi-index $\q=(\alpha_1,\dots,\alpha_D)^\top \in\Z_+^{D}$ of non-negative integers  write $\x^{\q}=x_1^{\alpha_1}\cdots x_D^{\alpha_D}$ and say that the  length of $\q$ is $|\q|\coloneq  \sum_{a=1}^{D} \alpha_a$. This length induces a total ordering of monomials:  $\x^{\q}<\x^{\q'}\Leftrightarrow|\q|<|\q'|$. For each non-negative integer $k\in\Z_+$  introduce the set
\begin{align*}
[k]\coloneq \{\q\in \Z_+^{D}: |\q|=k\},
\end{align*}
built up with those vectors  in the lattice $\Z_+^D$ with a given length $k$.
The graded  lexicographic order  for $\q_1,\q_2\in [k]$ is
\begin{align*}
\q_1>\q_2 \Leftrightarrow \exists p\in \Z_+ \text{ with } p<D \text{ such that } \alpha_{1,1}=\alpha_{2,1},\dots,\alpha_{1,p}=\alpha_{2,p} \text{ and } \alpha_{1,p+1}<\alpha_{2,p+1},
\end{align*}
and if $\q^{(k)}\in[k]$ and $\q^{(l)}\in[l]$, with $k<l$ then $\q^{(k)}<\q^{(l)}$.
Given the set of integer vectors of length $k$  use the  lexicographic order and write
\begin{align*}
[k]=\big\{\q_1^{(k)},\q_2^{(k)},\dots,\q^{(k)}_{|[k]|}\big\} \text{ with } \q_a^{(k)}>\q_{a+1}^{(k)}.
\end{align*}
Here $|[k]|$ is the cardinality of the set $[k]$, i.e., the number of elements in the set.
This is the dimension of the linear space of homogenous multivariate polynomials of total degree $k$.
 Either counting weak compositions or multisets one obtains the multi-choose number,  $|[k]|= \big(\!\binom{D}{k} \!\big) = {D+k-1 \choose k} $.
The dimension of the linear space $\C_k[x_1,\dots,x_D]$ of multivariate polynomials of degree less or equal  to $k$ is
\begin{align*}
N_{k}=1+|[2]|+\dots+|[k]|=\binom{D+k}{D}.
\end{align*}

The  vector of monomials
\begin{align*}
\chi&\coloneq \PARENS{\begin{matrix}\chi_{[0]} \\ \chi_{[1]} \\ \vdots \\ \chi_{[k]} \\ \vdots \end{matrix}}
& \mbox{where} & &
\chi_{[k]}&\coloneq  \PARENS{\begin{matrix} \x^{\q_1} \\  \x^{\q_2} \\\vdots \\ \x^{\q_{|[k]|}} \end{matrix}},&
\chi^*&\coloneq  \Big(\prod_{a=1}^D x_a^{-1}\Big)\chi(x_1^{-1},\dots,x_D^{-1}).
\end{align*}
will be useful.
Observe that for $k=1$ we have that the vectors $\q^{(1)}_a=\ee_a$ for $a\in\{1,\dots,D\}$ form  the canonical basis of $\R^D$, and for any $\q_j\in[k]$ we have $\q_j=\sum_{a=1}^D \alpha_{j}^a\ee_a$ .
For the sake of simplicity unless needed we will drop off the super-index and write $\q_j$ instead of $\q^{(k)}_j$, as it is understood that $|\q_j|=k$.

The dual space of the symmetric tensor powers is isomorphic to the set of symmetric multilinear functionals on $\C^D$, $\big(\text{Sym}^k(\C^D)\big)^*\cong S((\C^D)^k,\C)$.
Hence,
homogeneous polynomials of a given total degree  can be identified with symmetric tensor powers.
Each multi-index $\q\in[k]$ can be thought as a weak $D$-composition of $k$ (or weak composition in  $D$ parts), $k=\alpha_{1}+\dots+\alpha_{D}$.
Notice that these weak compositions may be considered as multisets and that, given a linear basis $\{\ee_a\}_{a=1}^D$ of $\C^D$ one has the linear basis $\{\ee_{a_1}\odot\cdots\odot \ee_{a_k}\}_{\substack{1\leq a_1\leq\cdots\leq a_k\leq D\\ k\in\Z_+}}$ for the symmetric power $\operatorname{S}^k(\C^D)$, where the multisets $1\leq a_1\leq\cdots\leq a_k\leq D$ have been used. In particular, the vectors of this basis $\ee_{a_1}^{\odot M(a_1)}\odot\cdots\odot \ee_{a_p}^{\odot M(a_p)}$, or better its duals $(\ee_{a_1}^*)^{\odot M(a_1)}\odot\cdots\odot (\ee_{a_p}^*)^{\odot M(a_p)}$ are in bijection with monomials of the form $x_{a_1}^{M(a_1)}\cdots x_{a_p} ^{M(a_p)}$.
The  lexicographic order can be applied  to $\big(\C^D\big)^{\odot k}\cong \C^{|[k]|}$,  then a linear basis of $\operatorname{S}^k(\C^D)$ is the ordered set $B_c=\{\ee^{\q_1},\dots,\ee^{\q_{|[k]|}}\}$ with $\ee^{\q_j}\coloneq \ee_1^{\odot \alpha_{j}^1}\odot\dots\odot \ee_{D}^{\odot \alpha_{j}^D}$ so that
$\chi_{[k]}(\x)=\sum_{i=1}^{|[k]|}\x^{\q_j}\ee^{\q_j}$.  For more information see \cite{Comon2008Symmetric,Federer1969Geometric,Vinberg2003algebra}.

Consider semi-infinite matrices $A$ with a block or partitioned structure induced by the graded reversed lexicographic order
\begin{align*}
A&=\PARENS{\begin{matrix}
	A_{[0],[0]} & A_{[0],[1]} &  \cdots  \\
	A_{[1],[0]} & A_{[1],[1]} &  \cdots \\
	\vdots                &                 \vdots         &  \\
	\end{matrix}}, &
A_{[k],[\ell]}&=\PARENS{\begin{matrix}
	A_{\q^{(k)}_1,\q^{(\ell)}_1} &   \dots & A_{\q^{(k)}_1,\q^{(\ell)}_{|[l]|} }\\
	\vdots & & \vdots\\
	A_{\q^{(k)}_{|[k]|},\q^{(\ell)}_1} &  \dots & A_{\q^{(k)}_{|[k]|},\q^{(\ell)}_{|[l]|} }
	\end{matrix}} \in\C^{|[k]|\times |[l]|}.
\end{align*}
Use the notation $0_{[k],[\ell]}\in\C^{|[k]|\times|[l]|}$ for the rectangular zero matrix, $0_{[k]}\in\C^{|[k]|}$ for the zero vector, and $\I_{[k]}\in\C^{|[k]|\times|[k]|}$ for the identity matrix. For the sake of simplicity  just write $0$ or $\I$ for the zero or identity matrices, and assume that the sizes of these matrices are the ones indicated by their position in the partitioned matrix.

The vector space of complex  multivariate polynomials $\C_k[\x]$ in $D$ real variables of degree less or equal to $k$ with the norm
$	\Big\| \sum_{|\q|\leq k} P_\q\x^\q\Big\|_n\coloneq \sum_{|\q|\leq k} |P_\q|$,
gives a nesting of Banach spaces $\C_n[\x]\subset \C_{n+1}[\x]$ whose inductive limit gives a topology to the space $\C[\x]$.
The elements of the algebraic dual $u\in(\C[\x])^*$, which are called linear functionals,
are linear maps  $u:\C[\x]\to\C$; the notation $P(\x)\overset{u}{\mapsto}\langle u,P(\x)\rangle$ will be used. Two polynomials $P(\x),Q(\x)\in\C[\x]$ are said orthogonal with respect to $u$ if $\langle u,{P(\x)}Q(\x)\rangle =0$.
The topological dual $(\C[\x])'$ has the dual weak topology characterized by the semi-norms $\big\{\|\cdot\|_P\big\}_{P(\x)\in\C[\x]}$, $\|u\|_P\coloneq|\langle u,P(\x)\rangle|$. This family of seminorms is equivalent to the family of seminorms given by $\|u\|^{(k)}\coloneq \sup_{|\q|=k}|\langle u, \x^\q\rangle|$. Moreover, the topological dual $(\C[\x])'$ is a Fréchet space and $(\C[\x])'=(\C[\x])^*$ and every linear functional is continuous.
Linear functionals can be multiplied by polynomials $\langle Qu,P(\x)\rangle
\coloneq \langle u, Q(\x)P(\x)\rangle$, $\forall P(\x)\in\C[\x]$. It can be also shown, that in this case the space of generalized functions $(\mathbb C[\x])'$ coincide with the space of formal series
$\mathbb C[\![\x]\!]$.  For more information regarding  linear functional's approach to orthogonal polynomials see
\cite{Maroni1985espaces,Maroni1988calcul} and
\cite{Petronilho2006linear,Petronilho2004topological}.

However,  we  need to deal with generalized functions with a  support and the linear functionals we have discussed so far are not suitable for that.
We proceed to discuss several possibilities to overcome this problem.
The space of distributions is a space of  generalized functions when the fundamental functions space is the complex valued smooth functions of compact support $\mathcal D:=C_0^\infty(\mathbb R^D)$, the space of test functions, see \cite{Schwartz,gelfand-distribu1,gelfand-distribu2}. Now,  there is a clear meaning for the set of zeroes of a  distribution $u\in\mathcal D'$, $u$ is zero in a domain $\Omega\subset \mathbb R^D$ if for any fundamental function $f(\x)$ with support in $\Omega$ we have $\langle u, f\rangle =0$. The complement, which is closed, is the support $\operatorname{supp} u$ of the distribution $u$.   Distributions of compact support, $u\in\mathcal E'$, are  generalized functions with  fundamental functions space is the topological space of complex valued smooth  functions $\mathcal E=C^\infty(\mathbb R^D)$. Thus, as $\mathbb C[\x]\subsetneq \mathcal E$ we have $\mathcal E'\subsetneq (\mathbb C[\x])'\cap \mathcal D'$. These  distributions of compact support is a first example of an appropriate framework for the consideration of polynomials and supports simultaneously. More general setting appears within the space of tempered distributions $\mathcal S'$  --which are distributions, $\mathcal S'\subsetneq\mathcal D'$--. Now, the fundamental functions space is given by the Schwartz space $\mathcal S$ of complex valued fast decreasing functions, see \cite{Schwartz,gelfand-distribu1,gelfand-distribu2}.  Then, we can consider the space of fundamental functions of smooth functions of slow growth $\mathcal O_M\subset \mathcal E$,  whose elements are smooth functions having all its derivatives  bounded by a polynomial of certain degree.  As  $\C [\x],\mathcal S\subsetneq \mathcal O_M$,  for the corresponding set of generalized functions  we find that  $\mathcal O_M'\subset (\mathbb C[\x])'\cap \mathcal S'$.  Thus, these distributions give a second suitable framework. Finally, for a third suitable framework we need to introduce bounded distributions. Let us consider as  space of fundamental functions, the linear space $\mathcal B$ of bounded smooth functions, i.e.,  with all its derivatives in $L^\infty(\R^D)$,  being the corresponding space of generalized functions $\mathcal B'$ the bounded distributions (not to be confused with compact support). Notice that, as $\mathcal D\subsetneq \mathcal B$ we have that bounded distributions are distributions $\mathcal B'\subsetneq \mathcal D'$. Then, we consider the space of fast decreasing distributions $\mathcal O_c'$ given by those distributions $u\in\mathcal D'$ such that for each positive integer $k$, we have
$\big(\sqrt{1+(x_1)^2+\cdots+(x_D)^2}\big)^ku\in\mathcal B'$ is a bounded distribution.
Any polynomial $P(\x)\in\C[\x]$, with $\deg P=k$, can be written as 
\begin{align*}
P(\x)&=\Big(\sqrt{1+(x_1)^2+\cdots+(x_D)^2}\Big)^k F(\x), & F(\x)&=\frac{P(\x)}{\big(\sqrt{1+(x_1)^2+\cdots+(x_D)^2)}\big)^k}\in\mathcal B.
\end{align*} 
Therefore, given a fast decreasing distribution $u\in\mathcal O_c'$ we may consider 
\begin{align*}
\langle u,P(\x)\rangle =\left\langle\Big(\sqrt{1+(x_1)^2+\cdots+(x_D)^2}\Big)^ku, F(\x)\right\rangle
\end{align*}
which makes sense as $\big(\sqrt{1+(x_1)^2+\cdots+(x_D)^2}\big)^ku\in\mathcal B', F(\x)\in\mathcal B$.  Thus,  $\mathcal O'_c\subset  (\C[\x])'\cap \mathcal D'$.
Moreover  it can be proven that $\mathcal O_M'\subsetneq \mathcal O_c'$, see  
\cite{Maroni1985espaces}.
 Summarizing this discussion, we have found three  generalized function spaces suitable for the discussion of polynomials and supports simultaneously:
\begin{align*}
\mathcal E'\subset \mathcal O_M'\subset \mathcal O_c' \subset \big((\C[\x])'\cap \mathcal D'\big).
\end{align*}

\begin{definition}\label{moment}
	Associated with the linear functional $u\in(\C[x])'$ define the following moment matrix
	\begin{align*}
	G&\coloneq \langle u, \chi(\x)\big(\chi(\x)\big)^\top\rangle.
	\end{align*}
In block form can be written as
	\begin{align*}
	G=  \PARENS{\begin{matrix}
		G_{[0],[0]} & G_{[0],[1]} &  \dots \\
		G_{[1],[0]} & G_{[1],[1]} &  \dots \\
		\vdots                &   \vdots              &
		\end{matrix}}.
	\end{align*}
	Truncated  moment matrices are given by
	\begin{align*}
	G^{[l]}&\coloneq
	\PARENS{\begin{matrix}
		G_{[0],[0]} &  \cdots & G_{[0],[l-1]} \\
		\vdots                        &   & \vdots \\
		G_{[l-1],[0]}  &  \cdots & G_{[l-1],[l-1]}
		\end{matrix}}.
	\end{align*}
\end{definition}
Notice that from the above definition we know that 
\begin{pro}
	The moment matrix is a symmetric  matrix, $G=G^\top$.
\end{pro}
This result  implies that a Gauss--Borel factorization of it, in terms of lower unitriangular%\footnote{Lower triangular with the block diagonal populated by identity matrices.}
and upper triangular matrices, is a Cholesky factorization.

In terms of quasi-determinants, see \cite{Gelfand2005,Olver2006Multivariate}, we have
\begin{pro}\label{qd1}
	If the last quasi-determinants $ \Theta_*(G^{[k+1]})$, $k\in\{0,1,\dots\}$, of the truncated moment matrices are invertible
	the Cholesky factorization
		\begin{align}\label{cholesky}
		G&=S^{-1} H S^{-\top},
		\end{align}
with
		\begin{align*}
		S^{-1}&=\PARENS{\begin{matrix}
			\I    &             0                &  0                      &  \cdots            \\
			(S^{-1})_{[1],[0]}        & \I&     0                      &   \cdots         \\
			(S^{-1})_{[2],[0]}        & (S^{-1})_{[2],[1]} & \I&      \\
			\vdots                       &        \vdots                  &                               &\ddots
			\end{matrix}}, &
		H&=\PARENS{\begin{matrix}
			H_{[0]}           &   0         &     0         \\
			0                 & H_{[1]} &   0             &    \cdots       \\
			0                  &    0            & H_{[2]} &                        \\
			\vdots   &   \vdots &              &     \ddots       \\
			\end{matrix}},
		\end{align*}
	 and Hermitian quasi-tau matrices $H_{[k]}=(H_{[k]})^\top$, can be performed. Moreover,  the rectangular blocks can be expressed in terms of  last quasi-determinants of truncations of the moment matrix
	\begin{align*}
	H_{[k]}&=\Theta_*(G^{[k+1]}), &
	(S^{-1})_{[k],[l]}&=\Theta_*(G^{[l+1]}_k)\Theta_*(G^{[l+1]})^{-1}.
	\end{align*}
\end{pro}
\begin{definition}
	The monic MVOPR associated to the linear functional $u$  are
	\begin{align}\label{eq:polynomials}
	P(\x)&=S\chi(\x) =\PARENS{\begin{matrix}
		P_{[0]}(\x)\\
		P_{[1]}(\x)\\
		\vdots
		\end{matrix}}, & P_{[k]}(\x)&=\sum_{\ell=0}^k S_{[k],[l]} \chi_{[l]}(\x) =\PARENS{\begin{matrix}
		P_{\q^{(k)}_1}(\x)\\
		\vdots\\
		P_{\q^{(k)}_{|[k]|}}(\x)
		\end{matrix}},&
	P_{\q^{(k)}_i}&=\sum_{l=0}^k\sum_{j=1}^{|[l]|} S_{\q^{(k)}_i,\q^{(l)}_j} \x^{\q^{(l)}_j}.
	\end{align}
\end{definition}

Observe that $P_{[k]}(\x)=\chi_{[k]}(\x)+\beta_{[k]}\chi_{[k-1]}(\x)+\cdots$ is a vector constructed with the polynomials $P_{\q_i}(\x)$ of degree  $k$, each of which has only one monomial of degree $k$; i. e., we can write $P_{\q_i}(\x)=\x^{\q_i}+Q_{\q_i}(\x)$, with $\deg Q_{\q_i}<k$. Here $\beta$ is th semi-infinite matrix with all its elements being zero but for its first subdiagonal   $\beta=\operatorname{subdiag}_1(\beta_{[1]},\beta_{[2]},\dots)$ with coefficients given by $\beta_{[k]}\coloneq S_{[k],[k-1]}$.
\begin{pro}[Orthogonality relations]
	The MVOPR  satisfy
		\begin{align*}
		\big\langle u, P_{[k]}(\x)\big(P_{[l]}(\x)\big)^{\top}\big\rangle&=\delta_{k,l}H_{[k]}.
		\end{align*}
		which implies
	\begin{align}
	\big\langle u, P_{[k]}(x)\big(P_{[l]}(\x)\big)^{\top}\big\rangle&=\big\langle u, P_{[k]}(\x)\big(\chi_{[l]}(\x)\big)^{\top}\big\rangle=0, &
	l&=0,1,\dots,k-1,\label{orth}\\
	\big\langle u, P_{[k]}(\x) \big(P_{[k]}(\x)\big)^{\top}\big\rangle&=	\big\langle u, P_{[k]}(\x)\big(\chi_{[k]}(\x)\big)^{\top}\big\rangle=H_{[k]}.\label{H}
	\end{align}
\end{pro}
Therefore,  the following orthogonality conditions
\begin{align*}
	\langle u,  P_{\q^{(k)}_i} (\x) P_{\q^{(l)}_j}(\x)\rangle&=	\langle u, P_{\q^{(k)}_i}(\x) \x^{\q^{(l)}_j}\rangle=0,
\end{align*}
are fulfilled for $l\in\{0,1,\dots,k-1\}$, $i\in\{1,\dots,|[k]|\}$ and $j\in\{1,\dots,|[l]|\}$,
with the normalization conditions
\begin{align*}
	\langle u, P_{\q_i} (\x) P_{\q_j}(\x)\rangle&=	\langle u, P_{\q_i}(\x)  \x^{\q_j}\rangle=H_{\q_i,\q_j}, &  i,j&\in\{1,\dots,|[k]|\}.
\end{align*}

\begin{definition}\label{def: shift_jacobi}
	The spectral matrices are  given by
	\begin{align*}
	\Lambda_a&=\PARENS{\begin{matrix}
		0   & (\Lambda_a)_{[0],[1]} & 0 & 0 &\cdots\\
		0        & 0 &(\Lambda_a)_{[1],[2]} & 0 &\cdots\\
		0                &  0    &          0       & (\Lambda_a)_{[2],[3]}  &  \\
		0                &  0    &          0             &               0         &\ddots  \\
		\vdots        &  \vdots    &  \vdots         &\vdots
		\end{matrix}},& a&\in\{1,\dots,D\},
	\end{align*}
	where the entries in the first block superdiagonal are
	\begin{align*}
	(\Lambda_a)_{\q^{(k)}_i,\q^{(k+1)}_j}&=\delta_{\q^{(k)}_i+\ee_a,\q^{(k+1)}_j},&
	a&\in\{1,\dots, D\}, &
	i&\in\{1,\dots,|[k]|\},&
	j&\in\{1,\dots,|[k+1]|\},
	\end{align*}
	and the associated vector
	\begin{align*}
	 \boldsymbol \Lambda&\coloneq (\Lambda_1,\dots,\Lambda_D)^\top.
	 \end{align*}
	 Finally, we introduce the Jacobi matrices
	 \begin{align}
	 \label{eq:jacobi}
J_a\coloneq &S\Lambda_a S^{-1}, & a&\in\{1,\dots,D\},
	 \end{align}
	 and the Jacobi  vector
	 \begin{align*}
	 \boldsymbol J=(J_1,\dots,J_D)^\top.
	 \end{align*}
\end{definition}
\begin{pro}\label{pro:Lambda}
	\begin{enumerate}
		\item  The spectral matrices commute among them
		\begin{align*}
		\Lambda_a\Lambda_b&=\Lambda_b\Lambda_a,& a,b&\in\{1,\dots,D\}.
		\end{align*}
		\item The spectral  properties
		\begin{align}\label{eigen}
		\Lambda_a\chi(\x)&= x_a \chi(\x),& a&\in\{1,\dots,D\}
		\end{align}
		hold.
		\item The moment matrix $G$ satisfies
		\begin{align}\label{eq:symmetry}
		\Lambda_a G&= G \big(\Lambda_a\big)^\top,& a&\in\{1,\dots,D\}.
		\end{align}
		\item The Jacobi matrices $J_a$ are block tridiagonal and satisfy
		\begin{align*}
		J_aH=&HJ_a^\top, & a\in\{1,\dots,D\}.
		\end{align*}
	\end{enumerate}
\end{pro}
%Using these properties one derives the three term relations or the  Christoffel--Darboux formul{\ae}, see \cite{ariznabarreta2014multivariate}.
\begin{definition}\label{definition:CD}
	The Christoffel--Darboux kernel is
	\begin{align*}
	K_n(\x,\y):=\sum_{m=0}^n \big(P_{[m]}(\x)\big)^\top (H_{[m]})^{-1}P_{[m]}(\y)
	\end{align*}
	In terms of the Christoffel--Darboux kernel and a linear functional $u\in\mathcal O_M'$ we define the operator acting on $\mathcal O_M$ as follows
	\begin{align*}
	S_n(f)(\x):=\left\langle
	u, f(\y)K_n(\y,\x) 
	\right\rangle.
	\end{align*}
\end{definition}
%now, let us assume that $f(\x)=\sum_{j\geq 0} c_{[j]} P_{[j]}(\x)$ and compute
%\begin{align*}
%S_n(f)(\y):&=\sum_{j\geq 0}\sum_{m=0}^n c_{[j]}\left\langle u,
%P_{[j]}(\x) \big(P_{[m]}(\x)\big)^\dagger \right\rangle (H_{[m]})^{-1})P_{[m]}(\y)\\
%&=\sum_{j\geq 0}\sum_{m=0}^n c_{[j]}\delta_{j,m}H_{[m]} (H_{[m]})^{-1})P_{[m]}(\y)\\
%&=\sum_{m=0}^n c_{[m]}P_{[m]}(\y).
%\end{align*} 
\begin{pro}
	\begin{enumerate}
		\item 	If $P(\x)=\sum_{j\geq 0} c_{[j]} P_{[j]}(\x)\in\C[\x]\subset \mathcal O_M$ is an arbitrary multivariate polynomial of degree $n$, we have
		\begin{align}\label{eq:projection}
		S_n(P)(\x)
		&=\sum_{m=0}^n c_{[m]}P_{[m]}(\x).
		\end{align} 
		\item For any vector $\n\in\C^D$,  the following Christoffel--Darboux formula is fulfilled
		\begin{multline*}
		\big(\n\cdot (\x-\y)\big) K_n(\x,\y)\\=\big(P_{[n+1]}(\x)\big)^\dagger \Big( (\n\cdot\boldsymbol{\Lambda })_{[n],[n+1]}\Big)^\top (H_{[n]})^{-1}P_{[n]}(\y)-\big(P_{[n]}(\x)\big)^\dagger (H_{[n]})^{-1}(\n\cdot\boldsymbol{\Lambda })_{[n],[n+1]}P_{[n+1]}(\y).
		\end{multline*}
	\end{enumerate}
\end{pro}

\section{Geronimus  transformations}\label{S:2}
In this section a Geronimus transformation for MVOPR is discussed, if we  understand the Christoffel transformation as the perturbation by the multiplication by a polynomial, its right inverse, the Geronimus transformation, might be thought as the perturbation obtained by dividing by a polynomial. We also need a discrete part concentrated at the zeroes of the  polynomial denominator, now an algebraic hypersuface.

\subsection{Geronimus transformations in the multivariate scenario}
Given a polynomial $\mathcal Q_2(\x)\in\C[\x]$ we may consider its principal ideal
\begin{align*}
(\mathcal Q_2)\coloneq\big\{\mathcal Q_2(\x)P(\x): P(\x)\in\C[\x]\big\}.
\end{align*}
This ideal is closely related to the algebraic hypersurface in $\C^D$ of its zero set
\begin{align*}
  Z(\mathcal Q_2):=\{\x\in \C^D: P(\x)=0\}.
\end{align*}
The kernel of a linear functional $v\in\big(\R[\x]\big)'$ is defined by
\begin{align*}
\operatorname{Ker}(v)\coloneq\big\{P(\x)\in\C[\x]: \langle v,P(\x)\rangle=0 \big\}.
\end{align*}

We know that $\C[\x]$ acts on $(\C[\x])'$ by left multiplication, but for the transformations we are dealing with we also need the notion of division by polynomials.
\begin{definition}
	Given fastly decreasing generalized function $u\in\mathcal O_c'$  and a polynomial $\mathcal Q_2(\x)\in\C[\x]$, such that $Z(\mathcal Q_2)\cap \operatorname{supp}(u)=\varnothing$,
 the set of  all the linear functionals  $\check u\in\big(\C[\x]\big)'$ such that
	\begin{align}\label{eq:lf_geronimus}
	\mathcal Q_2 \check u&= u,
	\end{align}
	is called 	its   Geronimus transformation. 
\end{definition}

Notice that there is not a unique linear functional $\check u\in(\C[\x])'$ satisfying such a requirement. Indeed,  suppose that  a solution  is found and denote it by  $\frac{u}{\mathcal Q_2}$, then all possible perturbations $\check u$ verifying \eqref{eq:lf_geronimus} will have the form
\begin{align}\label{eq:Geronimus}
\check u=\frac{u}{\mathcal Q_2}+v,
\end{align}
where the linear functional $v\in(\C[\x])'$ is such that $(\mathcal Q_2)\subseteq \operatorname{Ker}(v)$; i.e.,
\begin{align*}
\mathcal Q_2 v=0.
\end{align*}
For example, given a positive Borel measure $\d\mu(\x)$ and the associated linear functional
\begin{align*}
\langle u, P(\x)\rangle=\int P(\x)\d\mu(\x),
\end{align*}
we can choose $\frac{u}{\mathcal Q_2}\in(\C[\x])'$ as the following linear functional
\begin{align*}
\Big\langle\frac{u}{\mathcal Q_2},P(\x)\Big\rangle=\int P(\x)\frac{\d\mu(\x)}{\mathcal Q_2(\x)},
\end{align*}
which makes sense if $Z(\mathcal Q_2)\cap \operatorname{supp}(\d\mu)=\varnothing$.
Any multivariate polynomial has a unique, up to constants, factorization in terms of prime polynomials
\begin{align*}
\mathcal Q_2 (\x)= \prod_{i=1}^{N}(\mathcal Q_{2,i}(\x))^{d_i},
\end{align*}
where $\mathcal Q_{2,i}$ are prime polynomials for $i\in\{1,\dots D\}$ and the multiplicities $\{d_1,\dots,d_N\}$ are positive integers such hat $m_2=\deg \mathcal Q_2=d_1\deg\mathcal Q_{2,1}+\dots+d_2\deg\mathcal Q_{2,N}$.
Let us consider for each prime factor $\mathcal Q_{2,i}$, $i\in\{1,\dots,N\}$  a set of measures $\big\{\d\xi_{i,\q}\big\}_{\substack{\q\in\Z_+^D\\
|\q|<d_i}}$ with $\operatorname{supp}\big(\d\xi_{i,\q} \big)\subseteq Z(\mathcal Q_{2,i})$. Then, a  linear functional $v$  of the form
\begin{align}\label{eq:v_general}
\langle v, P(\x)\rangle=\sum_{i=1}^N\sum_{\substack{\q\in\Z_+^D\\|\q|<d_i}}\int_{Z(\mathcal Q_{2,i})}\frac{\partial^\q P}{\partial \x^\q}(\x) \d\xi_{i,\q},
\end{align}
is such that $(\mathcal Q_2)\subseteq \operatorname{Ker}(v)$.

In the $D=1$ context, where up to constants
 $\mathcal Q_2(x)=(x-q_1)^{d_1}\cdots(x-q_{N})^{d_N}$, with different roots $\{q_,\dots,q_N\}$, and multiplicitities $\{d_1,\dots,d_N\}$ such that
$d_1+\dots+d_N=m_2$, the most general form of $v$ is, in terms of the Dirac linear functional $\delta$ and its derivatives, given by
\begin{align}\label{eq:nu}
v&=\sum_{i=1}^{N}\sum_{j=0}^{d_i-1}\zeta_i^{(j)}\delta^{(j)}(x-q_i), & \zeta_i^{(j)}&\in\R.
\end{align}
Observe that for multiplicities greater than 1 we have linear functionals of higher order and therefore not linked to measures, which are linear functionals of order zero.

From hereon we assume that both linear functionals $u$ and $\check u$ give rise to well defined families of MVOPR, equivalently that all their moment matrix block minors are nonzero $\det G^{[k]}\neq 0$, $\det \check G^{[k]}\neq 0$, $\forall k\in\{1,2,\dots\}$.
\begin{pro}\label{pro:fac_TG}
The moment matrices $\check G$  and $G$, of the perturbed linear functional $\check u$ and unperturbed  linear functional $u$, respectively, satisfy
	\begin{align}\label{eq:G-Geronimus}
\mathcal Q_2(\boldsymbol \Lambda)	\check G= \check G \mathcal Q_2(\boldsymbol \Lambda^\top)=G.
	\end{align}
\end{pro}
\begin{proof}
	It is a direct consequence of the spectral property $\mathcal Q_2(\boldsymbol{\Lambda})\chi(\x)=\mathcal Q_2(\x)\chi(\x)$, that is deduced from \eqref{eigen}. Indeed,
	\begin{align*}
	\mathcal Q_2(\boldsymbol \Lambda)\big\langle \check u,\chi(\x)\big(\chi(\x)\big)^\top\big\rangle&=
	\big\langle \check u, \mathcal Q_2(\x)\chi(\x)\big(\chi(\x)\big)^\top\big\rangle\\
	&=	\big\langle \mathcal Q_2\check u, \chi(\x)\big(\chi(\x)\big)^\top\big\rangle \\
	&=\langle  u,\chi(\x)\big(\chi(\x)\big)^\top\rangle&
	\text{  use \eqref{eq:lf_geronimus}.}
		\end{align*}
\end{proof}
Let us notice that for a given semi-infinite matrix $G$ there is not a a unique $\check G$ satisfying \eqref{eq:G-Geronimus}. In fact, observe that given any generalized function $v$ of the form \eqref{eq:v_general} and any semi-infinite block vector $\zeta=(\zeta_0,\zeta_1,\dots)^\top$, $\zeta_i\in\R$, we have
\begin{align*}
\mathcal Q_2(\boldsymbol \Lambda)\big\langle v, \chi(\x)\zeta^\top\big\rangle=0.
\end{align*}
and if $\check G$ satisfies \eqref{eq:G-Geronimus} so does $\check G+ \big\langle v, \chi(\x)\zeta^\top\big\rangle$.
\subsection{Resolvents and connection formul{\ae}}
\begin{definition}
The resolvent  matrices are
	\begin{align*}
	\omega_1\coloneq&\check SS^{-1}, & 	(\omega_2)^\top&\coloneq S\mathcal Q_2(\boldsymbol\Lambda)(\check S)^{-1},
	\end{align*}
	given in terms of the lower unitriangular block semi-infinite matrices $S$ and $\check S$ of the Cholesky factorizations of the moment matrices $G=S^{-1}H(S^{-1})^\top$ and  $\check G=(\check S)^{-1}(\check H)(\check S^{-1})^\top$, respectively.
\end{definition}

\begin{pro}
	We have that
	\begin{align}\label{eq:M-omega}
	\check H\omega_2=\omega_1H.
	\end{align}
\end{pro}
\begin{proof}
	It follows from the Cholesky factorization of $G$ and $\check G$ and from \eqref{pro:fac_TG}.
\end{proof}
We now decompose the perturbing multidimensional polynomial $\mathcal Q_2$ in its homogeneous parts $\mathcal Q_2(\x)=\sum_{n=0}^{m_2}\mathcal Q_2^{(n)}(\x)$ where $\mathcal Q_2^{(n)}(\x)$ are homogeneous polynomials of degree $n$, i.e.,
$\mathcal Q_2^{(n)}(s\x)=s^n\mathcal Q_2^{(n)}(\x)$, for all $s\in\R$.
\begin{pro}\label{pro:superdiagonals}
	In terms of block subdiagonals the adjoint resolvent $\omega_1$ can be expressed as follows
\begin{align*}
%  (\omega_2)^\top=&\underbracket{\mathcal Q_2^{(m_2)}(\boldsymbol{\Lambda})}_{\text{$m_2$-th superdiagonal}}\\&+
%\underbracket{\beta\mathcal Q_2^{(m_2-1)}(\boldsymbol{\Lambda})-
%\mathcal Q_2^{(m_2-1)}(\boldsymbol{\Lambda})\hat\beta}_{\text{$(m_2-1)$-th superdiagonal}}\\&\shortvdotswithin{+}&+
%  \underbracket{H\hat H^{-1}}_{\text{diagonal}}
  \omega_1=&\underbracket{\check H\mathcal Q_2^{(m_2)}(\boldsymbol{\Lambda}^\top)H^{-1}}_{\text{$m_2$-th subdiagonal}}\\&+
  \underbracket{\check H \Big(\mathcal Q_2^{(m_2-1)}(\boldsymbol{\Lambda}^\top)+\mathcal Q_2^{(m_2)}(\boldsymbol{\Lambda}^\top)\beta^\top-\hat \beta^\top \mathcal Q_2^{(m_2)}\big(\boldsymbol{\Lambda}^\top\big)\Big)H^{-1}
  }_{\text{$(m_2-1)$-th subdiagonal}}\\&\shortvdotswithin{+}&+
  \underbracket{\,\I\,}_{\text{diagonal}}
\end{align*}
\end{pro}
\begin{proof}
	%where $\beta$ is the first subdiagonal of $S$. Moreover, $M=H\omega^\top (TH)^{-1}$.
The  resolvent $\omega_1$ is a block lower unitriangular semi-infinite matrix and the adjoint resolvent $(\omega_2)^\top$ has all its superdiagonals but for the first $m$ equal to zero. The result follows from \eqref{eq:M-omega}.
\end{proof}
Incidentally, and not essential for further developments in this paper, we have the following two Propositions regarding Jacobi matrices
\begin{pro}\label{pro:jacobi-LU}
The following $UL$ and $LU$ factorizations
	\begin{align*}
\mathcal Q_2(\boldsymbol J)=&(\omega_2)^\top\omega_1,  &
\mathcal Q_2(\boldsymbol{\check{ J}})=&\omega_1 (\omega_2)^\top ,
\end{align*}
hold.
\end{pro}
\begin{proof}
Both  follow from Proposition \ref{pro:fac_TG} and the Cholesky factorization which imply
\begin{align*}
\mathcal Q_2(\boldsymbol \Lambda)(\check S)^{-1}\check H(\check S^{-1})^\top = S^{-1} H (S^{-1})^\top,
\end{align*}
and a proper cleaning does the job.
\end{proof}
From the first equation in the previous Proposition we get
\begin{pro}
	The block truncations $(\mathcal Q_2 (\boldsymbol{ \check J}))^{[k]}$ admit a $LU$ factorization
	\begin{align*}
(\mathcal Q_2 (\boldsymbol{ \check J}))^{[k]}=\omega_1^{[k]}(\omega_2^{[k]})^\top
	\end{align*}
	in terms of the corresponding truncations of resolvents.
\end{pro}
\begin{pro}\label{pro:regularity-truncation-jacobi}
	We have
	\begin{align*}
	\det ((\mathcal Q_2 (\boldsymbol{ \check J}))^{[k]})=\prod_{l=0}^{k-1}\frac{\det H_{[l]}}{\det \check H_{[l]}}
	\end{align*}
and therefore $(\mathcal Q_2 (\boldsymbol{ \check J}))^{[k]}$ is a regular matrix.
\end{pro}
\begin{proof}
To prove this result just use Propositions \ref{pro:jacobi-LU} and \ref{pro:superdiagonals} and the assumption that
the minors of the moment matrix and the perturbed moment matrix are not zero.
\end{proof}

The next connection relations will be relevant for the finding of the Gerominus formul\ae{}

\begin{pro}[Connection formul\ae]
	The followings relations are fulfilled
	\begin{align}
\notag	(\omega_2)^\top\check P(\x)&=\mathcal Q_2(\x)P(\x),\\
\label{omega-P}	\omega_1 P(\x)&=\check P(\x).
	\end{align}
\end{pro}
\subsection{The multivariate Geronimus formula}

To  extend to multidimensions the Geronimus determinantal expressions for the Geronimus transformations  \cite{Geronimus1940polynomials} we need a new object. In the 1D case it is enough to use the Cauchy transforms of the OPRL, so closely related to the Stieljes functions. However, in this multivariate scenario we have not been able to use the corresponding multivariate Cauchy transforms, see \cite{ariznabarreta2014multivariate}, precisely because of complications motivated by the multidimensionality. Instead, we have been able to use an alternative path by introducing a semi-infinite matrix $R$ that in the 1D case, using a partial fraction expansion, can be expressed  in terms of the mentioned Cauchy transforms and Geronimus type combinations. This new element is essential in the finding of a  new  multivariate Geronimus quasi-determinantal formula.
\begin{definition}\label{def:r}
We introduce the semi-infinite block matrices
\begin{align*}
R&\coloneq \big\langle\check u, P(\x)\big(\chi(\x)\big)^\top\big\rangle.
\end{align*}
\end{definition}

\begin{pro}
The formula
	\begin{align*}
	R&=\rho+\theta, &
	\rho&\coloneq
	\bigg\langle u,\frac{P(\x)\big(\chi(\x)\big)^\top}{\mathcal Q_2(\x)} \bigg\rangle, &
	\theta&\coloneq \big\langle v, P(\x)\big(\chi(\x)\big)^\top\big\rangle,
	\end{align*}
	holds.
\end{pro}
\begin{proof}
	Just write $\check u=\frac{u}{\mathcal Q_2}+v$, with $(\mathcal Q_2)\subseteq\operatorname{Ker} v$.
\end{proof}
\begin{pro}
	If the linear functional $u$ is of order zero with an  associated Borel measure $\d\mu(\x)$ we can write
\begin{align*}
\rho=\int P(\x)(\chi(\x))^\top\frac{\d\mu(\x)}{\mathcal Q_2(\x)},
\end{align*}
and if  $\mathcal Q_2(\x)=(\mathcal Q_{2,1}(\x))^{d_1}\cdots (\mathcal Q_{2,N}(\x))^{d_N}$ is a prime factorization, and $v$ is taken as in \eqref{eq:v_general} we can write
\begin{align*}
\theta=\sum_{i=1}^N\sum_{\substack{\q\in\Z_+^D\\|\q|<d_i}}\int_{Z(\mathcal Q_{2,i})}\frac{\partial^\q \big(P(\x)(\chi(\x))^\top\big)}{\partial \x^\q} \d\xi_{i,\q}(\x).
\end{align*}
\end{pro}

\begin{pro}\label{pro:r-alpha}
	The following relations
	\begin{align*}
	(\omega_1R)_{[k],[l]}&=0, & l&\in\{0,1,\dots,k-1\},\\
	(\omega_1R)_{[k],[k]}&=\check H_{[k]},
\end{align*}
	hold true.
\end{pro}
\begin{proof}
A direct computation leads to  the result. Indeed,
	\begin{align*}
	\omega_1 R=&\big \langle \check u,\omega_1 P(\x)\big(\chi(\x)\big)^\top\big\rangle\\
	=&\big\langle\check u, \check P(\x) \big(\chi(\x)\big)^\top\big\rangle &\text{ recall \eqref{omega-P}}
	\end{align*}
	and the orthogonality equations \eqref{orth} and \eqref{H} give the desired conclusion.
\end{proof}

\begin{pro}\label{pro:M}
\begin{enumerate}
	\item The truncations $R^{[k]}$ are nonsingular for all $k\in\Z_+$.
\item	The adjoint resolvent entries satisfy
	\begin{align}\label{eq:M-r0}
		 \big((\omega_1)_{[k],[0]},\dots,(\omega_1)_{[k],[k-1]}\big) =-(R_{[k],[0]},\dots,R_{[k],[k-1]})\big(R^{[k]}\big)^{-1}.
	\end{align}
		\item We can express each entry of the adjoint resolvent as
			\begin{align}\label{eq:Mk-r0}
		 (\omega_1)_{[k],[l]}&=
-(R_{[k],[0]},\dots,R_{[k],[k-1]})\big(R^{[k]}\big)^{-1}
\PARENS{		
	\begin{matrix}
	0_{[0],[l]}\\
	\vdots\\
	0_{[l-1],[l]}\\
	\I_{[l]}\\
	0_{[l+1],[l]}\\
	\vdots \\
	0_{[k],[l]}
		\end{matrix}}, & l&\in\{0,1,\dots,k-1\},.
 			\end{align}
\end{enumerate}
\end{pro}
\begin{proof}
\begin{enumerate}
	\item We can write
	\begin{align}
R^{[k+1]}	=S^{[k+1]} \check G^{[k+1]}
\label{eq:RSG}	\end{align}
	so that
	\begin{align*}
\det R^{[k+1]}=\prod_{l=0}^{k}\det \check H_{[l]}\neq 0.
	\end{align*}
	
	\item 	 From Propositions \ref{pro:superdiagonals}  and \ref{pro:r-alpha} we deduce
	 \begin{align*}
	 (\omega_1)_{[k],[0]}R_{[0],[l]}+\dots+ (\omega_1)_{[k],[k-1]}R_{[k-1],[l]}&=-R_{[k],[l]}, & l&\in\{0,1,\dots,k-1\}.
	 \end{align*}
Therefore,	  we get
	 \begin{align*}
	 \big((\omega_1)_{[k],[0]},\dots,(\omega_1)_{[k],[k-1]}\big) R^{[k]}=-(R_{[k],[0]},\dots,R_{[k],[k-1]}),
	 \end{align*}
	 from where \eqref{eq:M-r0} follows.
\end{enumerate}
\end{proof}

\begin{theorem}
We can express the new MVOPR, $\check P_{[k]}(\x)$,  and the quasi-tau matrices $\check H_{[k] }$ in terms of the non-perturbed ones as follows
		\begin{align}\label{eq:TPk}
\check P_{[k]}(\x)&=\Theta_*\PARENS{
\begin{matrix}
			R_{[0],[0]} & \dots &R_{[k],[k-1]} &P_{[0]}(\x)\\
		\vdots &      & \vdots &\vdots\\
		R_{[k],[0]} & \dots &R_{[k],[k-1]} &P_{[k]}(\x)		
\end{matrix}
		},	
\\
\label{eq:TH}
\check H_{[k]}&=\Theta_*(R^{[k+1]}).
		\end{align}
\end{theorem}
\begin{proof}
From \eqref{omega-P} we deduce
	\begin{align}\label{eq:TP-M}
\check P_{[k]}(\x)=
	(\omega_1)_{[k],[0]}P_{[0]}(\x)+\dots+(\omega_1)_{[k],[k-1]} P_{[k-1]}(\x)+P_{[k]}(\x)
	\end{align}
	and Proposition \ref{pro:M} implies
	\begin{align*}
	\check P_{[k]}(\x)=P_{[k]}(\x)-(R_{[k],[0]},\dots,R_{[k],[k-1]})\big(R^{[k]}\big)^{-1}
\PARENS{\begin{matrix}
P_{[0]}(\x)\\ \vdots \\P_{[k-1]}(\x)
	\end{matrix}}
\end{align*}
and, consequently, \eqref{eq:TPk} follows.

From Proposition \ref{pro:r-alpha} we get
\begin{align*}
	 (\omega_1)_{[k],[0]}R_{[0],[k]}+\dots+ (\omega_1)_{[k],[k-1]}R_{[k-1],[k]}+R_{[k],[k]}&=\check H_{[k]},
\end{align*}
now recall \eqref{eq:M-r0} to deduce
\begin{align*}
\check H_{[k]}
=R_{[k],[k]}-(R_{[k],[0]},\dots,R_{[k],[k-1]})\big(R^{[k]}\big)^{-1}\PARENS{
	\begin{matrix}
	R_{[0],[k]}\\\vdots\\R_{[k-1],[k]}
	\end{matrix}},
\end{align*}
so that \eqref{eq:TH} is proven. Let us mention that it also follows  from \eqref{eq:RSG}.
\end{proof}

The previous relations involve a growing number of terms as $k$ increases. However, for  $k\geq m_2$ this changes.
\begin{definition}
	\begin{enumerate}
		\item 	If $k> m_2$, take an ordered set of multi-indices
\begin{align*}
		\mathcal M_k\coloneq\big\{\boldsymbol{\beta}_i\in(\Z_+)^D: |\boldsymbol{\beta_i}|< k\big\}_{i=1}^{r_{k,m_2}}		
\end{align*}
		 with cardinal given by
		 \begin{align*}
		 r_{k,m_2}\coloneq |\mathcal M_k|=		 	N_{k-1}-N_{k-m_2-1}=|[k-m_2]|+\dots+|[k-1]|.
		 \end{align*}
		 \item Associated with this  set consider the truncations
		 \begin{align*}
		 R^{[\mathcal M_k]}&\coloneq
		\PARENS{	\begin{matrix}
			R_{[k-m_2],\boldsymbol{\beta}_1} & \dots & R_{[k-m_2],\boldsymbol{\beta}_{r_{k,m_2}}}\\
			\vdots &            &\vdots\\
			R_{[k-1],\boldsymbol{\beta}_1} & \dots & R_{[k-1],\boldsymbol{\beta}_{r_{k,m_2}}}
						\end{matrix}},\\
				R_{\mathcal M_k}&\coloneq (R_{[k],\boldsymbol{\beta}_1},\dots,R_{[k],\boldsymbol{\beta}_{r_{k,m_2}}})	
		 \end{align*}
		 \item Then, the set $\mathcal M_k$ is said to be  poised if the corresponding truncation is not singular
		\begin{align*}
	\begin{vmatrix}
	R_{[k-m_2],\boldsymbol{\beta}_1} & \dots & R_{[k-m_2],\boldsymbol{\beta}_{r_{k,m_2}}}\\
	\vdots &            &\vdots\\
	R_{[k-1],\boldsymbol{\beta}_1} & \dots & R_{[k-1],\boldsymbol{\beta}_{r_{k,m_2}}}
	\end{vmatrix}\neq 0.
		\end{align*}
	\end{enumerate}
\end{definition}
\begin{pro}
	Poised sets do exist.
\end{pro}
\begin{proof}
	We need to ensure that among all subsets $\mathcal M_k$ of multi-indices of length less than $k$ there is at least one such that $\det R^{[\mathcal M_k]}\neq 0$. We proceed by contradiction. If we assume that there is no such set the matrix
	\begin{align*}
	\PARENS{
		\begin{matrix}
		R_{[k-m_2],[0]} &\dots & R_{[k-m_2],[k-1]}\\
		\vdots & & \vdots \\
		R_{[k-1],[0]} &\dots & R_{[k-1],[k-1]}
		\end{matrix}
	}
	\end{align*}
	is not full rank and, consequently, $R^{[k]}$ will be singular, which is  in contradiction with our assumptions.
\end{proof}

\begin{pro}\label{pro:M2_v1}
	For $k\geq m_2$  and a poised set of multi-indices  $\mathcal M_k$, we have
	\begin{align*}
	\big((\omega_1)_{[k],[k-m_2]},\dots, (\omega_1)_{[k],[k-1]}\big)=-R_{\mathcal M_k}\big(R^{[\mathcal M_k]}\big)^{-1}.
	\end{align*}
\end{pro}

\begin{proof}
	Observe  that Propositions \ref{pro:superdiagonals}  and \ref{pro:r-alpha} imply
\begin{align*}
	(\omega_1)_{[k],[k-m_2]}R_{[k-m_2],[l]}+\dots+ (\omega_1)_{[k],[k-1]}R_{[k-1],[l]}=-R_{[k],[l]},
		\end{align*}
		for $l\in\{0,1,\dots,k-1\}$.
Hence, we deduce
	\begin{align*}
	\big((\omega_1)_{[k],[k-m_2+1]},\dots, (\omega_1)_{[k],[k]}\big)
R^{[\mathcal M_k]}
	&=-R_{\mathcal M_k},
	\end{align*}
	from where  the result follows.
\end{proof}

\begin{theorem}[Multivariate Gerominus formul{\ae}]
	For $k\geq m_2$ and a given a poised set of multi-indices $\mathcal M_k$ we can write
	\begin{align}\label{eq:TP>}
	\check P_{[k]}(\x)&=	\Theta_*\PARENS{
\begin{matrix}
R_{[k-m_2],\boldsymbol{\beta}_1} & \dots & R_{[k-m_2],\boldsymbol{\beta}_{r_{k,m_2}}}&	P_{[k-m_2]}(\x)\\
\vdots &            &\vdots&\vdots\\
R_{[k],\boldsymbol{\beta}_1} & \dots & R_{[k],\boldsymbol{\beta}_{r_{k,m_2}}} &P_{[k]}(\x)
		\end{matrix}
	}.
	\end{align}
In this case, for the quasi-tau matrices we have  the following two expressions
	\begin{align}\label{eq:THjm}
	\check H_{[k]}\Big(\big(\mathcal Q_2(\boldsymbol \Lambda)\big)_{[k-m_2],[k]}\Big)^\top&=\Theta_*\PARENS{
\begin{matrix}
R_{[k-m_2],\boldsymbol{\beta}_1} & \dots & R_{[k-m_2],\boldsymbol{\beta}_{r_{k,m_2}}}&H_{[k-m_2]}\\
R_{[k-m_2+1],\boldsymbol{\beta}_1} & \dots & R_{[k-m_2+1],\boldsymbol{\beta}_{r_{k,m_2}}}&0_{[k-m_2+1],[k-m_2]}\\
\vdots &            &\vdots&\vdots\\
R_{[k],\boldsymbol{\beta}_1} & \dots & R_{[k],\boldsymbol{\beta}_{r_{k,m_2}}} &0_{[k],[k-m_2]}
\end{matrix}
	},\\
	\label{eq:THjm2}
	\check H_{[k]}&=\Theta_*\PARENS{
		\begin{matrix}
		R_{[k-m_2],\boldsymbol{\beta}_1} & \dots & R_{[k-m_2],\boldsymbol{\beta}_{r_{k,m_2}}}&R_{[k-m_2],[k]}\\
		R_{[k-m_2+1],\boldsymbol{\beta}_1} & \dots & R_{[k-m_2+1],\boldsymbol{\beta}_{r_{k,m_2}}}&R_{[k-m_2+1],[k]}\\
		\vdots &            &\vdots&\vdots\\
		R_{[k],\boldsymbol{\beta}_1} & \dots & R_{[k],\boldsymbol{\beta}_{r_{k,m_2}}} &R_{[k],[k]}
		\end{matrix}
	},
	\end{align}			
\end{theorem}
\begin{proof}
	When $k\geq m_2$ we can use \eqref{omega-P}
	\begin{align*}
	\check P_{[k]}(\x)&=	(\omega_1)_{[k],[k-m_2]} P_{[k-m_2]}(\x)+\dots+ (\omega_1)_{[k],[k-1]}P_{[k-1]}(\x)+P_{[k]}(\x),
	\end{align*}
	and   Proposition \ref{pro:M2_v1} leads to \eqref{eq:TP>}.
From Proposition \ref{pro:superdiagonals} we get
\begin{align*}
(\omega_1)_{[k],[k-m_2]}=\check H_{[k]}\Big(\big(\mathcal Q_2(\boldsymbol \Lambda)\big)_{[k-m_2],[k]}\Big)^\top \big(H_{[k-m_2]}\big)^{-1},
\end{align*}
while  Proposition \ref{pro:M2_v1} tells us that
\begin{align*}
(\omega_1)_{[k],[k-m_2]}=-R_{\mathcal M_k}\big(R^{[\mathcal M_k]}\big)^{-1}\PARENS{
	\begin{matrix}
\I_{[k-m_2]}\\0_{[k-m_2+1],[k-m_2]} \\\vdots \\0_{[k],[k-m_2]}
\end{matrix}
},
\end{align*}
and, consequently, \eqref{eq:THjm} is proven. Then, to prove \eqref{eq:THjm2} just recall Proposition \ref{pro:r-alpha} and write
\begin{align*}
\check H_{[k]}=	\big((\omega_1)_{[k],[k-m_2]},\dots, (\omega_1)_{[k],[k-1]}\big)\PARENS{
	\begin{matrix}
	R_{[k-m_2],[k]}\\\vdots\\R_{[k-1],[k]}
	\end{matrix}}+R_{[k],[k]},
\end{align*}
and use Proposition \ref{pro:M2_v1} to conclude
\begin{align*}
\check H_{[k]}=R_{[k],[k]}	-R_{\mathcal M_k}\big(R^{[\mathcal M_k]}\big)^{-1}\PARENS{
	\begin{matrix}
	R_{[k-m_2],[k]}\\\vdots\\R_{[k-1],[k]}
	\end{matrix}}.
\end{align*}
\end{proof}
\subsection{Recovering the 1D Geronimus formula}
Let us assume that $D=1$, then $|[k]|=1$ and $N_{k-1}=k$ and for $k\geq  m_2$ we have $r_{k,m_2}=m_2$, so we can   choose the indices as $\{0,1,\dots,m_2-1\}$ (there are other possibilities but let us suppose that it is poised)  as  they all are less than $k$.
Let us assume that $Q_2(x)=(x-q_1)\cdots (x-q_{m_2})$, has $m_2$ simple zeroes $\{q_1,\dots,q_{m_2}\}$, and let us consider the Cauchy transforms $C_k(x)$ of the orthogonal polynomials $P_k(x)$ of the original measure $\d\mu(x)$ given by
\begin{align*}
C_k(x)\coloneq\int \frac{P_k(y)}{y-x}\d\mu(y).
\end{align*}
The point is that the two set of numbers $\{C_k(q_1),\dots, C_k(q_{m_2})\}$ and $\{\rho_{k,0},\rho_{k,1},\dots,\rho_{k,m_2-1}\}$
are linked by the Vandermonde matrix
\begin{align*}
\mathcal V=\PARENS{
	\begin{matrix}
	1&\dots &1\\
	q_1& \dots & q_{m_2}\\
	\vdots & & \vdots\\
	q_1^{m_2-1}& \dots & q_{m_2}^{m_2-1}
	\end{matrix}},
\end{align*}
and the diagonal matrix
\begin{align*}
\mathcal{D}\coloneq \diag\Big(\prod_{\substack{i\in\{1,\dots,m_2\}\\i\neq 1}}(q_1-q_i), \dots, \prod_{\substack{i\in\{1,\dots,m_2\}\\ i\neq m_2}}(q_{m_2}-q_i)\Big),
\end{align*}
by the formula
\begin{align}\label{eq:R_Cauchy}
\big(
\rho_{k,0},
\dots,
\rho_{k,m_2-1}
\big)=
\big(C_k(q_1),\dots,
C_k(q_{m_2})\big)\mathcal D^{-1}\mathcal V^\top.
\end{align}
This relation can be obtained from the identity
\begin{align*}
\frac{(x-q_1)\cdots\widehat{(x-q_i)}\cdots (x-q_{m_2})}{(x-q_1)\cdots(x-q_{m_2})}=\frac{1}{x-q_i},
\end{align*}
where by  $\widehat{(x-q_i)}$ we mean that this factor has been deleted from the product, by expanding the numerator ---according to Vieta's formul\ae{}--- in terms of elementary symmetric polynomials of the roots, $e_j(q_1,\dots,q_{m_2})$, $j\in\{0,1,\dots,m_2\}$.
Moreover, we have the following formul{\ae}
\begin{align*}
\PARENS{
	\begin{matrix}
	\rho_{k-m_2,0}&\dots & \rho_{k-m_2,m_2-1}\\
	\vdots & &\vdots\\
	\rho_{k-1,0}&\dots & \rho_{k-1,m_2-1}
	\end{matrix}}
&=
\PARENS{
	\begin{matrix}
	C_{k-m_2}(q_1)&\dots & C_{k-m_2}(q_{m_2})\\
	\vdots & &\vdots\\
	C_{k-1}(q_1)&\dots &C_{k-1}(q_{m_2})
	\end{matrix}}\mathcal D^{-1}
	\mathcal V^\top.
\end{align*}

Regarding the $\theta_{k,n}$ terms we must recall that a general form of $\d\nu$ in the 1D scenario is given in \eqref{eq:nu},
from where one concludes that
\begin{align*}
\PARENS{
	\begin{matrix}
	\theta_{k-m_2,0}&\dots & \theta_{k-m_2,m_2-1}\\
	\vdots & &\vdots\\
	\theta_{k-1,0}&\dots & \theta_{k-1,m_2-1}	
	\end{matrix}}
&=\PARENS{
	\begin{matrix}
	P_{k-m_2}(q_1)&\dots & P_{k-m_2}(q_{m_2})\\
	\vdots & &\vdots\\
	P_{k-1}(q_1)&\dots &P_{k-1}(q_{m_2})
	\end{matrix}}\zeta \mathcal V^\top
\end{align*}
where
\begin{align*}
\zeta=\diag(\zeta_1,\dots,\zeta_{m_2}).
\end{align*}
Hence, if
\begin{align*}
\xi_j&\coloneq\zeta_j\prod_{\substack{i\in\{1,\dots,m_2\}\\ i\neq j}}(q_j-q_i), &
\phi_l(x,\xi)&\coloneq C_l(x)+\xi P_l(x),
\end{align*}
we  get
\begin{align*}
R^{[\mathcal M_k]}
	&=
	\PARENS{
		\begin{matrix}
			\phi_{k-m_2}(q_1,\xi_1)&\dots & \phi_{k-m_2}(q_{m_2},\xi_{m_2})\\
			\vdots & &\vdots\\
			\phi_{k-1}(q_1,\xi_1)&\dots &\phi_{k-1}(q_{m_2},\xi_{m_2})
		\end{matrix}}\mathcal D^{-1}
		\mathcal V^\top,\\
		 R_{\mathcal M_k}&=\big(\phi_{k}(q_1,\xi_1),\dots ,\phi_{k}(q_{m_2},\xi_{m_2})\big)\mathcal D^{-1}
		 \mathcal V^\top.
		\end{align*}
Therefore,
\begin{align*}
R_{\mathcal M_k}\big(R^{[\mathcal R_k]}\big)^{-1}=\big(\phi_{k}(q_1,\xi_1),\dots ,\phi_{k}(q_{m_2},\xi_{m_2})\big)	\PARENS{
		\begin{matrix}
			\phi_{k-m_2}(q_1,\xi_1)&\dots & \phi_{k-m_2}(q_{m_2},\xi_{m_2})\\
			\vdots & &\vdots\\
			\phi_{k-1}(q_1,\xi_1)&\dots &\phi_{k-1}(q_{m_2},\xi_{m_2})
		\end{matrix}}^{-1}.
\end{align*}

We finally get for, $k\geq m_2$, the perturbed polynomials the Geronimus formula \cite{Geronimus1940polynomials}
\begin{align*}
\check P_k(x)& =\Theta_* \PARENS{
		\begin{matrix}
			\phi_{k-m_2}(q_1,\xi_1)&\dots & \phi_{k-m_2}(q_{m_2},\xi_{m_2}) &P_{k-m_2}(x)\\
			\vdots & &\vdots\\
			\phi_{k}(q_1,\xi_1)&\dots &\phi_k(q_{m_2},\xi_{m_2})&P_k(x)
		\end{matrix}}\\&=\frac{	\begin{vmatrix}
					\phi_{k-m_2}(q_1,\xi_1)&\dots & \phi_{k-m_2}(q_{m_2},\xi_{m_2}) &P_{k-m_2}(x)\\
					\vdots & &\vdots\\
					\phi_{k}(q_1,\xi_1)&\dots &\phi_k(q_{m_2},\xi_{m_2})&P_k(x)
				\end{vmatrix}}{	\begin{vmatrix}
							\phi_{k-m_2}(q_1,\xi_1)&\dots & \phi_{k-m_2}(q_{m_2},\xi_{m_2}) \\
							\vdots & &\vdots\\
							\phi_{k-1}(q_1,\xi_1)&\dots &\phi_{k-1}(q_{m_2},\xi_{m_2})
						\end{vmatrix}},
						\end{align*}
and the perturbed squared norms
\begin{align*}
\check H_k &=\Theta_*\PARENS{
		\begin{matrix}
			\phi_{k-m_2}(q_1,\xi_1)&\dots & \phi_{k-m_2}(q_{m_2},\xi_{m_2}) &H_{k-m_2}\\
			\phi_{k-m_2+1}(q_1,\xi_1)&\dots & \phi_{k-m_2+1}(q_{m_2},\xi_{m_2}) &0\\
			\vdots & &\vdots\\
			\phi_{k}(q_1,\xi_1)&\dots &\phi_k(q_{m_2},\xi_{m_2})&0
		\end{matrix}}\\&=
		(-1)^{m_2+1}\frac{\begin{vmatrix}
											\phi_{k-m_2+1}(q_1,\xi_1)&\dots & \phi_{k-m_2+1}(q_{m_2},\xi_{m_2}) \\
											\vdots & &\vdots\\
											\phi_{k}(q_1,\xi_1)&\dots &\phi_{k}(q_{m_2},\xi_{m_2})
										\end{vmatrix}}{\begin{vmatrix}
									\phi_{k-m_2}(q_1,\xi_1)&\dots & \phi_{k-m_2}(q_{m_2},\xi_{m_2}) \\
									\vdots & &\vdots\\
									\phi_{k-1}(q_1,\xi_1)&\dots &\phi_{k-1}(q_{m_2},\xi_{m_2})
								\end{vmatrix}}H_{k-m_2}.
\end{align*}

\section{Linear spectral  transformations}\label{S:3}

Once we have discussed the multivariate Geronimus  transformation we are ready to consider the more general  linear spectral transform, that might be thought as the multiplication by a rational function, plus an extra contribution living in the zeroes of the polynomial in the denominator. Uvarov perturbations are treated in the Appendix.

%\subsection{The general multivariate linear spectral transformation}
\begin{definition}\label{def:linear_spectral}
For a given generalized function $u\in\mathcal O_c'$, let us consider  two  coprime polynomials  $\mathcal Q_1(\x),\mathcal Q_2(\x)\in\C[\x]$,  i.e., with no common prime factors,  degrees
$\deg \mathcal Q_1=m_1$ and $\deg\mathcal Q_2=m_2,$ and such that $Z(\mathcal Q_2)\cap\operatorname{supp}(u)=\varnothing$. Then,  the set of linear functionals $\hat u$ such that
\begin{align}\label{eq:lf_spectral_linear}
\mathcal Q_2\hat u &=\mathcal Q_1u.
\end{align}
is called a linear spectral  transformation.
\end{definition}

Again, there is not a unique $\hat u$ satisfying this condition. In fact, assume we have found such  linear functional that we denote as $\frac{\mathcal Q_1}{\mathcal Q_2}u$, then all possible perturbations $\hat u$ verifying \eqref{eq:lf_geronimus} will have the form
\begin{align*}%\label{eq:linear}
\hat u=\frac{\mathcal Q_1}{\mathcal Q_2}u+v
\end{align*}
where, as for the Geronimus transformation, the linear functional $v\in(\C[\x])'$ is such that $(\mathcal Q_2)\subseteq \operatorname{Ker}(v)$; i.e.,
\begin{align*}
\mathcal Q_2 v=0.
\end{align*}
\begin{pro}
	A linear spectral transformation $u\mapsto \hat u$ can be obtained by performing first a Geronimus transformation and then and a Christoffel transformation:
\begin{align*}
u\mapsto \check u \mapsto \hat u,
\end{align*}
where
\begin{align*}
\mathcal Q_2 \check u &= u, & \hat u&= \mathcal Q_1 \check u.
\end{align*}
\end{pro}

For example, for a  given a positive Borel measure $\d\mu(\x)$ with associated zero order linear functional
\begin{align*}
\langle u, P(\x)\rangle=\int P(\x)\d\mu(\x),
\end{align*}
we can choose $\frac{\mathcal Q_1}{\mathcal Q_2}u\in(\C[\x])'$ as the following linear functional
\begin{align*}
\Big\langle\frac{\mathcal Q_1}{\mathcal Q_2}u,P(\x)\Big\rangle=\int P(\x)\frac{\mathcal Q_1(\x)}{\mathcal Q_2(\x)}\d\mu(\x),
\end{align*}
which makes sense if $Z(\mathcal Q_2)\cap \operatorname{supp}(\d\mu)=\varnothing$.

\begin{pro}\label{pro:fac_TG_2}
	If $\hat G$ is the moment matrix of the perturbed linear functional $\hat u$ we have
	\begin{align*}
	\mathcal Q_2(\boldsymbol \Lambda)	\hat G= \hat G \mathcal Q_2(\boldsymbol \Lambda^\top)=
		\mathcal Q_1(\boldsymbol \Lambda)	 G=  G \mathcal Q_1(\boldsymbol \Lambda^\top).
	\end{align*}
\end{pro}
\begin{proof}
It is proven as follows
	\begin{align*}
	\mathcal Q_2(\boldsymbol \Lambda)\big\langle \hat u,\chi(\x)\big(\chi(\x)\big)^\top\big\rangle&=
	\big\langle \hat u, \mathcal Q_2(\x)\chi(\x)\big(\chi(\x)\big)^\top\big\rangle\\
	&=	\big\langle \mathcal Q_2\hat u, \chi(\x)\big(\chi(\x)\big)^\top\big\rangle \\
	&=\big\langle  \mathcal Q_1u,\chi(\x)\big(\chi(\x)\big)^\top\big\rangle &
	\text{  use \eqref{eq:lf_spectral_linear}}\\
	&=
	\big\langle u, \mathcal Q_1(\x)\chi(\x)\big(\chi(\x)\big)^\top\big\rangle \\
		&=\mathcal Q_1(\boldsymbol{\Lambda})\big\langle u, \chi(\x)\big(\chi(\x)\big)^\top\big\rangle
		\end{align*}
\end{proof}

\subsection{Resolvents and connection formul{\ae}}

\begin{definition}\label{def:resolvents_linear}
	The resolvent  matrices are
	\begin{align}\label{eq:resolvents}
	\omega_1\coloneq&\hat S \mathcal Q_1(\boldsymbol{\Lambda}) S^{-1}, &
		(\omega_2)^\top\coloneq S\mathcal Q_2(\boldsymbol\Lambda)\hat S^{-1},
	\end{align}
	given in terms of the lower unitriangular matrices $S$ and $\hat S$ of the Cholesky factorizations of the moment matrices $G=S^{-1}H(S^{-1})^\top$ and  $\hat G=(\hat S)^{-1}(\hat H)(\hat S^{-1})^\top$.
\end{definition}

\begin{pro}
	The resolvent matrices satisfy
	\begin{align}\label{eq:M-omega2}
	\hat H\omega_2=\omega_1H.
	\end{align}
\end{pro}
\begin{proof}
	It follows from the Cholesky factorization of $G$ and $\hat G$ and from Proposition \ref{pro:fac_TG_2}.
\end{proof}
\begin{pro}\label{pro:superdiagonals2}
The resolvent matrices $\omega_1$ and $\omega_2$ are block banded matrices. All their block superdiagonals above the $m_1$-th and all their subdiagonals below $m_2$-th are zero. In particular, the $m_1$-th block superdiagonal of $\omega_1$ is  $\mathcal Q_1^{(m_1)}(\boldsymbol \Lambda)$  while its $m_2$-th block subdiagonal  is $\hat H\big(\mathcal Q_2^{(m_2)}(\boldsymbol \Lambda^\top)\big)H^{-1}$.
\end{pro}
\begin{proof}
	From Definition \ref{def:resolvents_linear} we deduce that both $\omega_1$ or $(\omega_2)^\top$ are semi-infinite matrices with all its block superdiagonals  outside the block diagonal band going from the $m_1$-th superdiagonal to $m_2$-th subdiagonal being zero, and with the  $m_1 $ or $m_2$ superdiagonal equal to $\mathcal Q_1^{(m_1)}(\boldsymbol \Lambda)$ and $\mathcal Q_2^{(m_2)}(\boldsymbol \Lambda)$, respectively.
	Consequently, if \eqref{eq:M-omega2} is taken into account we deduce the band block structure.
\end{proof}
Incidentally, and as a byproduct let us notice
\begin{pro}
	The following factorizations hold
\begin{align}\label{eq:factorization_linear_jacobi}
	&\begin{aligned}
	\mathcal Q_1(\boldsymbol J)\mathcal Q_2(\boldsymbol J)=\mathcal Q_2(\boldsymbol J)\mathcal Q_1(\boldsymbol J)&=
	\big(\omega_2\big)^\top\omega_1,\\
	\mathcal Q_1(\boldsymbol{\hat {J}})
	\mathcal Q_2
	(\boldsymbol{\hat {J}})
	=
	\mathcal Q_2(\boldsymbol{\hat {J}})\mathcal Q_1(\boldsymbol{\hat {J}} )&=
	\omega_1\big(\omega_2\big)^\top.	\end{aligned}
\end{align}
The truncations satisfy
\begin{align*}
\det\Big[ \big(\mathcal Q_1(\boldsymbol J)\big)^{[k]}\Big]&=\det\Big(\big(\omega_1)^{[k]}\Big), &
\det\Big[ \big(\mathcal Q_2(\boldsymbol{ \hat {J}})\big)^{[k]}\Big]&= \det \Big(\big(\omega_2\big)^{[k]}\Big). &
\end{align*}
\end{pro}
\begin{proof}
	In the one hand, Definitions \ref{def: shift_jacobi} and \ref{def:resolvents_linear} imply
\begin{align*}
\mathcal Q_1(\boldsymbol J)&=S\hat S^{-1} \omega_1, & \mathcal Q_2(\boldsymbol J)&=\omega_2^\top \hat SS^{-1},\\
\mathcal Q_1(\boldsymbol{ \hat {J}})&= \omega_1S\hat S^{-1}, & \mathcal Q_2(\boldsymbol{ \hat {J}})&=\hat SS^{-1}\omega_2^\top,
\end{align*}
from where we conclude the factorizations \eqref{eq:factorization_linear_jacobi}.
\end{proof}

\begin{pro}[Connection formul\ae]
	The followings relations are fulfilled
	\begin{align}
	\notag	(\omega_2)^\top\hat P(\x)&=\mathcal Q_2(\x) P(\x),\\
	\label{omega-P_2}	\omega_1 P(\x)&=\mathcal Q_1(\x)\hat P(\x).
	\end{align}
\end{pro}
\begin{proof}
	It follows from \eqref{eq:polynomials} and Definition \ref{def:resolvents_linear}.
\end{proof}

\subsection{The multivariate Christoffel--Geronimus--Uvarov formula}

We are ready to deduce a multivariate extension of the Christoffel--Geronimus--Uvarov formula for linear spectral transformations,
\cite{Christoffel1858uber,Geronimus1940polynomials,Uvarov1969connection,Zhedanov1997Rational}.

\begin{definition}\label{def:r2}
	We introduce the semi-infinite block matrices
	\begin{align*}
	R&\coloneq \big\langle\check u, P(\x)\big(\chi(\x)\big)^\top\big\rangle.
	\end{align*}
\end{definition}

\begin{pro}
	The formula
	\begin{align*}
	R&=\rho+\theta, &
	\rho&\coloneq
	\left\langle u,\frac{P(\x)\big(\chi(\x)\big)^\top}{\mathcal Q_2(\x)} \right\rangle, &
	\theta&\coloneq \left\langle v, \frac{P(\x)\big(\chi(\x)\big)^\top}{\mathcal Q_1(\x)}\right\rangle,
	\end{align*}
	holds.
\end{pro}
\begin{proof}
	Just write $\check u=\frac{u}{\mathcal Q_2}+\frac{v}{\mathcal Q_1}$, with $(\mathcal Q_2)\subseteq\operatorname{Ker} v$, and $\hat u=\frac{\mathcal Q_1}{\mathcal Q_2}u+v$.
\end{proof}

As in the Geronimus situation
\begin{pro}
	When the linear functional $u$ is of order zero with  associated  Borel measure $\d\mu(\x)$ we have
\begin{align*}
\rho=\int P(\x)(\chi(\x))^\top\frac{\d\mu(\x)}{\mathcal Q_2(\x)}
\end{align*}
and  for a given prime factorization  $\mathcal Q_2=(\mathcal Q_{2,1})^{d_1}\cdots (\mathcal Q_{2,N})^{d_N}$ and $v$  taken as in \eqref{eq:v_general}
we can write
\begin{align*}
\theta=\sum_{i=1}^N\sum_{\substack{\q\in\Z_+^D\\|\q|<d_i}}\int_{Z(\mathcal Q_{2,i})}
\frac{\partial^\q }{\partial \x^\q}
\Big(
\frac{P(\x)(\chi(\x))^\top}{\mathcal Q_1(\x)}
	\Big)
 \d\xi_{i,\q}(\x).
\end{align*}
\end{pro}

%For the Uvarov case where $\mathcal Q_1(\x)=\mathcal Q_2(\x)=\mathcal Q(\x)=\prod_{i=1}^N\mathcal (\Pi_i(\x))^{d_i}$, being $\Pi_i(\x)$ irreducible polynomials, we have
%\begin{align*}
%\rho&=\int P(\x)(\chi(\x))^\top\frac{\d\mu(\x)}{\mathcal Q(\x)},\\
%\theta&=\sum_{i=1}^N\sum_{\substack{\q\in\Z_+^D\\|\q|<d_i}}\int_{Z(\Pi_{i})}
%\frac{\partial^\q }{\partial \x^\q}
%\Big(
%\frac{P(\x)(\chi(\x))^\top}{\mathcal Q(\x)}
%\Big)
%\d\xi_{i,\q}(\x).
%\end{align*}

\begin{pro}\label{pro:r-alpha_2}
	The following relations
	\begin{align}\label{eq:omega_R}
(\omega_1 R)_{[k],[l]}&=0, & l<k,\\
\label{eq:omega_RH}
(\omega_1 R)_{[k],[k]}&=\hat H_k,
	\end{align}
	hold for the linear spectral type transformation.
\end{pro}
\begin{proof}
Just follow  the proof of Proposition \ref{pro:r-alpha}.
\end{proof}

\begin{definition}\label{def:linear_spectral_1}
	For $m_1>0$ we consider a set of different multi-indices $\mathcal M_k=\big\{\boldsymbol\beta_i: |\boldsymbol\beta_i|<k\big\}_{i=1}^{r_{2|k,m_2}}$, with cardinal given by
	\begin{align*}
	r_{2|k,m_2}\coloneq|\mathcal M_k|=\ccases{
		N_{k-1}=|[0]|+\dots+|[k-1]|,& k<m_2\\
		N_{k-1}-N_{k-m_2-1}= |[k-m_2]|+\cdots+|[k-1]|, & k\geq m_2.
	}
	\end{align*}
	We also  consider a set of different nodes  $\mathcal N_k=\big\{\boldsymbol p_i\big\}_{i=1}^{r_{1|k,m_1}}$, in the algebraic hypersurface $Z(\mathcal Q_1)$ of  zeroes of $\mathcal Q_1$,  where
\begin{align*}
r_{1|k,m_1}\coloneq |\mathcal N_k|=N_{k+m_1-1}-N_{k-1}=|[k]|+\cdots+|[k+m_1-1]|.
\end{align*}
Finally, we introduce the set $\mathcal S_k\coloneq\mathcal M_k\cup\mathcal N_k$,
the union of the sets of multi-indices and nodes with cardinal given by
\begin{align*}
r_{k,m}\coloneq |\mathcal S_k|=r_{1|k,m_1}+r_{2|k,m_2}=\ccases{
	N_{k+m_1-1}, & k<m_2,\\
	N_{k+m_1-1}-N_{k-m_2-1}, &k\geq m_2.
}
\end{align*}
\end{definition}

\begin{definition}\label{def:linear_spectral_2}
	When $k< m_2$ a set of nodes is poised if
	\begin{align*}
			\begin{vmatrix}
			R_{[0],[0] }&\dots & R_{[0],[k-1]}& P_{[0]}(\boldsymbol p_1) & \dots & P_{[0]}(\boldsymbol p_{r_{1|k,m_1}})\\
			\vdots & &\vdots&\vdots& &\vdots\\
			R_{[k+m_1-1],[k-1]} &\dots & R_{[k+m_1-1],[k-1]}&	P_{[k+m_1-1]}(\boldsymbol p_1) & \dots & P_{[k+m_1-1]}(\boldsymbol p_{r_{1|k,m_1}})
			\end{vmatrix}\neq 0.
	\end{align*}
For $k\geq m_2$, we say that the set $\mathcal S_k$ of nodes and mult-indices is poised if
\begin{align*}
		\begin{vmatrix}
		R_{[k-m_2],\boldsymbol \beta_1} &\dots & R_{[k-m_2],\boldsymbol \beta_{r_{2|k,m_2}}} &	P_{[k-m_2]}(\boldsymbol p_1) & \dots & P_{[k-m_2]}(\boldsymbol p_{r_{1|k,m_1}})
			\\
			\vdots & &\vdots&\vdots& &\vdots\\
		R_{[k+m_1-1],\boldsymbol \beta_1} &\dots & R_{[k+m_1-1],\boldsymbol \beta_{r_{2|k,m_2}}}& 	P_{[k+m_1-1]}(\boldsymbol p_1) & \dots & P_{[k+m_1-1]}(\boldsymbol p_{r_{1|k,m_1}})&
		\end{vmatrix}\neq 0.
\end{align*}
\end{definition}
\begin{theorem}[Christoffel--Geronimus--Uvarov formula for multivariate linear spectral transformations]\label{the:linear_spectral}
	Given a poised set $\mathcal S_k$, of multi-indices and nodes, the perturbed orthogonal polynomials,
	generated by the linear spectral transformation given in Definition \ref{def:linear_spectral},
	 can be expressed, for each $k\in\Z_+$, as
\begin{multline*}
	\hat P_{[k]}(\x)=\frac{\big(\mathcal Q_1(\boldsymbol \Lambda)\big)_{[k],[k+m_1]}}{\mathcal Q_1(\x)}\\\times \Theta_* \PARENS{
		\begin{matrix}
		R_{[0],[0] }&\dots & R_{[0],[k-1]}& P_{[0]}(\boldsymbol p_1) & \dots & P_{[0]}(\boldsymbol p_{r_{1|k,m_1}})&P_{[0]}(\x)\\
		\vdots & &\vdots&\vdots& &\vdots&\vdots\\
		R_{[k+m_1],[0]} &\dots & R_{[k+m_1],[k-1]}&	P_{[k+m_1]}(\boldsymbol p_1) & \dots & P_{[k+m_1]}(\boldsymbol p_{r_{1|k,m_1}})&
	P_{[k+m_1]}(\x)
		\end{matrix}},
 	\end{multline*}
 	and
 	\begin{multline*}
 	\hat H_{[k]}=\big(\mathcal Q_1(\boldsymbol \Lambda)\big)_{[k],[k+m_1]}\\\times \Theta_* \PARENS{
 		\begin{matrix}
 		R_{[0],[0] }&\dots & R_{[0],[k-1]}& P_{[0]}(\boldsymbol p_1) & \dots & P_{[0]}(\boldsymbol p_{r_{1|k,m_1}})&R_{[0],[k]}\\
 		\vdots & &\vdots&\vdots& &\vdots&\vdots\\
 		R_{[k+m_1],[0]} &\dots & R_{[k+m_1],[k-1]}&	P_{[k+m_1]}(\boldsymbol p_1) & \dots & P_{[k+m_1]}(\boldsymbol p_{r_{1|k,m_1}})&
 		R_{[k+m_1],[k]}
 		\end{matrix}}.
 	\end{multline*}
 	When $k\geq m_2$, we also have for the perturbed MVOPR
\begin{multline*}
\hat P_{[k]}(\x)=\frac{\big(\mathcal Q_1(\boldsymbol \Lambda)\big)_{[k],[k+m_1]}}{\mathcal Q_1(\x)}\\\times\Theta_* \PARENS{
	\begin{matrix}
R_{[k-m_2],\boldsymbol \beta_1} &\dots & R_{[k-m_2],\boldsymbol \beta_{r_{2|k,m_2}}}&	P_{[k-m_2]}(\boldsymbol p_1) & \dots & P_{[k-m_2]}(\boldsymbol p_{r_{1|k,m_1}})
	&P_{[k-m_2]}(\x)\\
	\vdots & &\vdots&\vdots& &\vdots&\vdots\\
	R_{[k+m_1],\boldsymbol \beta_1} &\dots & R_{[k+m_1],\boldsymbol \beta_{r_{2|k,m_2}}}&	P_{[k+m_1]}(\boldsymbol p_1) & \dots & P_{[k+m_1]}(\boldsymbol p_{r_{1|k,m_1}})&
P_{[k+m_1]}(\x)
	\end{matrix}}.
\end{multline*}
The quasi-tau matrices  are subject to
\begin{multline*}
\hat H_{[k]}\Big(\big(\mathcal Q_2(\boldsymbol \Lambda)\big)_{[k-m_2],[k]}\Big)^\top=\big(\mathcal Q_1(\boldsymbol \Lambda)\big)_{[k],[k+m_1]}\\\times\Theta_* \PARENS{
	\begin{matrix}
	R_{[k-m_2],\boldsymbol \beta_1} &\dots & R_{[k-m_2],\boldsymbol \beta_{r_{2|k,m_2}}}&	P_{[k-m_2]}(\boldsymbol p_1) & \dots & P_{[k-m_2]}(\boldsymbol p_{r_{1|k,m_1}})
	&H_{[k-m_2]}\\
		R_{[k-m_2+1],\boldsymbol \beta_1} &\dots & R_{[k-m_2+1],\boldsymbol \beta_{r_{2|k,m_2}}}&	P_{[k-m_2+1]}(\boldsymbol p_1) & \dots & P_{[k-m_2+1]}(\boldsymbol p_{r_{1|k,m_1}})
		&0_{[k-m_2+1],[k-m_2]}\\
	\vdots & &\vdots&\vdots& &\vdots&\vdots\\
	R_{[k+m_1],\boldsymbol \beta_1} &\dots & R_{[k+m_1],\boldsymbol \beta_{r_{2|k,m_2}}}&	P_{[k+m_1]}(\boldsymbol p_1) & \dots & P_{[k+m_1]}(\boldsymbol p_{r_{1|k,m_1}})&
	0_{[k+m_1],[k-m_2]}
	\end{matrix}}
\end{multline*}
or
\begin{multline*}
\hat H_{[k]}=\big(\mathcal Q_1(\boldsymbol \Lambda)\big)_{[k],[k+m_1]}\\\times\Theta_* \PARENS{
	\begin{matrix}
	R_{[k-m_2],\boldsymbol \beta_1} &\dots & R_{[k-m_2],\boldsymbol \beta_{r_{2|k,m_2}}}&	P_{[k-m_2]}(\boldsymbol p_1) & \dots & P_{[k-m_2]}(\boldsymbol p_{r_{1|k,m_1}})
	&R_{[k-m_2],[k]}\\
	R_{[k-m_2+1],\boldsymbol \beta_1} &\dots & R_{[k-m_2+1],\boldsymbol \beta_{r_{2|k,m_2}}}&	P_{[k-m_2+1]}(\boldsymbol p_1) & \dots & P_{[k-m_2+1]}(\boldsymbol p_{r_{1|k,m_1}})
	&R_{[k-m_2+1],[k]}\\
	\vdots & &\vdots&\vdots& &\vdots&\vdots\\
	R_{[k+m_1],\boldsymbol \beta_1} &\dots & R_{[k+m_1],\boldsymbol \beta_{r_{2|k,m_2}}}&	P_{[k+m_1]}(\boldsymbol p_1) & \dots & P_{[k+m_1]}(\boldsymbol p_{r_{1|k,m_1}})&
	R_{[k],[k]}
	\end{matrix}}.
\end{multline*}
\end{theorem}
\begin{proof}
First, we reckon that
	\begin{align*}
	(\omega_1)_{[k],[k+m_1]}&=\big(\mathcal Q_1(\boldsymbol \Lambda)\big)_{[k],[k+m_1]}.
	\end{align*}
	Second, we analyze the consequences of \eqref{eq:omega_R} and \eqref{omega-P_2}.
	In the one hand, from \eqref{eq:omega_R}  we have for $l<k$
	%(for each $k\in\Z_+$ we have $N_{k-1}$ multi-indices $\q\in\Z_+^D$ with $l< k$ )
	\begin{align*}
	(\omega_1)_{[k],[0]}R_{[0],[l]}+\dots+ (\omega_1)_{[k],[k+m_1]}R_{[k+m_1-1],[l]}=-\big(\mathcal Q_1(\boldsymbol \Lambda)\big)_{[k],[k+m_1]}(R)_{[k+m_1],[l]}.
	\end{align*}
	Moreover, when $k\geq m_2$ and $l<k$,  it is also true that
	\begin{align*}
	(\omega_1)_{[k],[k-m_2]}R_{[k-m_2],[l]}+\dots+ (\omega_1)_{[k],[k+m_1]}R_{[k+m_1-1],[l]}=-\big(\mathcal Q_1(\boldsymbol \Lambda)\big)_{[k],[k+m_1]}(R)_{[k+m_1],[l]}.
	\end{align*}
	On the other hand,  from \eqref{omega-P_2},  given  a zero $\boldsymbol p$ of $\mathcal Q_1(\x)$ we can write
	\begin{align*}
	(\omega_1)_{[k],[0]}P_{[0]}(\boldsymbol p)+\dots+ (\omega_1)_{[k],[k+m_1-1]}P_{[k+m_1-1]}(\boldsymbol p)=-\big(\mathcal Q_1(\boldsymbol \Lambda)\big)_{[k],[k+m_1]}P_{[k+m_1]}(\boldsymbol p),
	\end{align*}
	and when $k\geq m_2$ it can be written as follows
	\begin{align*}
	(\omega_1)_{[k],[k-m_2]}P_{[0]}(\boldsymbol p)+\dots+ (\omega_1)_{[k],[k+m_1]}P_{[k+m_1-1]}(\boldsymbol p)=-
	\big(\mathcal Q_1(\boldsymbol \Lambda)\big)_{[k],[k+m_1]}
	P_{[k+m_1]}(\boldsymbol p).
	\end{align*}
	Regarding the sizes of the resolvent matrices involved  let us remark
	\begin{align*}
	\big((\omega_1)_{[k],[0]},\dots, (\omega_1)_{[k],[k+m_1-1]}\big)&\in
	\R^{|[k]|\times (N_{k+m_1-1})},\\
	\big((\omega_1)_{[k],[k-m_2]},\dots, (\omega_1)_{[k],[k+m_1-1]}\big)&\in
	\R^{|[k]|\times (N_{k+m_1-1} -N_{k-m_2-1})}, &k\geq m_2.
	\end{align*}
	
	Thus, for $k<m_2$ we can write
		\begin{multline}\label{eq:ERL1}
		((\omega_1)_{[k],[0]}, \dots, (\omega_1)_{[k],[k+m_1-1]})=\\\begin{multlined}
		-\big(\mathcal Q_1(\boldsymbol \Lambda)\big)_{[k],[k+m_1]}\Big(R_{[k+m_1],[0]} ,\dots , R_{[k+m_1],[k-1]},P_{[k+m_1]}(\boldsymbol p_1),\dots , P_{[k+m_1]}(\boldsymbol p_{r_{1|k,m_1}})\Big)\\
		\times\PARENS{
			\begin{matrix}
			R_{[k],[0]} &\dots & R_{[0],[k-1]}&	P_{[k]}(\boldsymbol p_1) & \dots & P_{[k]}(\boldsymbol p_{r_{1|k,m_1}})\\
			\vdots & &\vdots&\vdots& &\vdots\\
			R_{[k+m_1-1],[0]} &\dots & R_{[k+m_1-1],[k-1]}&	P_{[k+m_1-1]}(\boldsymbol p_1) & \dots & P_{[k+m_1-1]}(\boldsymbol p_{r_{1|k,m_1}})
			\end{matrix}}^{-1},
		\end{multlined}
		\end{multline}
while	for $k\geq m_2$
	\begin{multline}\label{eq:ERL}
			((\omega_1)_{[k],[k-m_2]}, \dots, (\omega_1)_{[k],[k+m_1-1]})=\\\begin{multlined}
	-\big(\mathcal Q_1(\boldsymbol \Lambda)\big)_{[k],[k+m_1]}\Big(R_{[k+m_1],\boldsymbol \beta_1} ,\dots , R_{[k+m_1],\boldsymbol \beta_{r_{2|k,m_2}}},P_{[k+m_1]}(\boldsymbol p_1),\dots , P_{[k+m_1]}(\boldsymbol p_{r_{1|k,m_1}})\Big)\\
		\times\PARENS{
			\begin{matrix}
			R_{[k-m_2],\boldsymbol \beta_1} &\dots & R_{[k-m_2],\boldsymbol \beta_{r_{2|k,m_2}}}&	P_{[k-m_2]}(\boldsymbol p_1) & \dots & P_{[k-m_2]}(\boldsymbol p_{r_{1|k,m_1}})\\
			\vdots & &\vdots&\vdots& &\vdots\\
			R_{[k+m_1-1],\boldsymbol \beta_1} &\dots & R_{[k+m_1-1],\boldsymbol \beta_{r_{2|k,m_2}}}&	P_{[k+m_1-1]}(\boldsymbol p_1) & \dots & P_{[k+m_1-1]}(\boldsymbol p_{r_{1|k,m_1}})
			\end{matrix}}^{-1},
		\end{multlined}
		\end{multline}
	and similarly for $k<m_2$.
	Now, recalling the connection formula \eqref{omega-P_2} we derive the stated result.
	
	 Proposition \ref{pro:superdiagonals2} implies
	\begin{align*}
	(\omega_1)_{[k],[k-m_2]}=\hat H_{[k]}\Big(\big(\mathcal Q_2(\boldsymbol \Lambda)\big)_{[k-m_2],[k]}\Big)^\top \big(H_{[k-m_2]}\big)^{-1},
	\end{align*}
	i.e., the first quasi-determinantal expression for $\hat H_k$ is proven.

	Finally, from \eqref{eq:omega_RH} we get
		\begin{align*}
		(\omega_1)_{[k],[k-m_2]}R_{[k-m_2],[k]}+\dots+ (\omega_1)_{[k],[k+m_1]}R _{[k+m_1-1],[k]}+\big(\mathcal Q_1(\boldsymbol \Lambda)\big)_{[k],[k+m_1]}R_{[k+m_1],[k]}=\hat H_{[k]},
		\end{align*}
		and \eqref{eq:ERL} we get the second quasi-determinantal expression for $\hat H_k$.
\end{proof}

 For the finding of a multivariate Christoffel formula for Christoffel transformations we need the concourse of poised sets, and the existence of them deeply depends on the algebraic hypersurface of the zeros $Z(\mathcal Q_1(\x))$ of the perturbing polynomial $\mathcal Q_1(\x)$, see \cite{ariznabarreta2015darboux}. In fact, for a factorization in terms of irreducible polynomials, $\mathcal Q_1(\x)=\prod_{i=1}^N \big(\mathcal Q_{1,i}(\x)\big)^{d_i}$,  with $d_1=\cdots=d_N=1$ we require the poised set to  belong only to the mentioned algebraic hypersurface and not to any other of lower degree.  Moreover, if any of the multiplicities $d_1,\dots,d_N$ is bigger than 1 we need to introduce multi-Wronskians expressions. For the Geronimus case this is not necessary as we have  Wronskians already encoded  in the linear functional $v$ and, consequently, in $R$. However, the linear spectral transformations is a composition of  Geronimus and  Christoffel transformations. Therefore,  we have a similar situation as that described in \cite{ariznabarreta2015darboux}. In fact, to have poised sets the requirements discussed in that paper are necessary. Thus, the formul{\ae} given make sense only when
 all multiplicities of the irreducible factors of $\mathcal Q_1$ are 1. Otherwise, a multi-Wronskian generalization is needed  \cite{ariznabarreta2015darboux}.
 %Obviously this is also true for the particular case of multivariate Uvarov transformations.

\subsection{The 1D case: recovering the 1D Christoffel--Geronimus--Uvarov formula}
In the scalar case $D=1$ we take two polynomials with simple roots
\begin{align*}
\mathcal Q_1(x)&=(x-p_1)\cdots (x-p_{m_1}), &\mathcal Q_2(x)&=(x-q_1)\cdots (x-q_{m_2}).
\end{align*}
Then, we have $r_{1|k,m_1}=m_1$ and $r_{2|k,m_2}=m_2$ and we can take the $m_2$ indexes (not \emph{multi} as we have $D=1$) as $\beta=0,1,\dots,m_2-1$ (we have more possibilities).
Moreover, we have
\begin{align*}
\PARENS{
	\begin{matrix}
	\rho_{k-m_2,0}&\dots & \rho_{k-m_2,m_2-1}\\
	\vdots & &\vdots\\
	\rho_{k+m_1-1,0}&\dots & \rho_{k+m_1-1,m_2-1}	
	\end{matrix}}
&=
\PARENS{
	\begin{matrix}
	C_{k-m_2}(q_1)&\dots & C_{k-m_2}(q_{m_2})\\
	\vdots & &\vdots\\
	C_{k+m_1-1}(q_1)&\dots &C_{k+m_1-1}(q_{m_2})
	\end{matrix}}\mathcal D^{-1}
\mathcal V^\top,\\
\big(
	\rho_{k+m_1,0},\dots, \rho_{k+m_1,m_2-1}
\big)
&=\big(
	C_{k+m_1}(q_1),\dots ,C_{k+m_1}(q_{m_2})
\big)\mathcal D^{-1}
\mathcal V^\top.
\end{align*}
For the $\theta_{k,n}$ terms we must recall that the general form of $\d\nu$ in the 1D scenario is given in \eqref{eq:nu},
and obtain
\begin{align*}
\PARENS{
	\begin{matrix}
	\theta_{k-m_2,0}&\dots & \theta_{k-m_2,m_2}\\
	\vdots & &\vdots\\
	\theta_{k,0}&\dots & \theta_{k,m_2}	
	\end{matrix}}
&=\PARENS{
	\begin{matrix}
	P_{k-m_2}(q_1)&\dots & P_{k-m_2}(q_{m_2})\\
	\vdots & &\vdots\\
	P_{k}(q_1)&\dots &P_{k}(q_{m_2})
	\end{matrix}}\xi\mathcal D^{-1}
\mathcal V^\top,\\
\big(
\theta_{k+m_1,0},\dots, \theta_{k+m_1,m_2-1}
\big)
&=\big(
P_{k+m_1}(q_1),\dots ,P_{k+m_1}(q_{m_2})
\big)\xi\mathcal D^{-1}
\mathcal V^\top,
\end{align*}
where
\begin{align*}
\xi_j\coloneq\frac{\zeta_j}{\mathcal Q_1(q_j)}\prod\limits_{\substack{i\in\{1,\dots,m_2\}\\ i\neq j}}(q_j-q_i),
\end{align*}
and consider
\begin{align*}
\phi_l(x,\xi)\coloneq C_l(x)+\xi P_l(x).
\end{align*}
Consequently, we have the perturbed polynomials determinantal  expressions
%{\small
	\begin{align*}
	\hat P_{k}(x)&=\frac{1}{\mathcal Q_1(x)}\Theta_* \PARENS{
		\begin{matrix}
		\phi_{k-m_2}(q_1,\xi_1) &\dots & \phi_{k-m_2}(q_{m_2},\xi_{m_2})&	P_{k-m_2}( p_1) & \dots & P_{k-m_2}( p_{m_1})
		&P_{k-m_2}(x)\\
		\vdots & &\vdots&\vdots& &\vdots&\vdots\\
			\phi_{k+m_1}(q_1\xi_1) &\dots & \phi_{k+m_1}(q_{m_2},\xi_{m_2})&	P_{k+m_1}( p_1) & \dots & P_{k+m_1}( p_{m_1})
			&P_{k+m_1}(x)
		\end{matrix}}
		\\
&=	\frac{1}{\mathcal Q_1(x)}
\frac{
				\begin{vmatrix}
				\phi_{k-m_2}(q_1,\xi_1) &\dots & \phi_{k-m_2}(q_{m_2},\xi_{m_2})&	P_{k-m_2}( p_1) & \dots & P_{k-m_2}( p_{m_1})
				&P_{k-m_2}(x)\\
				\vdots & &\vdots&\vdots& &\vdots&\vdots\\
					\phi_{k+m_1}(q_1\xi_1) &\dots & \phi_{k+m_1}(q_{m_2},\xi_{m_2})&	P_{k+m_1}( p_1) & \dots & P_{k+m_1}( p_{m_1})
					&P_{k+m_1}(x)
				\end{vmatrix}
}{
\begin{vmatrix}
								\phi_{k-m_2}(q_1,\xi_1) &\dots & \phi_{k-m_2}(q_{m_2},\xi_{m_2})&	P_{k-m_2}( p_1) & \dots & P_{k-m_2}( p_{m_1})
							\\
								\vdots & &\vdots&\vdots& &\vdots\\
									\phi_{k+m_1-1}(q_1\xi_1) &\dots & \phi_{k+m_1-1}(q_{m_2},\xi_{m_2})&	P_{k+m_1-1}( p_1) & \dots & P_{k+m_1-1}( p_{m_1})
								\end{vmatrix}},
	\end{align*}
	%}
which coincides  with formul{\ae} (3.19) and (3.20) in \cite{Zhedanov1997Rational}. Notice that, it also coincide with the Uvarov's  formul\ae{} in \cite{Uvarov1969connection} when $\xi_i=\zeta_i=0$.
Moreover, for the perturbed  squared norms we have
%{\small
	\begin{align*}
	\hat H_{k}&=\Theta_* \PARENS{
		\begin{matrix}
		\phi_{k-m_2}(q_1,\xi_1) &\dots & \phi_{k-m_2}(q_{m_2},\xi_{m_2})&	P_{k-m_2}( p_1) & \dots & P_{k-m_2}( p_{m_1})
		&H_{k-m_2}\\
		\phi_{k-m_2+1}(q_1,\xi_1) &\dots & \phi_{k-m_2+1}(q_{m_2},\xi_{m_2})&	P_{k-m_2+1}( p_1) & \dots & P_{k-m_2+1}( p_{m_1})
				&0\\
		\vdots & &\vdots&\vdots& &\vdots&\vdots\\
			\phi_{k+m_1}(q_1\xi_1) &\dots & \phi_{k+m_1}(q_{m_2},\xi_{m_2})&	P_{k+m_1}( p_1) & \dots & P_{k+m_1}( p_{m_1})
			&0
		\end{matrix}}
		\\&=(-1)^{k+m_2}
		\frac{
				\begin{vmatrix}
				\phi_{k-m_2+1}(q_1,\xi_1) &\dots & \phi_{k-m_2+1}(q_{m_2},\xi_{m_2})&	P_{k-m_2+1}( p_1) & \dots & P_{k-m_2+1}( p_{m_1})\\
				\vdots & &\vdots&\vdots& &\vdots\\
					\phi_{k+m_1}(q_1\xi_1) &\dots & \phi_{k+m_1}(q_{m_2},\xi_{m_2})&	P_{k+m_1}( p_1) & \dots & P_{k+m_1}( p_{m_1})
				\end{vmatrix}}{\begin{vmatrix}
								\phi_{k-m_2}(q_1,\xi_1) &\dots & \phi_{k-m_2}(q_{m_2},\xi_{m_2})&	P_{k-m_2}( p_1) & \dots & P_{k-m_2}( p_{m_1})\\
								\vdots & &\vdots&\vdots& &\vdots\\
									\phi_{k+m_1-1}(q_1\xi_1) &\dots & \phi_{k+m_1-1}(q_{m_2},\xi_{m_2})&	P_{k+m_1-1}( p_1) & \dots & P_{k+m_1-1}( p_{m_1})
								\end{vmatrix}}H_{k-m_2}.
	\end{align*}

\section{Extension to a multispectral 2D  Toda lattice}\label{S:4}

We explore the situation described in \S \ref{MV} but not specifically with multivariate polynomials in mind. The block structure of the semi-infinite matrices has been  described there. In \cite{ariznabarreta2014multivariate} we considered a semi-infinite matrix $G$ such that $\Lambda_aG=G(\Lambda_a)^\top$, $a\in\{1,\dots,D\}$,  a Cholesky factorization
\begin{align*}
G=S^{-1}H (S)^{-\top},
\end{align*}
and flows preserving this structure. In that manner we obtained nonlinear equations for which the MVOPR provided solutions.
Then, in \cite{ariznabarreta2015darboux} we derived a quasi-determinantal  Christoffel formula for the multivariate Christoffel transformations for MVOPR. A similar development could be performed here with the more general linear spectral transformations, but we will follow an even more general approach.

The Toda type flows discussed in \cite{ariznabarreta2014multivariate} for multivariate  moment matrices can be extended further. The integrable hierarchy has the MVOPR as solutions, but this is only a part of its space of solutions, as the MVOPR sector corresponds to a  particular choice of $G$. In this paper we will analyze this Toda hierarchy, that we name as multispectral 2D Toda hierarchy, in its own, associated as we will see to non standard orthogonality.
Therefore, we now consider  any possible  block Gaussian factorizable semi-infinite matrix
\begin{align*}
G=(S_1)^{-1} H (S_2)^{-\top}
\end{align*}
where, $S_1,S_2$ are lower unitriangular block semi-infinite matrices, and $H$ is a diagonal block semi-infinite matrix.

\subsection{Non-standard multivariate biorthogonality. Bilinear forms}

\begin{definition}
	In the linear space of multivariate polynomials $\R[\x]$ we consider a bilinear form $\langle\cdot,\cdot\rangle$  whose Gramm semi-infinite matrix is $G$, i.e.
\begin{align}\label{eq:bilinear}
\langle P(\x),Q(\x)\rangle&=\sum _{\substack{|\q|\leq\deg P\\|\boldsymbol{\beta}\leq|\deg Q}} P_{\q}G_{\q,\boldsymbol{\beta}}  Q_{\boldsymbol \beta}, & G_{\q,\boldsymbol{\beta}}&=\big\langle \x^\q,\x^{\boldsymbol{\beta}}\big\rangle.
\end{align}
 Whenever the sum $\sum_{\substack{\q,\boldsymbol{\beta}\in\Z_+^D}} P_{\q}G_{\q,\boldsymbol{\beta}}  Q_{\boldsymbol \beta}$ converges in some sense, the corresponding extension of this bilinear form  to the linear space of  power series $\C[\![\x]\!]$ can be considered.
\end{definition}
In general, the semi-infinite matrix $G$ has no further structure and, consequently, we do not expect it to be symmetric or to be related to a linear functional, for example. We say that weare dealing with a non standard bilinear form. The  bilinear form \eqref{eq:bilinear} induces another bilinear form which is a bilinear map  from  semi-infinite vectors of polynomials (or power series when possible) into the  semi-infinite matrices.

\begin{definition}
	 Given to semi-infinite vectors of polynomials $v(\x)=(v_\q(\x))_{\q\in\Z_+^D}$ and $w(\x)=(w_\q(\x))_{\q\in\Z_+^D}$, with $v_\q,w_\q\in\C[\x]$ (or $\C[\![\x]\!]$ when possible) we consider the following semi-infinite matrix
\begin{align*}
\big\langle v(\x),\big(w(\x)\big)^\top\big\rangle&=\Big(\big\langle v(\x), \big(w(\x)\big)^\top \big\rangle_{\q,\boldsymbol{\beta}}\Big), &\big\langle v(\x),\big(w(\x)\big)^\top\big\rangle_{\q,\boldsymbol{\beta}}&\coloneq\big\langle v_\q(\x),w_{\boldsymbol{\beta}}(\x)\big\rangle, & \q,\boldsymbol{\beta}\in\Z_+^D.
\end{align*}
A similar definition holds for a polynomial $p(\x)\in\C[\x]$, i.e.,
\begin{align*}
\langle v(\x),p(\x)\rangle&\coloneq\Big(\big\langle v_\q(\x), p(\x)\big\rangle\Big)_{\q\in \Z_+^D}, &
\langle p(\x),(v(\x))^\top\rangle&\coloneq\bigg(\Big(\big\langle p(\x),v_\q(\x)\big\rangle\Big)_{\q\in \Z_+^D}\bigg)^\top.
\end{align*}
\end{definition}
\begin{pro}\label{pro:bilinear_action}
	Given three semi-infinite vectors  $v^{(i)}(\x)=\big(v^{(i)}_\q(\x)\big)_{\q\in\Z_+^D}$, $i\in\{1,2,3\}$, the formul{\ae}
	\begin{align}\label{eq:bilinear_action}
	\big\langle v^{(1)}(\x),\big(v^{(2)}(\x)\big)^\top\big\rangle v^{(3)}(\z)
	&=
	\big\langle v^{(1)}(\x),\big(v^{(2)}(\x)\big)^\top v^{(3)}(\z)\big\rangle,\\\notag
	\big( v^{(3)}(\z)\big)^\top\big\langle v^{(1)}(\z),\big(v^{(2)}(\x)\big)^\top\big\rangle
	&=
	\big\langle \big(v^{(3)}(\z)\big)^\top v^{(1)}(\x),v^{(2)}(\x)\big\rangle
	\end{align}
	hold.
\end{pro}

Using this non standard bilinear form we can write
\begin{align}\label{eq:G_bilinear}
G=\big\langle \chi(\x), \big(\chi(\x)\big)^\top\big\rangle.
\end{align}
When there is a linear form $u\in\big(\C[\x]\big)'$ such that $\langle P(\x), Q(\x)\rangle=\langle u, P(\x) Q(\x)\rangle$ we find that $G=\big\langle u,\chi(\x)\big(\chi(\x)\big)^\top\big\rangle$ is the corresponding moment matrix.

\begin{pro}\label{pro:G-bilinear form}
	For any polynomial $\mathcal Q(\x)\in\C[\x]$ we have
\begin{align*}
\mathcal Q(\boldsymbol{\Lambda}) G=&\big\langle \mathcal Q(\x)\chi(\x), \big(\chi(\x)\big)^\top\big\rangle, &
G\big(\mathcal Q(\boldsymbol{\Lambda}) \big)^\top=&\big\langle \chi(\x), \big(\chi(\x)\big)^\top\mathcal Q(\x)\big\rangle.
\end{align*}
\end{pro}
\begin{proof}
	Use \eqref{eigen}.
\end{proof}

\subsection{A multispectral 2D Toda hierarchy}

In terms of the  continuous time parameters sequences $t=\{t_1,t_2\}\subset \R$ given by
\begin{align*}
t_i&\coloneq \{t_{i,\q}\}_{\q\in\Z_+^D}, & i\in\{1,2\},
\end{align*}
we consider the time power series
\begin{align*}
t_i(\x)&\coloneq\sum_{\q\in\Z_+^D}t_{i,\q} \x^\q, & i\in\{1,2\},
\end{align*}
the following vacuum wave semi-infinite matrices
\begin{align*}%	\label{eq:undressedW}
W^{(0)}_i(t_i)&=
\exp\Big(\sum_{\q\in\Z_+^D}t_{i,\q}
{\boldsymbol\Lambda}^\q\Big),
 & i&\in\{1,2\},
\end{align*}
and the perturbed semi-infinite matrix
\begin{align}\label{eq:flows}
G(t)=W_1^{(0)}(t_1)G\Big(W_2^{(0)}(t_2)\Big)^{-\top}.
\end{align}
Notice that these flows do respect the multi-Hankel condition, if initially we have  $\Lambda_aG=G(\Lambda_a)^\top$, $a\in\{1,\dots,D\}$, then, for any further time, we will  have  $\Lambda_aG(t)=G(t)\big(\Lambda_a\big)^\top$,  $a\in\{1,\dots,D\}$.

We will assume that the block Gaussian factorization do exist, at least for an open subset of times containing $t=0$
\begin{align}\label{eq:general_LU}
G(t)=\big(S_1(t)\big)^{-1} H(t) \big(S_2(t)\big)^{-\top}.
\end{align}
Then, we consider the semi-infinite vectors of polynomials
\begin{align}\label{eq:def_poly}
P_1(t,\x)&\coloneq S_1(t)\chi(\x), &P_2(t,\x)&\coloneq S_2(t)\chi(\x),
\end{align}
being its component $P_{i,\q}(t,\x)$, $i\in\{1,2\},\q\in\Z_+^D$, a $t$-dependent  monic multivariate polynomial in $\x$ of degree $|\q|$.

Then, the Gaussian factorization  \eqref{eq:general_LU} implies  the bi-orthogonality condition
\begin{align*}
\big\langle P_{1,[k]}(t,\x),P_{2,[l]}(t,\x)\big\rangle=\delta_{k,l}H_{[k]}(t).
\end{align*}
Here we used  the bilinear form $\langle\cdot,\cdot\rangle$ with Gramm matrix $G(t)$. We also consider the wave matrices
 \begin{align}\label{eq:def_wave}
 W_1(t)\coloneq&S_1(t)W_1^{(0)}(t_1), &
 W_2(t)\coloneq & \tilde S_2(t)  \big(W^{(0)}_2(t_2)\big)^{\top},
 \end{align}
 where $\tilde S_2\coloneq H(t)\big(S_2(t)\big)^{-\top} $.

 \begin{pro}
 	The wave matrices satisfy
 	\begin{align}\label{eq:central}
 	\big(W_1(t)\big)^{-1} W_2(t)=G.
 	\end{align}
 \end{pro}
 \begin{proof}
 	It follows from the Gauss--Borel factorization \eqref{eq:general_LU}.
 \end{proof}
 Given a semi-infinite matrix $A$ we have unique splitting $A=A_++A_-$ where $A_+$ is an upper  triangular block matrix while is $A_-$ a  strictly lower triangular block matrix.
 The Gaussian factorization \eqref{eq:central} has the following differential consequences
 \begin{pro}\label{pro:evolution S}
 	The following equations hold
 	\begin{align*}
 	\frac{\partial S_1}{\partial t_{1,\q}}(S_1)^{-1}&=-\Big(S_1\boldsymbol{\Lambda}^{\q} (S_1)^{-1}\Big)_-, &
 	\frac{\partial S_1}{\partial t_{2,\q}}(S_1)^{-1}&=\Big(\tilde S_2\big(\boldsymbol{\Lambda}^{\top}\big)^{\q} (\tilde S_2)^{-1}\Big)_-,\\
 	\frac{\partial \tilde S_2}{\partial t_{1,\q}}(\tilde S_2)^{-1}&=\Big(S_1\boldsymbol{\Lambda}^{\q} (S_1)^{-1}\Big)_+, &
 	\frac{\partial \tilde S_2}{\partial t_{2,\q}}(\tilde S_2)^{-1}&=-\Big(\tilde S_2\big(\boldsymbol{\Lambda}^{\top}\big)^{\q} (\tilde S_2)^{-1}\Big)_+.
 	\end{align*}
 \end{pro}
 \begin{proof}
 	Taking right derivatives  of \eqref{eq:central} yields
 	\begin{align*}
 	\frac{\partial W_1}{\partial t_{i,\q}}(W_1)^{-1}&=\frac{\partial  W_2}{\partial t_{i,\q}}( W_2)^{-1}, & i&\in\{1,2\}, & j&\in\Z_+,
 	\end{align*}
 	where
 	\begin{align*}
 	\frac{\partial W_1}{\partial t_{1,\q}}(W_1)^{-1}&=\frac{\partial S_1}{\partial t_{1,\q}}(S_1)^{-1}+S_1\boldsymbol{\Lambda}^{\q} (S_1)^{-1}, &
 	\frac{\partial W_1}{\partial t_{2,\q}}(W_1)^{-1}&=\frac{\partial S_1}{\partial t_{2,\q}}(S_1)^{-1},\\
 	\frac{\partial  W_2}{\partial t_{1,\q}}( W_2)^{-1}&=\frac{\partial \tilde S_2}{\partial t_{1,\q}}(\tilde S_2)^{-1}, &
 	\frac{\partial  W_2}{\partial t_{2,\q}}( W_2)^{-1}&=\frac{\partial \tilde S_2}{\partial t_{2,\q}}(\tilde S_2)^{-1}+\tilde S_2\big(\boldsymbol{\Lambda}^{\top}\big)^\q (\tilde S_2)^{-1},
 	\end{align*}
 	and the result follows immediately.
 \end{proof}
 As a consequence, we deduce
 \begin{pro}
 	The multicomponent 2D Toda lattice equations
 	\begin{multline*}
 	\frac{\partial}{\partial t_{2,\ee_b
 			}}\Big(\frac{\partial H_{[k]}}{\partial t_{1,\ee_a}}(H_{[k]})^{-1}\Big)
 			+({\Lambda_a})_{[k],[k+1]}H_{[k+1]}\Big((\Lambda_b)_{[k],[k+1]}\Big)^\top(H_{[k]})^{-1}
 			\\-H_{[k]}\Big((\Lambda_b)_{[k-1],[k]}\Big)^\top(H_{[k-1]})^{-1}({\Lambda_a})_{[k-1],[k]}=0
 			%E_{a,a}H_{k+1}E_{b,b}(H_{k})^{-1}-H_kE_{b,b}(H_{k-1})^{-1}E_{a,a}=0.
 	\end{multline*}
 	hold.
 \end{pro}
 \begin{proof}
 	From Proposition \ref{pro:evolution S} we get
 	\begin{align*}
 	\frac{\partial H_{[k]}}{\partial t_{1,\ee_a}}(H_{[k]})^{-1}&=\beta_{[k]}({\Lambda_a})_{[k-1],[k]}-({\Lambda_a})_{[k],[k+1]}\beta_{[k+1]},&
 	\frac{\partial \beta_{[k]}}{\partial t_{2,\ee_b}}&=H_{[k]}\Big((\Lambda_b)_{[k-1],[k]}\Big)^\top(H_{[k-1]})^{-1},
 	\end{align*}
 	where $\beta_{[k]}\in\R^{|[k]|\times |[k-1]|}$, $k=1,2,\dots$,  are the first subdiagonal  coefficients in $S_1$.
 \end{proof}

 These equations are just the first members  of an infinite set of nonlinear partial differential equations, an integrable hierarchy. Its elements are given by
 \begin{definition}\label{def:integrable}
 	The Lax and  Zakharov--Shabat matrices are given by
 	\begin{align*}
 %	\Pi_{1,a}&:=S_1 E_{a,a} (S_1)^{-1}, &\Pi_{2,a}&:=\tilde S_2 E_{a,a} (\tilde S_2)^{-1}, \\
 	L_{1,a}&:=S_1\Lambda_a(S_1)^{-1}, &
 	L_{2,a}&:=\tilde S_2(\Lambda_a)^{\top} (\tilde S_2)^{-1},\\
 	B_{1,\q}&:=\big(({\boldsymbol L}_{1})^\q\big)_+, & B_{2,\q}&:=\big(({\boldsymbol L}_{2})^\q\big)_-.
 	\end{align*}
 	The Baker functions are defined as
 	\begin{align*}
 	\Psi_1(t,\z)&\coloneq W_1(t)\chi(\z), &	\Psi_2(t,\z)&\coloneq  W_2(t)\chi^*(\z),
 	\end{align*}
 	and the adjoint Baker functions by
 	 	\begin{align*}
 	 	\Psi^*_1(t,\z)&\coloneq (W_1(t))^{-\top}\chi^*(\z), &	\Psi^*_2(t,\z)&\coloneq (W_2(t))^{-\top}\chi(\z),
 	 	\end{align*}
 	 	here we switch for $\x\in\R^D$ to $\z\in\C$.
 	 	We also consider the multivariate Cauchy kernel
 	 	\begin{align*}
 	 	\mathcal{C}(\z,\x)\coloneq\frac{1}{\prod_{i=1}^D(z_i-x_i)}.
 	 	\end{align*}
 \end{definition}
 \begin{pro}
 	The Lax matrices can be written as
 	\begin{align}
 	\label{eq:lax_wave}
 	L_{1,a}(t)&=W_1(t)\Lambda_a(W_1(t))^{-1}, &
 		L_{2,a}(t)&=W_2(t)(\Lambda_a)^\top(W_2(t))^{-1},
 	\end{align}
 	and satisfy commutativity properties
 	\begin{align*}
[L_{1,a}(t),L_{1,b}(t)]&=0, & [L_{2,a}(t),L_{2,b}(t)]&=0, & a,b&\in\{1,\dots,D\},
 	\end{align*}
 	and the spectral properties
 	  \begin{align*}
 	  L_{1,a}(t)\Psi_1(t,\x)&=x_a\Psi_1(t,\x),& (L_{2,a}(t))^\top\Psi_2^*(t,\x)&=x_a\Psi_2^*(t,\x), & a&\in\{1,\dots,D\}.
 	  \end{align*}
 	  The Cauchy kernel satisfies
 	  	\begin{align}\label{eq:Cauchy}
 	  	\big(\chi(\x)\big)^\top\chi^*(\z)&=\mathcal C(\z,\x), & |z_i|&>|x_i|, &i&\in\{1,\dots,D\}.
 	  	\end{align}
 \end{pro}

\begin{theorem}
The Baker functions can be expressed  in terms of the orthogonal polynomials, the multivariate Cauchy kernel and the bilinear form as follows
 \begin{align}
 \Psi_1(t,\z)&=\Exp{t_1(\x)}P_1(t,\z), \label{eq:Baker1} \\
 \Psi_2^*(t,\z)&=\Exp{- t_2(\z)}(H(t))^{-\top}P_2(t,\z),\label{eq:Baker2*}\\
 \Psi_2(t,\z) & =\langle\Psi_1(t,\x),\mathcal C(\z,\x)\rangle,& |z_i|&>|x_i|, &i&\in\{1,\dots,D\},\label{eq:Baker2}
 \\
  \big(\Psi_1^*(t,\z) \big)^\top& =\big\langle\mathcal C(\z,\x),\big(\Psi_2^*(t,\x)\big)^\top\big\rangle,& |z_i|&>|x_i|, &i&\in\{1,\dots,D\},\label{eq:Baker1*}
 \end{align}
\end{theorem}\begin{proof}
Equation \eqref{eq:Baker1} follows easily
\begin{align*}
	\Psi_1(t,\x)=&W_1(t)\chi_1(\x),& &\text{from Definition \ref{def:integrable}}\\
	=&S_1(t)W_1^{(0)}(t_1)\chi(\x) & &\text{see \eqref{eq:def_wave}}\\
	=&\Exp{t_1(x)}S_1(t)\chi_1(\x)& &\text{consequence  of \eqref{eigen}}	\\
	=&\Exp{t_1(x)}P_1(t,\x) & & \text{directly from \eqref{eq:def_poly}}.
	\end{align*}
 To get  \eqref{eq:Baker2*} we argue similarly
\begin{align*}
\Psi_2^*(t,\z)=&\big(W_2(t)\big)^{-\top}\chi(\z), & &\text{from Definition \ref{def:integrable}},\\
=&H^{-\dagger}S_2(t)\big(W^{(0)}_2( t_2)\big)^{-1}\chi(\z)&& \text{see \eqref{eq:def_wave}}\\
=&\Exp{-\bar t_2(\z)}H^{-\top}S_2(t)\chi(\z)& &\text{consequence  of \eqref{eigen}}\\
=&\Exp{-\bar t_2(\z)}H^{-\top}P_2(t,\z)& &\text{follows  from \eqref{eq:def_poly}.}\\
\end{align*}
To show \eqref{eq:Baker2} we proceed as follows, assume that $|z_i|>|x_i|$, $i\in\{1,\dots,D\}$.
\begin{align*}
\Psi_2(t,\z)=&W_2(t)\chi^*(\z) & &\text{from Definition \ref{def:integrable}} \\
=&W_1(t) G\chi^*(\z) & & \text{use the factorization \eqref{eq:central}}\\
=& W_1(t) \big\langle\chi(\x),\big(\chi(\x)\big)^\top\big\rangle\chi^*(\z)&& \text{introduce the bilinear form expresion  \eqref{eq:G_bilinear}}\\
=& \big\langle W_1(t) \chi(\x),\big(\chi(\x)\big)^\top\chi^*(\z)\big\rangle&&\text{use porperties  \eqref{eq:bilinear_action}}\\
=& \langle\Psi_1(t,\x),\mathcal C(\z,\x)\rangle&&\text{consequence of \eqref{eq:Cauchy} and Definition  \ref{def:integrable}} .
%=& \langle\Exp{t_1(\x)}P_1(t,\x),\mathcal C(\z,\x)\rangle&&\text{follows from \eqref{eq:Baker1}}.
\end{align*}

We now prove \eqref{eq:Baker1*}, for $|z_i|>|x_i|$, $i\in\{1,\dots,D\}$,
\begin{align*}
\Psi_1^*(t,\z)=&\big(W_1(t)\big)^{-\top}\chi^*(\z)& &\text{from Definition \ref{def:integrable}} \\
=&\big(W_2(t)\big)^{-\top}G^\dagger\chi^*(\z)&& \text{follows from factorization  \eqref{eq:central}}\\
=&\big(W_2(t)\big)^{-\top}\big(\big(\chi^*(\z)\big)^\top G\big)^\top\\
=&\big(W_2(t)\big)^{-\top}\Big(\big\langle\big(\chi^*(\z)\big)^\top \chi(\x),\big(\chi(\x)\big)^\top\big\rangle\Big)^\top&& \text{use the bilinear expression \eqref{eq:G_bilinear}}\\
=&\big(W_2(t)\big)^{-\top}\Big(\big\langle\mathcal C(\z,\x),\big(\chi(\x)\big)^\top\big\rangle\Big)^\top&& \text{see \eqref{eq:Cauchy}}\\
=&\Big(\Big\langle\mathcal C(\z,\x),\Big(\big(W_2(t)\big)^{-\top}\chi(\x)\Big)^\top\Big\rangle\Big)^\top\\
=&\Big(\Big\langle\mathcal C(\z,\x),\big(\Psi_2^*(t,\x)\big)^\top\Big\rangle\Big)^\top&& \text{from Definition \ref{def:integrable}, again.}
%=&(H(t))^{-\top}\Big(\big\langle\mathcal C(\z,\x),\Exp{-t_2(\x)}\big(P_2(t,\x)\big)^\top\big\rangle\Big)^\top&& \text{consequence of \eqref{eq:Baker2*}}.
\end{align*}

\end{proof}

 \begin{pro}[The integrable hierarchy]\label{pro:integrable}
 	The wave matrices obey the evolutionary linear systems
 	\begin{align*}
 	\frac{\partial W_1}{\partial t_{1,\q}}&= B_{1,\q}W_1,&\frac{\partial W_1}{\partial t_{2,\q}}&= B_{2,\q}W_1,&
 	\frac{\partial  W_2}{\partial t_{1,\q}}&=B_{1,\q} W_2, &\frac{\partial  W_2}{\partial t_{2,\q}}&= B_{2,\q} W_2,
 	\end{align*}
 	the Baker and adjoint Baker functions solve the following linear equations
 	\begin{align*}
	\frac{\partial \Psi_1}{\partial t_{1,\q}}&= B_{1,\q}\Psi_1,&\frac{\partial \Psi_1}{\partial t_{2,\q}}&= B_{2,\q}\Psi_1,&
	\frac{\partial  \Psi_2}{\partial t_{1,\q}}&=B_{1,\q} \Psi_2, &\frac{\partial  \Psi_2}{\partial t_{2,\q}}&= B_{2,\q} \Psi_2,\\
		\frac{\partial \Psi_1^*}{\partial t_{1,\q}}&= -(B_{1,\q})^\top\Psi_1,&\frac{\partial \Psi_1}{\partial t_{2,\q}}&= -(B_{2,\q})^\top\Psi^*_1,&
		\frac{\partial  \Psi_2^*}{\partial t_{1,\q}}&=-(B_{1,\q})^\top \Psi^*_2, &\frac{\partial  \Psi^*_2}{\partial t_{2,\q}}&= -(B_{2,\q})^\top \Psi^*_2,
 	\end{align*}
 	the  Lax matrices are subject to the following  \emph{Lax equations}
 	\begin{align*}
 	\frac{\partial L_{i,a}}{\partial t_{j,\q}}&=\big[B_{j,\q}, L_{i,a}\big],
 	\end{align*}
 	and Zakharov--Sabat matrices fulfill the following  \emph{Zakharov--Shabat equations}
 	\begin{align*}
 	\frac{\partial B_{i',\q'}}{\partial t_{i,\q}}- \frac{\partial B_{i,\q}}{\partial t_{i',\q'}}+\big[B_{i,\q},B_{i',\q'}\big]=0.
 	\end{align*}
 \end{pro}
 \begin{proof}
 	Follows from Proposition \ref{pro:evolution S}.
 \end{proof}
 In this Proposition, as expected, given two semi-infinite block matrices $A,B$ the notation $[A,B]=AB-BA$ stands for the usual commutator of matrices.

 \subsection{KP type hierarchies}

 In \cite{ariznabarreta2014multivariate} it is shown that the KP type construction appears also in the MVOPR context. Here we show that they admit an extension to this broader scenario not linked to MVOPR of multispectral Toda hierarchies.

 \begin{definition}\label{def:asymptotic-module}
 	Given two semi-infinite matrices $Z_1(t)$ and $Z_2(t)$ we say that
 	\begin{itemize}
 		\item  $Z_1(t)\in\mathfrak{l}W^{(0)}_1$ if $Z_1(t)\big(W^{(0)}_1(t_1)\big)^{-1}$ is a block strictly lower triangular matrix.
 		\item  $Z_2(t)\in\mathfrak{u}W^{(0)}_2$ if $Z_2(t)\big(W^{(0)}_2(t_2)\big)^{-\top}$ is a block upper triangular matrix.
 	\end{itemize}
 \end{definition}
 Then, we can state the following \emph{congruences}
 \begin{pro}\label{pro:asymptotic-module}
 	Given two semi-infinite matrices $Z_1(t)$ and $Z_2(t)$ such that
 	\begin{itemize}
 		\item  $Z_1(t)\in\mathfrak{l}W^{(0)}_1$,
 		\item  $Z_2(t)\in\mathfrak{u}W^{(0)}_2$,
 		\item $Z_1(t)G=Z_2(t)$.
 	\end{itemize}
 	then
 	\begin{align*}
 	Z_1(t)&=0, & Z_2(t)&=0.
 	\end{align*}
 \end{pro}
 \begin{proof}
 	Observe that
 	\begin{align*}
 	Z_2(t)=Z_1(t)G=Z_1(t)\big(W_1(t)\big)^{-1}W_2(t),
 	\end{align*}
 	where we have used \eqref{eq:central}. From here we get
 	\begin{align*}
 	Z_1(t)\big(W_1^{(0)}(t_1)\big)^{-1}\big(S_1(t)\big)^{-1}=Z_2(t)\big(W_2^{(0)}(t_2)\big)^{-\top}\big(\tilde S _2(t)\big)^{-1},
 	\end{align*}
 	and, as in the LHS we have a strictly lower triangular block semi-infinite matrix while in the RHS we have an upper triangular block semi-infinite  matrix, both sides must vanish and the result follows.
 \end{proof}

 \begin{definition}
 	When $A-B\in\mathfrak{l}W^{(0)}_1$  we write $A=B+\mathfrak{l}W^{(0)}_1$ and if $A-B\in\mathfrak{u}W^{(0)}_2$ we write $A=B+\mathfrak{u}W^{(0)}_2$.
 \end{definition}

 Within this subsection we will write $t_{i,(a_1,a_2,\dots,a_p)}$ to denote $t_{i,\q}$ with $\q=\ee_{a_1}+\dots+\ee_{a_p}$.
  We introduce the diagonal block matrices $V_{a,b}=\diag((V_{a,b})_{[0]},(V_{a,b})_{[1]},(V_{a,b})_{[2]},\dots)$
   \begin{align}\label{eta}
   V_{a,b}\coloneq & \frac{\partial \beta_1}{\partial t_{1,a}}\Lambda_b,&
   (V_{a,b})_{[k]}= &\frac{\partial \beta_{1,[k]}}{\partial t_{1,a}}(\Lambda_b)_{[k-1],[k]}, &
   U_{a,b}\coloneq & -V_{a,b}-V_{b,a},
   \end{align}
in terms of the first block subdiagonal $\beta_1$ of $S_1$.

   \begin{pro}
   	The  Baker function $\Psi_1$ satisfies
   	\begin{align}\label{eq: linear.wave}
   	\frac{\partial \Psi_1}{\partial t_{1,(a,b)}}&=\frac{\partial^2\Psi_1}{\partial t_{1,a}\partial t_{2,b}}+U_{a,b}\Psi_1.
   	\end{align}
   \end{pro}
   \begin{proof}
   	In the one hand,
   	\begin{align*}
  	\frac{\partial W_1}{\partial t_{1,(a,b)}}&=\Big(\frac{\partial S_1}{\partial t_{1,(a,b)}}+S_1\Lambda_a\Lambda_b\Big)W^{(0)}_1(t_1)\\
  \frac{\partial^2 W_1}{\partial t_{1,a}\partial t_{1,b}}&=
  	\Big(  \frac{\partial^2 S_1}{\partial t_{1,a}\partial t_{1,b}}+\frac{\partial S_1}{\partial t_{1,a} }\Lambda_b+\frac{\partial S_1}{\partial t_{1,b} }\Lambda_a
   	+S_1\Lambda_a\Lambda_b\Big)W^{(0)}_1(t_1)
   	\end{align*}
   so that
   	\begin{align*}
  \bigg( \frac{\partial }{\partial t_{1,(a,b)}}-	 \frac{\partial^2 }{\partial t_{1,a}\partial t_{1,b}}\bigg)(W_1)&=-\bigg(\frac{\partial \beta_1}{\partial t_{1,a} }\Lambda_b+\frac{\partial \beta_1}{\partial t_{1,b} }\Lambda_a\bigg)W_1^{(0)}(t_1)+\mathfrak lW_1^{(0)}
   	\end{align*}
   	and, consequently,
   	   	\begin{align*}
   	   	\bigg( \frac{\partial }{\partial t_{1,(a,b)}}-	 \frac{\partial^2 }{\partial t_{1,a}\partial t_{1,b}}+\frac{\partial \beta_1}{\partial t_{1,a} }\Lambda_b+\frac{\partial \beta_1}{\partial t_{1,b} }\Lambda_a\bigg)(W_1)&=\mathfrak lW_1^{(0)}.
   	   	\end{align*}

   	On the other hand,
      	\begin{align*}
      	\frac{\partial W_2}{\partial t_{1,(a,b)}}&=\frac{\partial \tilde S_2}{\partial t_{1,(a,b)}}W^{(0)}_2(t_2),&
      	\frac{\partial^2 W_2}{\partial t_{1,a}\partial t_{1,b}}&=
  \frac{\partial^2 \tilde S_2}{\partial t_{1,a}\partial t_{1,b}}W^{(0)}_2(t_2)
      	\end{align*}
   	Now, we apply Proposition \ref{pro:asymptotic-module} with
   	\begin{align*}
   	Z_i&=  \bigg( \frac{\partial }{\partial t_{1,(a,b)}}-	 \frac{\partial^2 }{\partial t_{1,a}\partial t_{1,b}}-U_{a,b}\bigg)\big(W_i\big), &i &=1,2,
   	\end{align*}
   	to get the result.
   \end{proof}

Proceeding similarly we can reproduce the results of \cite{ariznabarreta2014multivariate} for this more general case. The proofs are essentially as are there with slight modifications as just shown in the above developments.
Associated with the third order times $t_{1,(a,b,c)}$  we introduce the following block diagonal matrices
\begin{align*}
V_{a,b,c}=\diag((V_{a,b,c})_{[0]},(V_{a,b,c})_{[1]},(V_{a,b,c})_{[2]},\dots)
\end{align*}
with
\begin{align*}
V_{a,b,c}\coloneq & \frac{\partial\beta_1^{(2)}}{\partial t_a}\Lambda_b\Lambda_c
-\frac{\partial\beta_1}{\partial t_{1,a}}\Lambda_b\beta_1\Lambda_c,\\ ( V_{a,b,c})_{[k]} =&\bigg(\frac{\partial\beta^{(2)}_{1,[k]}}{\partial t_{1,a}}
\big(\Lambda_b\big)_{[k-2],[k-1]}-\frac{\partial\beta_{1,[k]}}{\partial t_a}\big(\Lambda_b\big)_{[k-1],[k]}\beta_{1,[k]}\bigg)\big(\Lambda_c\big)_{[k-1],[k]},
\end{align*}
The Baker functions $\Psi_1$ satisfies the third order linear differential equations
	\begin{multline*}
	\frac{\partial\Psi_1}{\partial t_{1,(a,b,c)}}=\frac{\partial^3\Psi_1}{\partial t_{1,a}\partial t_{1,b}\partial t_{1,c}}
	-V_{a,b}\frac{\partial\Psi}{\partial t_{c}} -V_{c,a}\frac{\partial\Psi}{\partial t_{1,b}} -V_{b,c}\frac{\partial\Psi}{\partial t_{1,a}}
\\	-\Big(\frac{\partial V_{a,b}}{\partial t_{1c}}+\frac{\partial V_{b,c}}{\partial t_{1,a}}+\frac{\partial V_{c,a}}{\partial t_{1,b}}+ V_{a,b,c}+ V_{b,c,a}+ V_{c,b,a}\Big)\Psi_1,
	\end{multline*}
 and a matrix type KP system of equations for $\beta_{1,[k]}$ and $\beta_{1,{[k]}}^{(2)}$ emerges \cite{ariznabarreta2014multivariate}. For example, if we denote  $t_{1,a}^{(3)}=t_{3,(a,a,a)}$ and $t_{1,a}^{(2)}=t_{1,{(a,a)}}$ we get the nonlinear partial differential system
	\begin{align*}
	0=&
	\frac{\partial}{\partial t_{1,a}}\left[\frac{\partial \beta_1}{\partial t_{1,a}} \Lambda_a \beta_1-	\frac{\partial \beta_1^{(2)}}{\partial t_{1,a}}\Lambda_a-
	\frac{1}{2}\frac{\partial^2 \beta_1}{\partial t_{1,a}^2}+\frac{1}{4}\frac{\partial \beta_1}{\partial t_{1,a}^{(2)}} \right],\\
	0=&3\frac{\partial^2}{\partial t_{1,a}^2}\left[\frac{1}{2}\frac{\partial \beta_1 }{\partial t_{1,a}^{(2)}}-\frac{\partial^2 \beta_1 }{\partial t_{1,a}^{2}}+2\frac{\partial \beta_1 }{\partial t_{1,a}}\Lambda_a \beta_1 \right]\Lambda_a
	\\&+\frac{\partial}{\partial t_{1,a}}\left[2\frac{\partial^3 \beta_1 }{\partial t_{1,a}^{3}}-\frac{\partial \beta_1 }{\partial t_{1,a}^{(3)}}
	+\left(\frac{\partial \beta }{\partial t_{1,a}}\Lambda_a \beta_1-\frac{\partial \beta_1^{(2)} }{\partial t_{1,a}} \Lambda_a\right)\Lambda_a \beta_1\right]\Lambda_a\\
	&+3\frac{\partial}{\partial t_{1,a}}\left[\left(2\frac{\partial \beta_1 }{\partial t_{1,a}}\Lambda_a \beta_1^{(2)}+
	\frac{1}{2}\frac{\partial \beta_1^{(2)}}{\partial t_{1,a}^{(2)}}-\frac{\partial^2 \beta_1^{(2)}}{\partial t_{1,a}^2}\right)\Lambda_a^2
	-2\frac{\partial \beta_1}{\partial t_{1,a}}\Lambda_a \frac{\partial \beta_1}{\partial t_{1,a}}\Lambda_a \right]\\
	&+3 \frac{\partial \beta_1}{\partial t_{1,a}}\Lambda_a \left[\frac{\partial^2 \beta_1}{\partial t_{1,a}^2}-
	2\frac{\partial \beta_1}{\partial t_{1,a}}\Lambda_a \beta_1 -\frac{1}{2}\frac{\partial \beta_1}{\partial t_{1,a}}\right]\Lambda_a
	-6\frac{\partial^2 \beta_1}{\partial t_{1,a}^2} \Lambda_a \beta_1^{(2)}(\Lambda_a) ^2.
	\end{align*}

	\subsection{Reductions}
	We explore superficially some possibilities for reductions
\begin{definition}
Given two polynomials $\mathcal Q_1(\x),\mathcal Q_2(\x)\in\C[\x]$ a semi-infinite matrix $G$  is said $(\mathcal Q_1,\mathcal Q_2)$-invariant if
\begin{align}\label{eq:invariance_polynomial}
\mathcal Q_1(\boldsymbol \Lambda)G=G \mathcal Q_2(\boldsymbol \Lambda^\top)
\end{align}	
We will use the notation
\begin{align*}
\boldsymbol L_1&\coloneq(L_{1,1},\dots, L_{1,D})^\top, & \boldsymbol L_2&\coloneq(L_{2,1},\dots, L_{2,D})^\top.
\end{align*}
\end{definition}

	Observe that according to Proposition \ref{pro:G-bilinear form} this reduction implies for the associated bilinear forms
	\begin{align*}
	\big\langle \mathcal Q_1(\x)\chi(\x),\big(\chi(\x)\big)^\top\big\rangle	=\big\langle \chi(\x),\big(\chi(\x)\big)^\top\mathcal Q_2(\x)\big\rangle.
	\end{align*}

\begin{pro}
Given two polynomials $\mathcal Q_1(\x),\mathcal Q_2(\x)\in\C[\x]$, with powers written as
\begin{align*}
(\mathcal Q_1(\x))^n&=\sum_{\q\in\Z_+^D}\mathcal Q^n_{1,\q}\x^\q,&(\mathcal Q_2(\x))^n&=\sum_{\q\in\Z_+^D}\mathcal Q^n_{2,\q}\x^\q
\end{align*} and a  $(\mathcal Q_1,\mathcal Q_2)$-invariant initial condition $G$ we find that
\begin{enumerate}
	\item The Lax semi-infinite matrices satisfy
		\begin{align}\label{eq:q1_q2_invariance_Lax}
		\mathcal Q_1(\boldsymbol L_1)=\mathcal Q_2(\boldsymbol L_2).
		\end{align}
	\item For  $n\in\{1,2,\dots\}$  the 	wave matrices satisfy
	\begin{align}\label{eq:reduction_wave}
\begin{aligned}
	\sum_{\q\in\Z_+^D}\mathcal Q^n_{1,\q}\frac{\partial W_1}{\partial t_{1,\q}}+\sum_{\q\in\Z_+^D}\mathcal Q^n_{2,\q}\frac{\partial W_1}{\partial t_{2,\q}}&=W_1\big(\mathcal Q_1(\boldsymbol \Lambda )\big)^n,\\
		\sum_{\q\in\Z_+^D}\mathcal Q^n_{1,\q}\frac{\partial W_2}{\partial t_{1,\q}}+\sum_{\q\in\Z_+^D}\mathcal Q^n_{2,\q}\frac{\partial W_2}{\partial t_{2,\q}}&=W_2\big(\mathcal Q_2(\boldsymbol \Lambda^\top)\big)^n,		
\end{aligned}
	\end{align}
	and  the Lax matrices fulfill the invariance conditions
		\begin{align}\label{eq:reduction_lax}
	\begin{aligned}
	\sum_{\q\in\Z_+^D}\mathcal Q^n_{1,\q}\frac{\partial \boldsymbol L_1}{\partial t_{1,\q}}+\sum_{\q\in\Z_+^D}\mathcal Q^n_{2,\q}\frac{\partial \boldsymbol L_1}{\partial t_{2,\q}}&=0,\\
				\sum_{\q\in\Z_+^D}\mathcal Q^n_{1,\q}\frac{\partial \boldsymbol L_2}{\partial t_{1,\q}}+\sum_{\q\in\Z_+^D}\mathcal Q^n_{2,\q}\frac{\partial \boldsymbol L_2}{\partial t_{2,\q}}&=0.		
	\end{aligned}
						\end{align}	
\end{enumerate}
\end{pro}
\begin{proof}
\begin{enumerate}
	\item Use \eqref{eq:central}, \eqref{eq:lax_wave} and \eqref{eq:invariance_polynomial} for \eqref{eq:q1_q2_invariance_Lax}.
	\item Observe that
	\begin{align*}
\sum_{\q\in\Z_+^D}	\mathcal Q^n_{1,\q}B_{1,\q}&=\Big(\big(\mathcal Q_1(\boldsymbol L_1 )\big)^n\Big)_+,&
\sum_{\q\in\Z_+^D}	\mathcal Q^n_{2,\q}B_{2,\q}&= \Big(\big(\mathcal Q_1(\boldsymbol L_2 )\big)^n\Big)_-
	\end{align*}
	and, consequently,
	\begin{align*}
\sum_{\q\in\Z_+^D}	\mathcal Q^n_{1,\q}B_{1,\q}+\sum_{\q\in\Z_+^D}	\mathcal Q^n_{2,\q}B_{2,\q}=\big(\mathcal Q_1(\boldsymbol L_1)\big)^n=\big(\mathcal Q_2(\boldsymbol L_2)\big)^n,
	\end{align*}
	and systems \eqref{eq:reduction_wave} and \eqref{eq:reduction_lax} follow from Proposition \eqref{pro:integrable}.
\end{enumerate}\end{proof}

An illustration of these type of the reductions is the case studied in previous sections involving multivariate orthogonal polynomials to a given generalized function $u\in(\C[x])'$ with $G=\langle u,\chi\chi^\top\rangle$. As we know this implies
 $\Lambda_a G=G(\Lambda_a)^\top$, $a\in\{1,\dots,D\}$, so that $   L_{1,a}=S_1\Lambda S_1^{-1}= \tilde S_2\Lambda^\top \tilde S_2^{-1}=L_{2,a}$, $a\in\{1,\dots,D\}$.   The Lax matrices $L_{1,a}$ and $L_{2,a}$ are lower and upper Hessenberg  block matrices, respectively.    Consequently, we have a  tridiagonal block matrix form; i.e., a Jacobi block matrix
 \begin{align*}
 \boldsymbol L_1=\boldsymbol L_2=\boldsymbol J.
 \end{align*}
   Moreover, these conditions imply an invariance property under the flows introduced, as we have that $G(t)= W_1^{(0)}(t_1-t_2)G$, i.e., there are only one type of flows, or in differential form
	\begin{align*}
	(\partial_{1,\q}+ \partial_{2,\q})W_1 &=W_1 \boldsymbol \Lambda^\q, &(\partial_{1,\q}+ \partial_{2,\q})W_2 &=W_2 (\boldsymbol\Lambda^\top)^\q,\\
	(\partial_{1,\q}+ \partial_{2,\q})L_{1,a} &=0, & (\partial_{1,\q}+ \partial_{2,\q})L_{2,a }&=0.
	\end{align*}
	
	\subsection{The linear spectral transformation for the multispectral 2D Toda hierarchy}
We extend the linear spectral transform for MVOPR to the more general framework of the multispectral Toda lattice just discussed. 	As a main result in Theorem \ref{the:linear_spectral_toda}  we get quasi-determinantal expressions for the transformed Baker function $(\hat \Psi_1)_{[k]}(t)$ and the quasi-tau matrices $\hat H_{[k]}(t)$.	
	
\begin{definition}
	Given two coprime polynomials $\mathcal Q_1(\x)$ and $\mathcal Q_2(\x)$, $\deg \mathcal Q_i=m_i$, we consider an initial condition $G$
and a perturbed one $\hat G$ such that
	\begin{align}\label{eq:GhatG}
	\hat G\mathcal Q_2(\boldsymbol\Lambda^\top)=\mathcal Q_1(\boldsymbol\Lambda)G.
	\end{align}
		We can achieve the perturbed semi-infinite matrix $\hat G$ in two steps, using an intermediate matrix $\check G$. First, we perform a Geronimus type transformation
		\begin{align}\label{eq:GhatG1}
		\check G \mathcal Q_2(\boldsymbol \Lambda^\top)=G
		\end{align}
	and second, a Christoffel type transformation
		\begin{align}\label{eq:GhatG2}
		\hat G=\mathcal Q_1(\boldsymbol{\Lambda})\check G.
		\end{align}
\end{definition}
\begin{pro}\label{pro:transformation_t}
Under the evolution prescribed in \eqref{eq:flows} if \eqref{eq:GhatG}, \eqref{eq:GhatG1} and \eqref{eq:GhatG2} we have
		\begin{align*}
		\hat G(t)\mathcal Q_2(\boldsymbol\Lambda^\top)&=\mathcal Q_1(\boldsymbol\Lambda)G(t), &
		\check G(t) \mathcal Q_2(\boldsymbol \Lambda^\top)&=G(t),&
		\hat G(t)&=\mathcal Q_1(\boldsymbol{\Lambda})\check G(t).
		\end{align*}
\end{pro}
\begin{proof}
	We just check the first as the others follow in an analogous manner:
	\begin{align*}
	\hat G(t)\mathcal Q_2(\boldsymbol\Lambda^\top)
	=& W_1^{(0)}(t_1)\hat G\big(W_2^{(0)}(t_2)\big)^{-\top}\mathcal Q_2(\boldsymbol\Lambda^\top)\\
		=& W_1^{(0)}(t_1)\hat G\mathcal Q_2(\boldsymbol\Lambda^\top)\big(W_2^{(0)}(t_2)\big)^{-\top}\\
				=& W_1^{(0)}(t_1)\mathcal Q_1(\boldsymbol\Lambda) G\big(W_2^{(0)}(t_2)\big)^{-\top}\\
	=&\mathcal Q_1(\boldsymbol\Lambda) G(t).
	\end{align*}
\end{proof}

In terms of bilinear forms \eqref{eq:GhatG1} reads
\begin{align*}
\big\langle\chi(\x),\big(\chi(\x)\big)^\top\mathcal Q_2(\x)\big\rangle\,\check{}=\big\langle\chi(\x),\big(\chi(\x)\big)^\top\big\rangle
\end{align*}
so that  assuming we can divide by polynomials inside these bilinear forms a  solution to \eqref{eq:GhatG1} \begin{align}\label{eq:R_toda}
\check G=\Big\langle\chi(\x),\frac{\big(\chi(\x)\big)^\top}{\mathcal Q_2(\x)}\Big\rangle+\big\langle v,\chi(\x)\big(\chi(\x)\big)^\top\big\rangle
\end{align}
where $v\in\big(\C[x]\big)'$ and $(\mathcal Q_2(\x))\subset\operatorname{Ker}(v)$.
In fact, a more general case will be
\begin{align*}
\check G=\Big\langle\chi(\x),\frac{\big(\chi(\x)\big)^\top}{\mathcal Q_2(\x)}\Big\rangle+\big\langle v,A\chi(\x)\big(\chi(\x)\big)^\top\big\rangle
\end{align*}
	where $A$ is a semi-infinite matrix with  rows having only a finite number of  non vanishing coefficients.
	
\begin{definition}\label{def:resolvents_linear_t}
		We   introduce the resolvents
	\begin{align*}
	\omega_1(t)&:=\hat S_1(t)\mathcal Q_1(\boldsymbol\Lambda)  \big(S_1(t)\big)^{-1}, & \omega_2(t)&:=\Big(S_2(t)\mathcal Q_2(\boldsymbol\Lambda) \big(\hat S_2(t)\big)^{-1}\Big)^\top.
	\end{align*}
\end{definition}
	\begin{pro}
		The resolvent matrices satisfy
		\begin{align}\label{eq:M-omega2_t}
		\hat H(t)\omega_2(t)=\omega_1(t)H(t).
		\end{align}
		The resolvents $\omega_1(t),\omega_2(t)$ are block banded matrices, having different from zero only the first $m_1$ block superdiagonals and the first $m_2$ block subdiagonals.
	\end{pro}
	\begin{proof}
	From the $LU$ factorization we get
	\begin{align*}
	\big(\hat S_1(t)\big)^{-1}\hat H(t)\big(\hat S_2(t)\big)^{-\top}\mathcal Q_2(\boldsymbol\Lambda^\top) =\mathcal Q_1(\boldsymbol\Lambda)  \big(S_1(t)\big)^{-1} H(t) \big(S_2(t)\big)^{-\top},
	\end{align*}
	so that
	\begin{align*}
	\hat H(t)\Big( S_2(t)\mathcal Q_2(\boldsymbol\Lambda) \big(\hat S_2(t)\big)^{-1}\Big)^\top=\hat S_1(t)\mathcal Q_1(\boldsymbol\Lambda)  \big(S_1(t)\big)^{-1} H(t).
	\end{align*}
		\end{proof}
		
		In this more general scenario Proposition \ref{pro:superdiagonals2} still holds for these new resolvents, not connected in principle with any linear functional.
		We have
		\begin{pro}[Connection formulas]
			We have
			\begin{align*}
			\omega_1(t) P_1(t,\x)&=\mathcal Q_1(\x) \hat P_1(t,\x),\\
			\big(\omega_2(t)\big)^\top \hat P_2(t,\x) &=\mathcal Q_2(\x) P_2(t,\x).
			\end{align*}
		\end{pro}
		\begin{definition}
			We introduce the semi-infinite matrix
			\begin{align}\label{eq:def_R}
			R(t)\coloneq S_1(t)\check G(t)
			\end{align}
		\end{definition}
\begin{pro}
			The matrix $R(t)$ can be expressed as follows
		\begin{align}\label{eq:R_explicit}
		R(t)=\Big\langle P_1(t,\x),\frac{\big(\chi(\x)\big)^\top}{\mathcal Q_2(\x)}\Big\rangle+\big\langle v,P_1(t,\x)\big(\chi(\x)\big)^\top\big\rangle.
		\end{align}
\end{pro}
		\begin{proof}
				Recall \eqref{eq:R_toda} and \eqref{eq:def_R}.
		\end{proof}

		\begin{pro}
			We have the following relations
			\begin{align*}
			(\omega_1(t)R(t))_{[k],[l]}&=0, & l=&0,1,\dots,k-1\\
			(\omega_1(t)R(t))_{[k],[k]}&=\hat H_{[k]}(t)
			\end{align*}
		\end{pro}
		\begin{proof}
			Just follow the next chain of equalities
			\begin{align}
			\omega_1(t) R(t)&= \hat S_1(t)\mathcal Q_1(\boldsymbol{\Lambda})(S_1(t))^{-1} S_1(t)\check G(t)\notag\\
			&=\hat S_1(t)\mathcal Q_1(\boldsymbol \Lambda) \check G(t)\notag\\
			&=\hat S_1(t)\hat G(t) & \text{ from \eqref{pro:transformation_t}}\notag\\
			&=\hat H(t) (\hat S_2(t))^{-\top}\label{eq:important_relation}
			\end{align}
			and the matrix $\omega_1R$ is an upper triangular block matrix with $\hat H$ as its block diagonal.
		\end{proof}
	Proceeding as we did for \eqref{eq:ERL1} and \eqref{eq:ERL}	we can  deduce analogous equations in this new context.
	For $k<m_2$ we can write
		\begin{multline*}%\label{eq:ERL1}
		\big((\omega_1)_{[k],[0]}(t), \dots, (\omega_1)_{[k],[k+m_1-1]}(t)\big)=\\\begin{multlined}
		-\big(\mathcal Q_1(\boldsymbol \Lambda)\big)_{[k],[k+m_1]}\big(R_{[k+m_1],[0]}(t) ,\dots , R_{[k+m_1],[k-1]}(t),\Psi_{1,[k+m_1]}(t,\boldsymbol p_1),\dots , \Psi_{[k+m_1]}(t,\boldsymbol p_{r_{1|k,m_1}})\big)\\
		\times\PARENS{
			\begin{matrix}
			R_{[k],[0]}(t) &\dots & R_{[0],[k-1]}(t)&	\Psi_{1,[k]}(t,\boldsymbol p_1) & \dots & \Psi_{1,[k]}(t,\boldsymbol p_{r_{1|k,m_1}})\\
			\vdots & &\vdots&\vdots& &\vdots\\
			R_{[k+m_1-1],[0]}(t) &\dots & R_{[k+m_1-1],[k-1]}(t)&	\Psi_{1,[k+m_1-1]}(t,\boldsymbol p_1) & \dots & \Psi_{1,[k+m_1-1]}(t,\boldsymbol p_{r_{1|k,m_1}})
			\end{matrix}}^{-1},
		\end{multlined}
		\end{multline*}
		while	for $k\geq m_2$
		\begin{multline*}%\label{eq:ERL}
		\big((\omega_1)_{[k],[k-m_2]}(t), \dots, (\omega_1)_{[k],[k+m_1-1]}(t)\big)=\\\begin{multlined}
		-\big(\mathcal Q_1(\boldsymbol \Lambda)\big)_{[k],[k+m_1]}\big(R_{[k+m_1],\boldsymbol \beta_1} (t),\dots , R_{[k+m_1],\boldsymbol \beta_{r_{2|k,m_2}}}(t),\Psi_{1,[k+m_1]}(t,\boldsymbol p_1),\dots , \Psi_{1,[k+m_1]}(t,\boldsymbol p_{r_{1|k,m_1}})\big)\\
		\times\PARENS{
			\begin{matrix}
			R_{[k-m_2],\boldsymbol \beta_1} (t)&\dots & R_{[k-m_2],\boldsymbol \beta_{r_{2|k,m_2}}}(t)&	\Psi_{1,[k-m_2]}(t,\boldsymbol p_1) & \dots & \Psi_{1,[k-m_2]}(t,\boldsymbol p_{r_{1|k,m_1}})\\
			\vdots & &\vdots&\vdots& &\vdots\\
			R_{[k+m_1-1],\boldsymbol \beta_1} (t)&\dots & R_{[k+m_1-1],\boldsymbol \beta_{r_{2|k,m_2}}}(t)&	\Psi_{1,[k+m_1-1]}(t,\boldsymbol p_1) & \dots & \Psi_{1,[k+m_1-1]}(t,\boldsymbol p_{r_{1|k,m_1}})
			\end{matrix}}^{-1},
		\end{multlined}
		\end{multline*}
		We also have
			\begin{align*}
		(\omega_1(t))_{[k],[k+m_1]}&=\big(\mathcal Q_1(\boldsymbol \Lambda)\big)_{[k],[k+m_1]}.
		\end{align*}
		
		Then, we extend  Definitions \ref{def:linear_spectral_1} and \ref{def:linear_spectral_2} to this new scenario,  and find a version of Theorem \ref{the:linear_spectral}  in terms of  the Baker functions
\begin{theorem}[Christoffel--Geronimus--Uvarov formula for multispectral Toda hierarchy]\label{the:linear_spectral_toda}
	A linear spectral transformation, as in  \eqref{eq:GhatG}, for the multispectral Toda hierarchy has the following effects on the Baker function $\Psi_{1,[k]}(t)$ and the quasi-tau matrices $H_{[k]}(t)$.
	Given a poised set $\mathcal S_k$, of multi-indices and nodes, we have a perturbed Baker function
	\begin{multline*}
	\hat \Psi_{1,[k]}(t,\x)=\frac{\big(\mathcal Q_1(\boldsymbol \Lambda)\big)_{[k],[k+m_1]}}{\mathcal Q_1(\x)}\\\times \Theta_* \PARENS{
		\begin{matrix}
		R_{[0],[0] }(t)&\dots & R_{[0],[k-1]}(t)& \Psi_{1,[0]}(t,\boldsymbol p_1) & \dots & \Psi_{1,[0]}(t,\boldsymbol p_{r_{1|k,m_1}})&\Psi_{1,[0]}(t,\x)\\
		\vdots & &\vdots&\vdots& &\vdots&\vdots\\
		R_{[k+m_1],[k-1]}(t) &\dots & R_{[k+m_1],[k-1]}(t)&	\Psi_{1,[k+m_1]}(t,\boldsymbol p_1) & \dots & \Psi_{1,[k+m_1]}(t,\boldsymbol p_{r_{1|k,m_1}})&
		\Psi_{1,[k+m_1]}(t,\x)
		\end{matrix}},
	\end{multline*}
	and a perturbed quasi-tau matrix
	\begin{multline*}
	\hat H_{[k]}(t)=\big(\mathcal Q_1(\boldsymbol \Lambda)\big)_{[k],[k+m_1]}\\\times \Theta_* \PARENS{
		\begin{matrix}
		R_{[0],[0] }(t)&\dots & R_{[0],[k-1]}(t)& \Psi_{1,[0]}(t,\boldsymbol p_1) & \dots & \Psi_{1,[0]}(t,\boldsymbol p_{r_{1|k,m_1}})&R_{[0],[k]}(t)\\
		\vdots & &\vdots&\vdots& &\vdots&\vdots\\
		R_{[k+m_1],[k-1]}(t) &\dots & R_{[k+m_1],[k-1]}(t)&	\Psi_{1,[k+m_1]}(t,\boldsymbol p_1) & \dots & \Psi_{1,[k+m_1]}(t,\boldsymbol p_{r_{1|k,m_1}})&
		R_{[k+m_1],[k]}(t)
		\end{matrix}}.
	\end{multline*}
	When $k\geq m_2$ we have the shorter alternative expressions
	\begin{multline*}
	\hat \Psi_{1,[k]}(t,\x)=\frac{\big(\mathcal Q_1(\boldsymbol \Lambda)\big)_{[k],[k+m_1]}}{\mathcal Q_1(\x)}\\\times\Theta_* \PARENS{
		\begin{matrix}
		R_{[k-m_2],\boldsymbol \beta_1} (t)&\dots & R_{[k-m_2],\boldsymbol \beta_{r_{2|k,m_2}}}(t)&	\Psi_{1,[k-m_2]}(t,\boldsymbol p_1) & \dots & \Psi_{1,[k-m_2]}(t,\boldsymbol p_{r_{1|k,m_1}})
		&\Psi_{1,[k-m_2]}(t,\x)\\
		\vdots & &\vdots&\vdots& &\vdots&\vdots\\
		R_{[k+m_1],\boldsymbol \beta_1} (t)&\dots & R_{[k+m_1],\boldsymbol \beta_{r_{2|k,m_2}}}(t)&	\Psi_{1,[k+m_1]}(t,\boldsymbol p_1) & \dots & \Psi_{1,[k+m_1]}(t,\boldsymbol p_{r_{1|k,m_1}})&
		\Psi_{1,[k+m_1]}(t,\x)
		\end{matrix}},
	\end{multline*}
		\begin{multline*}
		\hat H_{[k]}(t)=\big(\mathcal Q_1(\boldsymbol \Lambda)\big)_{[k],[k+m_1]}\\\times \Theta_* \PARENS{
			\begin{matrix}
			R_{[k-m_2],\boldsymbol \beta_1} (t)&\dots & R_{[k-m_2],\boldsymbol \beta_{r_{2|k,m_2}}}(t)& \Psi_{1,[k-m_2]}(t,\boldsymbol p_1) & \dots & \Psi_{1,[k-m_2]}(t,\boldsymbol p_{r_{1|k,m_1}})&R_{[k-m_2],[k]}(t)\\
			\vdots & &\vdots&\vdots& &\vdots&\vdots\\
					R_{[k+m_1],\boldsymbol \beta_1} (t)&\dots & R_{[k+m_1],\boldsymbol \beta_{r_{2|k,m_2}}}(t)&	\Psi_{1,[k+m_1]}(t,\boldsymbol p_1) & \dots & \Psi_{1,[k+m_1]}(t,\boldsymbol p_{r_{1|k,m_1}})&
			R_{[k+m_1],[k]}(t)
			\end{matrix}}.
		\end{multline*}
	and
	\begin{multline*}
	\hat H_{[k]}(t)\Big(\big(\mathcal Q_2(\boldsymbol \Lambda)\big)_{[k-m_2],[k]}\Big)^\top=\big(\mathcal Q_1(\boldsymbol \Lambda)\big)_{[k],[k+m_1]}\\\times\Theta_* \PARENS{
		\begin{matrix}
		R_{[k-m_2],\boldsymbol \beta_1} (t)&\dots & R_{[k-m_2],\boldsymbol \beta_{r_{2|k,m_2}}}(t)&	\Psi_{1,[k-m_2]}(t,\boldsymbol p_1) & \dots & \Psi_{1,[k-m_2]}(t,\boldsymbol p_{r_{1|k,m_1}})
		&H_{[k-m_2]}(t)\\
		R_{[k-m_2],\boldsymbol \beta_1} (t)&\dots & R_{[k-m_2],\boldsymbol \beta_{r_{2|k,m_2}}}(t)&	\Psi_{1,[k-m_2]}(t,\boldsymbol p_1) & \dots & \Psi_{1,[k-m_2]}(t,\boldsymbol p_{r_{1|k,m_1}})
		&0\\
		\vdots & &\vdots&\vdots& &\vdots&\vdots\\
		R_{[k+m_1],\boldsymbol \beta_1} (t)&\dots & R_{[k+m_1],\boldsymbol \beta_{r_{2|k,m_2}}}(t)&	\Psi_{1,[k+m_1]}(t,\boldsymbol p_1) & \dots & \Psi_{1,[k+m_1]}(t,\boldsymbol p_{r_{1|k,m_1}})&
		0
		\end{matrix}}.
	\end{multline*}
\end{theorem}		
Regarding the Baker function $\Psi_2$ and its behavior under a general linear spectral transformation, using \eqref{eq:Baker2}, we have for each component
\begin{align*}
\hat  \Psi_{2,[k]}(t,\z) & =\langle\hat \Psi_{1,[k]}(t,\x),\mathcal C(\z,\x)\rangle,
\end{align*}
and consequently Theorem \ref{the:linear_spectral_toda} provides quasi-determinantal expression for $\hat\Psi_{2,[k]}$ performing the following replacements
\begin{align*}
\Psi_{1,[l]}(t,\x)&\rightarrow \Big\langle\frac{\Psi_{1,[l]}(t,\x)}{\mathcal Q_1(\x)} ,\mathcal C(\z,\x)\Big\rangle, &
l&\in\{k-m_2,\dots,k+m_1\}.
\end{align*}
Alternative expressions are achieved if  the  relation \eqref{eq:important_relation} is recalled. Indeed,  it implies
\begin{align}\label{eq:hat_Psi_2}
\hat \Psi_2(t,\z)=\omega_1R\big(W^{(0)}_2(t_2)\big)^\top\chi^*(\z).
\end{align}
Then, using \eqref{eq:R_explicit} we conclude that
 the replacements to perform in
Theorem \ref{the:linear_spectral_toda} to find a quasi-determinantal expression for $\hat\Psi_{2,[k]}$ are
\begin{align*}
\Psi_{1,[l]}(t,\x)&\rightarrow \Big\langle P_{1,[l]}(t,\x),\Exp{t_2(\x)}\frac{\mathcal C(\z,\x)}{\mathcal Q_2(\x)}\Big\rangle+\big\langle v,\Exp{t_2(\x)}P_{1,[l]}(t,\x)\mathcal C(\z,\x)\big\rangle, &
l&\in\{k-m_2,\dots,k+m_1\}.
\end{align*}

In this general setting $G$ is not restricted by a Hankel type constraint, thus given a polynomial $\mathcal Q(\x)\in\R[\x]$ we have
\begin{align*}
G\mathcal Q(\boldsymbol{\Lambda}^\top)\not =\mathcal Q(\boldsymbol{\Lambda})G.
\end{align*}
For example, instead of \eqref{eq:GhatG} we may have considered
		\begin{align*}
	\mathcal Q_2(\boldsymbol\Lambda)	\hat G=G\mathcal Q_1(\boldsymbol\Lambda^\top).
		\end{align*}
In this case a  transposition formally gives
		\begin{align*}
			\hat G^\top\mathcal Q_2(\boldsymbol\Lambda^\top)=\mathcal Q_1(\boldsymbol\Lambda)G^\top,
		\end{align*}
which can be gotten from \eqref{eq:GhatG} by the replacement $G\mapsto G^\top$ and $\hat G\mapsto \hat G^\top$; i.e., at the level of the Gauss--Borel factorization \eqref{eq:general_LU}
\begin{align*}
S_1&\mapsto S_2,& H&\mapsto H^\top, & S_2&\mapsto S_1,\\
\hat S_1&\mapsto \hat S_2,& \hat H&\mapsto \hat H^\top, & \hat S_2&\mapsto \hat S_1.
\end{align*} 		
Thus, previous formul{\ae} hold by replacing $P_1$ by $P_2$ and transposing the matrices $H_{[k]}$ and $\hat H_{[k]}$.	

A quite  general transformation, which we will not explore in this paper, corresponds to
		\begin{align*}
		\mathcal Q_2^L(\boldsymbol\Lambda)	\hat G	\mathcal Q_2^R(\boldsymbol\Lambda^\top)	=	\mathcal Q_1^R(\boldsymbol\Lambda)G\mathcal Q_1^L(\boldsymbol\Lambda^\top),
		\end{align*}
for polynomials $\mathcal Q_1^L(\x),\mathcal Q^R_1(\x),\mathcal Q_2^L(\x),\mathcal Q_2^R(\x)\in\R[\x]$.
This transformation is preserved by the integrable flows introduced above; i.e.,
	\begin{align*}
	\mathcal Q_2^L(\boldsymbol\Lambda)	\hat G(t)	\mathcal Q_2^R(\boldsymbol\Lambda^\top)	=	\mathcal Q_1^R(\boldsymbol\Lambda)G(t)\mathcal Q_1^L(\boldsymbol\Lambda^\top).
	\end{align*}
	Notice that this transformation for a multi-Hankel reduction $\Lambda_aG=G(\Lambda_a)^\top$, $a\in\{1,\dots,D\}$, is just the one  considered in previous sections.

 \subsection{Generalized bilinear equations and linear spectral transformations}

 We are ready to show that the Baker functions at different times and their linear spectral transforms satisfy a bilinear equation as in the KP theory, see \cite{Date1981operator,Date1982transformation,Date1983transformation}. In the standard formulation \cite{Date1981operator,Date1982transformation,Date1983transformation} discrete times appeared in the bilinear equation, which in this case are identified, see for example \cite{doliwa2000transformations}, with the linear spectral transformations. To deduce the bilinear equations we use a similar method as  in \cite{Adler1999vertex,manas2009multicomponent,manas2009multicomponent1}.

 We begin with the following observation
 \begin{pro}\label{proposition:bilinear_wave}
 	Wave matrices $W_i(t)$, $i\in\{1,2\}$  and their linear spectral transformed wave  matrices  $\hat W_i(t')$,  $i\in\{1,2\}$, according to the coprime  polynomials $\mathcal Q_1(\x),\mathcal Q_2(\x)\in\C[\x]$,   fulfill
 	\begin{align*}
 	\hat W_1(t')\mathcal Q_1(\boldsymbol \Lambda)\big(W_1(t)\big)^{-1}=\hat W_2(t')\mathcal Q_2(\boldsymbol \Lambda^\top)\big(W_2(t)\big)^{-1}.
\\
 	\end{align*}
 \end{pro}
 \begin{proof}
 	We have
 	\begin{align*}
 	G=&\big(W_1(t)\big)^{-1}W_2(t), &   \hat G=&\big( \hat W_1(t')\big)^{-1}\hat W_2(t').
 	\end{align*}
 	Hence, using \eqref{eq:GhatG} we deduce
 	\begin{align*}
 	\mathcal Q_1(\boldsymbol \Lambda)\big(W_1(t)\big)^{-1}W_2(t)=\big( \hat W_1(t')\big)^{-1}\hat W_2(t')\mathcal Q_2(\boldsymbol \Lambda^\top).
 	\end{align*}
 \end{proof}
 Now, we need
 \begin{lemma}\label{lemma:U_V_integrals}
 	Given two semi-infinite matrices $U$ and $V$ we have
 	\begin{align*}
 	UV=&\frac{1}{(2\pi \operatorname{i})^D}\oint_{\mathbb T^D(\boldsymbol r)} U\chi(\boldsymbol z) \big(V^T\chi^*(\boldsymbol z)\big)^\top\d z_1\cdots\d z_D
 	=\frac{1}{(2\pi \operatorname{i})^D}\oint_{\mathbb T^D(\boldsymbol r)} U\chi^*(\boldsymbol z) \big(V^T\chi(\boldsymbol z)\big)^\top\d z_1\cdots\d z_D.
 	\end{align*}
 \end{lemma}
 \begin{proof}
 		Observe that
 		\begin{align*}
 		\chi(\chi^*)^\top&=\PARENS{\begin{matrix}
 			Z_{[0],[0]} &Z_{[0],[1]}&\dots\\
 			Z_{[1],[0]} &Z_{[1],[1]}&\dots\\
 			\vdots & \vdots&
 			\end{matrix}}, & Z_{[k],[\ell]}&\coloneq \frac{1}{z_1\cdots z_D}\PARENS{\begin{matrix}
 			\z^{\kk_1-\ele_1} &  \z^{\kk_1-\ele_2}&\dots &   \z^{\kk_1-\ele_{|[\ell]|}}\\
 			\z^{\kk_2-\ele_1} &  \z^{\kk_2-\ele_2}&\dots &   \z^{\kk_2-\ele_{|[\ell]|}}\\
 			\vdots & \vdots &&\vdots\\
 			\z^{\kk_{|[k]|}-\ele_1} &  \z^{\kk_{|[k]|}-\ele_2}&\dots &   \z^{\kk_{|[k]|}-\ele_{|[\ell]|}}
 			\end{matrix}}.
 		\end{align*}
 		If we now integrate in the polydisk distinguished border  $\T^{D}(\boldsymbol r)$ using the Fubini theorem we factor each integral
 		in a product of $D$ factors, where the $i$-th factor is an integral over $z_i$ on the circle centered at origin of radius $r_i$. This is zero unless the integrand is $z_i^{-1}$ which occurs only in the principal diagonal. Consequently, we have
 			\begin{align*}
 			\oint_{\mathbb T^D(\boldsymbol r)} \chi(\boldsymbol z)\chi^*(\boldsymbol z)^\top\d z_1\cdots\d z_D=
 			\oint_{\T^D(\boldsymbol r)} \chi^*(\boldsymbol z)\chi(\boldsymbol z)^\top\d z_1\cdots\d z_D=(2\pi \operatorname{i})^D\mathbb I,
 			\end{align*}
 			and the result follows.
 \end{proof}
 We notice that $\Psi_1$ and $\Psi_2^*$ lead to the computation of finite sums, i.e., polynomials, but $\Psi_1^*$ and $\Psi_2$ involve Laurent series. We will denote by $\Ds_{2,\q}(t)$ and $\Ds_{1,\q}^*(t)$ the domains of convergence of $\Psi_{2,\q}(t,\z) $ and  $\Psi_{1,\q}^*(t,\z)$, respectively.
 Recall that  these domains are  Reinhardt domains; i.e., if $\mathscr D\subset\C^D$ is the domain of convergence  then for any $\boldsymbol c=(c_1,\dots,c_D)^\top\in\mathscr D$ we have that $\T^D(|c_1|,\dots,|c_D|)\subset\Ds$.

 \begin{theorem}[Generalized bilinear equations]
 	\label{proposition:Baker_bilinear_equation}
 	For any pair of times $t$ and $t'$, points  $\boldsymbol r_1\in{\Ds}^*_{1,\q}(t)$ and $\boldsymbol r_2\in\hat{\Ds}_{2,\q}(t')$ in the respective Reinhardt domains  and $D$-dimensional tori  $\T^D(\boldsymbol r_1)$ and $\T^D(\boldsymbol r_2)$, and multi-indices $\q,\q'\in\Z_+^D$,   the
 	Baker and adjoint Baker functions and their  linear spectral transformations satisfy the following bilinear identity
 	\begin{align*}
 	\oint_{\T^D(\boldsymbol r_1)}\hat \Psi_{1,\q'}(t',\boldsymbol{z})\Psi_{1,\q}^*(t,\boldsymbol{z}) \mathcal Q_1(\z)\d z_1\cdots\d z_D=
 	\oint_{\T^D(\boldsymbol r_2)}\hat \Psi_{2,\q'}(t',\boldsymbol{z})\Psi_{2,\q}^*(t,\boldsymbol{z}) \mathcal Q_2(\z) \d z_1\cdots\d z_D.
 	\end{align*}
 \end{theorem}
 \begin{proof}
 	From Definition \ref{def:integrable} and Lemma  \ref{lemma:U_V_integrals},  choosing $U=\hat W_1(t')\mathcal Q_1\boldsymbol\Lambda)$ and $V=\big(W_1(t)\big)^{-1}$ we get
 	\begin{align*}
 	\hat W_1(t')\mathcal Q_1(\boldsymbol\Lambda)\big(W_1(t)\big)^{-1}&=	\frac{1}{(2\pi \operatorname{i})^D}\oint_{\T^D(\boldsymbol r_1)}\hat \Psi_{1}(t',\boldsymbol{z})\Psi_{1}^*(t,\boldsymbol{z}) \mathcal Q_1(\z)\d z_1\cdots\d z_D,
 	\end{align*}
 	and choosing $U=\hat W_2(t')$ and $V=\mathcal Q_2(\boldsymbol\Lambda^\top )\big(W_2(t)\big)^{-1}$ we get
 	\begin{align*}
 	\hat W_2(t')\mathcal Q_2(\boldsymbol\Lambda^\top)\big(W_2(t)\big)^{-1}&=	\frac{1}{(2\pi \operatorname{i})^D}\oint_{\T^D(\boldsymbol r_2)}\hat \Psi_{2}(t',\boldsymbol{z})\Psi_{2}^*(t,\boldsymbol{z}) \mathcal Q_2(\z)\d z_1\cdots\d z_D.
 	\end{align*}
 	Then,  Proposition \ref{proposition:bilinear_wave} implies the result.
 \end{proof}
		
		\appendix
		\section{Uvarov perturbations}

 Uvarov considered in \S 2 of 	 \cite{Uvarov1969connection} the addition of a finite number of masses, Dirac deltas, to a given measure in the OPRL situation.
	In this appendix we discuss some elements of the multivariate extension of this construction. There is an immediate extension when one considers masses, see \cite{Teresa0}. A bit more involved case is to consider higher multipoles, i.e., derivatives of the Dirac distributions. 
	In \cite{Teresa0} a Sobolev type modification was considered, see for example equation (2.16) in that paper, but this can not be modeled by a perturbation $\hat u=u+v$ of a  linear functional $u$ (or measure in that case, $u=\d\mu$).
	 All this can be considered as a 0-dimensional additive perturbation. However, more interesting and less trivial extensions are to consider higher dimensional additive perturbations. For example, 1D-Uvarov perturbations, i.e., additive perturbations supported over curves. For the 0D-Uvarov perturbations, as was found in \cite{Uvarov1969connection}, one needs to solve a linear system constructed in terms of the non-perturbed Christoffel--Darboux kernel evaluated at the 0D discrete support of the perturbation. We will se that for the 1D scenario the linear system of the 0D case is replaced by a Fredholm integral equation evaluated at the 1D support of the perturbation.. 
	
	Our approach to the problem is based on a simple relation among perturbed and non perturbed MVOPR which involves the non perturbed Christoffel--Darboux kernel.
	Let us consider a generalized function $u\in\big(\C[\x]\big)'$ such that is quasidefinite and consider an additive perturbation of it given by another generalized function $v\in\big(\C[\x]\big)'$
				\begin{align*}
		\hat u=u+v.
		\end{align*}
		\begin{pro}
			For an additive perturbation we have
			\begin{align}\label{eq:additive}
			\hat P_{[n]}(\x)
			&=P_{[n]}(\x)	-\left\langle v, \hat P_{[n]}(\y)K_{n-1}(\y,\x)\right\rangle,\\
				\hat H_{[n]}				&=	H_{[n]}+\left\langle v, \hat P_{[n]}(\x)\big(P_{[n]}(\x)\big)^\top\right\rangle. \label{eq:additive2}
				\end{align}

		\end{pro}
\begin{proof}
			From \eqref{orth} and \eqref{H} we deduce
		\begin{align*}
		\left\langle \hat u, \hat P_{[n]}(\x)\Big(P_{[m]}(\x)\Big)^\top\right\rangle&=0, & m\in\{0,1,\dots,n-1\},\\
			\left\langle \hat u, \hat P_{[n]}(\x)\Big(P_{[n]}(\x)\Big)^\top\right\rangle&=\hat H_{[n]}, 
		\end{align*}
	and, consequently,	
		\begin{align*}
			\left\langle u, \hat P_{[n]}(\x)\big(P_{[m]}(\x)\big)^\top\right\rangle&=	-\left\langle v, \hat P_{[n]}(\x)\big(P_{[m]}(\x)\big)^\top\right\rangle,  & m\in\{0,1,\dots,n-1\}.
		\end{align*}
Thus, in terms of the Christoffel--Darboux kernel, see Definition \ref{definition:CD}
\begin{align*}
\left\langle u, \hat P_{[n]}(\x) K_{n-1}(\x,\y)\right\rangle=-\left\langle v, \hat P_{[n]}(\x) K_{n-1}(\x,\y)\right\rangle.
\end{align*}
Observe that $\hat P_{[n]}(\x)-P_{[n]}(\x)$ is a multivariate polynomial of degree $n-1$ and, according to \eqref{eq:projection}, we conclude
\begin{align*}
\hat P_{[n]}(\x)-P_{[n]}(\x)&=\left\langle u, \big (\hat P_{[n]}(\y)-P_{[n]}(\y)\big)K_{n-1}(\y,\x)\right\rangle\\
&=\left\langle u, \hat P_{[n]}(\y)K_{n-1}(\y,\x)\right\rangle\\
&=	-\left\langle v, \hat P_{[n]}(\y)K_{n-1}(\y,\x)\right\rangle.
\end{align*}
Finally, we have
		\begin{align*}
		\hat H_{[n]}	&=	\left\langle u, \hat P_{[n]}(\x)\big(P_{[n]}(\x)\big)^\top\right\rangle+\left\langle v, \hat P_{[n]}(\x)\big(P_{[n]}(\x)\big)^\top\right\rangle\\
	&=	H_{[n]}+\left\langle v, \hat P_{[n]}(\x)\big(P_{[n]}(\x)\big)^\top\right\rangle.
		\end{align*}

\end{proof}
\subsection{0D-Uvarov multipolar perturbations. Masses (or charges) and dipoles}
Here we discuss the more general additive perturbation with finite discrete support. As we have a finite number of points for the support  we say that is a 0 dimensional perturbation.
Let us proceed and  consider a set of couples $S=\big\{\x_i,\boldsymbol \beta_i\big\}_{i=1}^q\subset \R^D\times \mathbb Z_+^D$ and define the associated generalized function
		\begin{align*}
		v_S:=\sum_{i=1}^q \sum_{\q \leq \boldsymbol\beta_i}
		\frac{(-1)^{|\q|}}{\q!}
		\xi_{i,\q}\delta^{(\q)}(\x-\x_i).
		\end{align*}
Here the sum over multi-indices extend to all those  multi-indices below a given one.  The Dirac delta distribution and its derivatives are given by 	
\begin{align*}
\langle\delta^{(\q)}(\x-\x_j),P(\x)\rangle&:=(-1)^{|\q|}\frac{\partial^{|\q|}P}{\partial \x^\q}\Big|_{\x=\x_i}, &\forall P(\x)\in\C[\x],
\end{align*}
and we have used the lexicographic order for the set of integer multi-indices.  Observe  that this is the more general distribution with support on 
$\{\x_j\}_{j=1}^q=\operatorname{supp}(v_S)$. From a physical point of view, the delta functions can be understood as  point masses. For higher order derivatives, we have an electromagnetic interpretation, for zero order derivatives we have point charges, and first order derivatives could be understood as dipoles, and in general for $j$-th order derivatives we are dealing with  $2^j$-multipoles (for $j=2$ we have  quadropoles,  for  $j=3$ we have  octopoles, and so on and so forth).\footnote{Given a charge density $\rho(\x)$ we get the multipolar expansion by writing $\rho(\x)=\int \delta(\x-\x')\rho(\x') \d\x'$ and recalling $\delta(\x-\x')=\sum\limits_{\q\in\Z_+^D}\frac{(-1)^{|\q|}}{\q!}(\x')^\q\delta^{(\q)}(\x)$. Then, the charge density is expressed as $\rho(\x)=\sum\limits_{\q\in\Z_+^D} \frac{(-1)^{|\q|}}{\q!}\xi_\q \delta^{(\q)}(\x)$ with multipole moments the \emph{tensors} $\xi_\q=\int \rho(\x)\x^\q\d\x$.}

\begin{definition}
\begin{enumerate}
	\item 	Given a multi-index $\boldsymbol\beta\in\mathbb Z_+^D$ with   $|\boldsymbol\beta|=k\in\mathbb Z_+$ we have a corresponding  lexicographic ordered set of multi-indices
\begin{align*}
\big\{\q^{(0)}_1,\q^{(1)}_1,\dots,\q^{(1)}_D,\q^{(2)}_1,\dots,\q^{(k)}_1,\dots,\q^{(k)}_{m}=\boldsymbol \beta\big\},
\end{align*} 
where $m$ denotes the position of the multi-index $\boldsymbol\beta$ among those of length $k$, see \S \ref{MV}. 
\item Given a couple $(\x,\boldsymbol{\beta})\in\mathbb R^D\times \mathbb Z_+^D$  and a   polynomial $P\in\mathbb C[\x]$ we define the jet
\begin{align*}
\mathcal J_P^{\boldsymbol\beta}(\x)=\begin{bmatrix}
\dfrac{1}{(\q^{(0)}_1)!}\dfrac{\partial ^{|\q^{(0)}_1|}P}{\partial \x^{\q^{(0)}_1}}(\x),
\dfrac{1}{(\q^{(1)}_1)!}\dfrac{\partial ^{|\q^{(1)}_1|}P}{\partial \x^{\q^{(1)}_1}}(\x),\dots,\dfrac{1}{(\q^{(k)}_1)!}\dfrac{\partial ^{|\q^{(k)}_1|}P}{\partial \x^{\q^{(k)}_1}}(\x),\dots,\dfrac{1}{\boldsymbol \beta!}\dfrac{\partial ^{|\boldsymbol \beta|}P}{\partial \x^{\boldsymbol \beta}}(\x)
	\end{bmatrix}.
\end{align*}
This a row vector with $N_{k-1}+m$ components. Recall that the  dimension of the linear space of multivariate polynomials of degree less or equal  to $k$ is
$N_{k}=\binom{D+k}{D}$.
\item Given the set, we define a matrix collecting the corresponding jets at each puncture $\x_i$
\begin{align*}
\mathcal J_P(S)=\begin{bmatrix}
\mathcal J_P^{\boldsymbol\beta_1}(\x_1),\dots,\mathcal J_P^{\boldsymbol\beta_q}(\x_q)
\end{bmatrix}.
\end{align*}
This a row vector with $N_S:=\sum\limits_{i=1}^q(N_{k_i-1}+m_i)$ components.
\item We consider the block antidiagonal matrices $\xi^{(i)}\in\C^{(N_{k_i-1}+m_i)\times (N_{k_i-1}+m_i)}$, $i\in\{1,\dots,q\}$,  with coefficients
\begin{align*}
(\xi^{(i)})_
{\q,\boldsymbol\beta}:=\xi_{i,\q+\boldsymbol{\beta}}
\end{align*}
and the matrix
\begin{align*}
\varXi:=\diag(\xi^{(1)},\dots,\xi^{(q)})\in\mathbb C^{N_S\times N_S}
\end{align*}
\item The Christoffel--Darboux jet is given in terms of product of truncations
\begin{align*}
\mathcal K_{n-1}(S)= \big(\mathcal J_{P^{[n]}}(S)\big)^\top (H^{[n]})^{-1} \mathcal J_{P^{[n]}}(S)\in\C^{N_S\times N_S}.
\end{align*}
Notice that the truncation $P^{[n]}(\x)$ is a vector of polynomials and, therefore,  $\mathcal J_{P^{[n]}}(S)$ is a $N_{n-1}\times N_S$ complex  matrix.
\end{enumerate}
\end{definition}
\begin{theorem}[0D-Uvarov multipolar perturbation]\label{teo:discrete uvarov}
Given a discrete additive perturbation of the form
	\begin{align*}
	\hat u=u+\sum_{i=1}^q \sum_{\q \leq \boldsymbol\beta_i}
	\frac{(-1)^{|\q|}}{\q!}
	\xi_{i,\q}\delta^{(\q)}(\x-\x_i),
	\end{align*}
the new MVOPR and quasi-tau matrices are given by the following quasi-determinantal expressions
	\begin{align*}
	\hat P_{[n]}(\x)&=\Theta_*\left(
	\begin{array}{c|c}
	I_{N_S}+\varXi \mathcal K_{n-1}(S) & \varXi\big( \mathcal J_{K_{n-1}(\cdot,\x)}(S)\big)^\top\\
	\hline \mathcal J_{ P_{[n]}}(S) &P_{[n]}(\x)
	\end{array}
	\right),\\
		\hat H_{[n]}&=\Theta_*\left(
		\begin{array}{c|c}
		I_{N_S}+\varXi \mathcal K_{n-1}(S) & -\varXi\big( \mathcal J_{ P_{[n]}}(S)\big)^\top\\
		\hline \mathcal J_{ P_{[n]}}(S) &H_{[n]}
		\end{array}
		\right).
	\end{align*}
\end{theorem}

\begin{proof}
	From \eqref{eq:additive} we conclude that
\begin{align*}
\hat P_{[n]}(\x)=P_{[n]}(\x)-\mathcal J_{\hat P_{[n]}}(S)\varXi \big( \mathcal J_{K_{n-1}(\cdot,\x)}(S)\big)^\top
\end{align*}
and, therefore, we deduce
\begin{align*}
\mathcal J_{\hat P_{[n]}}(S)=\mathcal J_{P_{[n]}}(S)-\mathcal J_{\hat P_{[n]}}(S)\varXi \mathcal K_{n-1}(S)
\end{align*}
that is, the unknowns $\mathcal J_{\hat P_{[n]}}(S)$ satisfy  the following linear system 
\begin{align*}
\mathcal J_{\hat P_{[n]}}(S)\big(I_{N_S}+\varXi \mathcal K_{n-1}(S)\big)=\mathcal J_{P_{[n]}}(S).
\end{align*}
Let us prove that the quasidefiniteness of $u$ implies that $I_{N_S}+\varXi \mathcal K_{n-1}(S)$ is not singular, we follow \cite{Teresa2}.
If we assume that $I_{N_s}+\varXi \mathcal K_{n-1}(S)$ is singular there must exist a non-zero vector $\boldsymbol C\in \C^{N_S}$ such that
$\big(I_{N_S}+\varXi \mathcal K_{n-1}(S)\big)\boldsymbol C=0$ and,  consequently,  such that $\mathcal J_{P_{[n]}}(S)\boldsymbol C=0$. 
Now, let us observe that from
%$\mathcal K_{n}(S)=\mathcal K_{n-1}(S)+\big(\mathcal J_{P_{[n]}}(S)\big)^\top (H_{[n]})^{-1} \mathcal J_{P_{[n]}}(S)$
%and hence 
\begin{align*}
I_{N_s}+\varXi \mathcal K_{n}(S)=I_{N_s}+\varXi \mathcal K_{n-1}(S)+\varXi\big(\mathcal J_{P_{[n]}}(S)\big)^\top (H_{[n]})^{-1} \mathcal J_{P_{[n]}}(S).
\end{align*}
we get $\big(I_{N_s}+\varXi \mathcal K_{n}(S)\big)\boldsymbol C=0$ and, consequently, $\mathcal J_{P_{[n+1]}}(S)\boldsymbol C=0$. By induction, we deduce that
$\mathcal J_{P_{[l]}}(S)\boldsymbol C=0$, for $l\in\{n,n+1,\dots\}$
%This means that
%\begin{align*}
%\sum_{i=1}^q \sum_{\q \leq \boldsymbol\beta_i} C_{i,\q}\dfrac{1}{\q!}\dfrac{\partial ^{|\q|}P_{[l]}}{\partial \x^\q}(\x_i)=0
%\end{align*}
; i.e., the generalized function $\mathcal J(S)\boldsymbol C:=\sum\limits_{i=1}^q \sum\limits_{\q \leq \boldsymbol\beta_i}\dfrac{(-1)^{|\q|}}{\q!} C_{i,\q}\delta^{\q}(\x-\x_i)$
is such that $\big\langle \mathcal J(S)\boldsymbol C,P_{[l]}(\x)\big\rangle=0$, for $l\in\{n,n+1,\dots\}$.  
Equivalently, that is to say
\begin{align*}
\mathcal J(S)\boldsymbol C&\in\Big(\{P_{\q}\}_{|\q|\geq n}\Big)^\perp=\big\{\tilde u \in\big(\C[\x]\big)^*: \langle\tilde u, P_\q\rangle =0, \forall\q\in\mathbb Z_+^D : |\q|\geq n\big\}.
%\\
%&=\big\{P_\q^*\}_{\substack{\q\in\Z_+^D\\|\q|<n}},
\end{align*}
In the one hand, the orthogonality relations for the MVOPR  leads us to conclude that  $\Big(\{P_{\q}\}_{|\q|\geq n}\Big)^\perp=\C\big\{P_\q^*\}_{|\q|<n}$, with the covectors defined by  $\langle P^*_\q,P_{\boldsymbol\beta }\rangle= \delta_{\q,\boldsymbol{\beta}}$; thus, $\dim \Big(\{P_{\q}\}_{|\q|\geq n}\Big)^\perp=N_{n-1}$.
In the other hand, we notice that $\C\big\{\x^\q u\big\}_{|\q|<n}\subset \Big(\{P_{\q}\}_{|\q|\geq n}\Big)^\perp$, but as
$ \dim\C\big\{\x^\q u\big\}_{|\q|<n}=N_{n-1}$  we deduce
$ \Big(\{P_{\q}\}_{|\q|\geq n}\Big)^\perp=\C\big\{\x^\q u\big\}_{|\q|<n}$.
 Consequently, there exists a non-zero polynomial $Q(\x)\in\C[\x]$, $\deg Q\leq n-1$, such that $\mathcal J(S)\boldsymbol C=Qu$. 
 Finally, let us notice that $\prod\limits_{i=1}^q(\x-\x_i)^{\boldsymbol\beta_i}\in\operatorname{Ker}(\mathcal J(S)\boldsymbol C)$, so that
 $\prod\limits_{i=1}^q(\x-\x_i)^{\boldsymbol\beta_i}Q(\x) u=0$,
 and $u$ can not be quasidefinite, in contradiction with the initial assumptions.
 
Now, as $I_{N_s}+\varXi \mathcal K_{n-1}(S)$ is not singular we deduce
\begin{align*}
\mathcal J_{\hat P_{[n]}}(S)=\mathcal J_{P_{[n]}}(S)\big(I_{N_s}+\varXi \mathcal K_{n-1}(S)\big)^{-1}.
\end{align*}
Thus,
\begin{align*}
\hat P_{[n]}(\x)=P_{[n]}(\x)-\mathcal J_{P_{[n]}}(S)\big(I_{N_s}+\varXi \mathcal K_{n-1}(S)\big)^{-1}
\varXi \big( \mathcal J_{K_{n-1}(\cdot,\x)}(S)\big)^\top
\end{align*}
and the result follows.
From \eqref{eq:additive2} we have
\begin{align*}
\hat H_{[n]}&=H_{[n]}+\mathcal J_{\hat P_{[n]}}(S)\varXi \big( \mathcal J_{P_{[n]}}(S)\big)^\top\\
&=H_{[n]}+\mathcal J_{P_{[n]}}(S)\big(I_{N_s}+\varXi \mathcal K_{n-1}(S)\big)^{-1}\varXi \big( \mathcal J_{P_{[n]}}(S)\big)^\top.
\end{align*}

\end{proof}
\begin{cor}[0D-Uvarov mass perturbation]
	Given the set of pairs $\{\x_i,\xi_i\}_{I=1}^q$, positions and masses,  and a discrete additive mass perturbation of the form
	\begin{align*}
	\hat u=u+\sum_{i=1}^q 
	\xi_i\delta(\x-\x_i),
	\end{align*}
	the new MVOPR and quasi-tau matrices are given by the following quasi-determinantal expressions
	\begin{align*}
\hat P_{[n]}(\x)=\Theta_*\begin{bmatrix}
1+\xi_1 K_{n-1}(\x_1,\x_1) & K_{n-1}(\x_1,\x_2) &\dots &K_{n-1}(\x_1,\x_q) &\xi_1K_{n-1}(\x_1,\x)\\
K_{n-1}(\x_2,\x_1) &1+\xi_2 K_{n-1}(\x_2,\x_2) &\dots &K_{n-1}(\x_2,\x_q) &\xi_2K_{n-1}(\x_2,\x)\\ 
\vdots &&\ddots & &\vdots\\  
K_{n-1}(\x_q,\x_1) &K_{n-1}(\x_q\x_2) &\dots &1+\xi_q K_{n-1}(\x_q,\x_q) &\xi_qK_{n-1}(\x_q,\x)\\ 
P_{[n]}(\x_1)&P_{[n]}(\x_2)&\dots & P_{[n]}(\x_q) & P_{[n]}(\x)
\end{bmatrix},\\
\hat H_{[n]}=\Theta_*\begin{bmatrix}
1+\xi_1 K_{n-1}(\x_1,\x_1) & K_{n-1}(\x_1,\x_2) &\dots &K_{n-1}(\x_1,\x_q) &-\xi_1 \big(P_{[n]}(\x_1) \big)^\top\\
K_{n-1}(\x_2,\x_1) &1+\xi_2 K_{n-1}(\x_2,\x_2) &\dots &K_{n-1}(\x_2,\x_q) &-\xi_2 \big(P_{[n]}(\x_2) \big)^\top\\ 
\vdots &&\ddots & &\vdots\\  
K_{n-1}(\x_q,\x_1) &K_{n-1}(\x_q,\x_2) &\dots &1+\xi_q K_{n-1}(\x_q,\x_q) &-\xi_q \big(P_{[n]}(\x_q) \big)^\top\\ 
P_{[n]}(\x_1)&P_{[n]}(\x_2)&\dots & P_{[n]}(\x_q) & H_{[n]}
\end{bmatrix}.
\end{align*}
\end{cor}
\begin{proof}
 We take $S=S_0=\big\{\x_i,\boldsymbol \beta_i\big\}_{i=1}^q $ with $\boldsymbol \beta_i=\q^{(0)}_1$, for $i\in\{1,\dots, q\}$, i.e., we have no derivatives in the delta functions, we have $N_S=q$ and
	\begin{align*}
	\mathcal J_{P_{[n]}}(S_0)&=\begin{bmatrix}
	P_{[n]}(\x_1),\dots,P_{[n]}(\x_q)
	\end{bmatrix},\\
	\xi^{(i)}&=\xi_i\in\mathbb C, \\ \varXi&=\diag(\xi_1,\dots,\xi_q)\\
	\mathcal K_{n-1}(S_0)&=\begin{bmatrix}
	P^{[n]}(\x_1),\dots,P^{[n]}(\x_q)
	\end{bmatrix}^\top\big(H^{[n]}\big)^{-1}\begin{bmatrix}
	P^{[n]}(\x_1),\dots,P^{[n]}(\x_q)
	\end{bmatrix}=\begin{bmatrix}K_{n-1}(\x_i,\x_j)\end{bmatrix},
	\\
	\mathcal J_{K_{n-1}(\cdot,\x)}(S_0)&=\begin{bmatrix}
	K_{n-1}(\x_1,\x),\dots,K_{n-1}(\x_q,\x)
	\end{bmatrix}
	\end{align*}
	and the result follows.
\end{proof}
This result was discussed by Uvarov in \S 2 of \cite{Uvarov1969connection} and its  multivariate extension was presented in \cite{Teresa0}.

We now illustrate the general  0D-Uvarov transformations formul\ae{} by considering the addition of first order derivatives, in \emph{physical language} the addition of  dipoles instead of masses.
We consider   $D$-dimensional  gradient operator $\nabla =\begin{bmatrix}
\dfrac{\partial}{\partial x_1},\dots,\dfrac{\partial}{\partial x_D}
\end{bmatrix}$. In terns of it, for each vector $\boldsymbol\xi=(\xi_1,\dots,\xi_D)^\top\in\C^D$ we have the normal derivative
$\nabla_{\boldsymbol{\xi}}=\sum\limits_{a=1}^D\xi_a\dfrac{\partial }{\partial x_a}$. Finally, given a function  $K(\x,\y):\R^D\times\R^D\to\C$ we denote by
\begin{align*}
K^{(\nabla,0)}(\x,\y)&:=\begin{bmatrix}
\dfrac{\partial K}{\partial x_1}(\x,\y),\dots,\dfrac{\partial K}{\partial x_D}(\x,\y)
\end{bmatrix}^\top,\\
K^{(0,\nabla)}(\x,\y)&:=\begin{bmatrix}
\dfrac{\partial K}{\partial y_1}(\x,\y),\dots,\dfrac{\partial K}{\partial y_D}(\x,\y)
\end{bmatrix},\end{align*}
\begin{align*}
K^{(\nabla,\nabla)}(\x,\y)&:=\begin{bmatrix}
\dfrac{\partial^2 K}{\partial x_1\partial y_1}(\x,\y)&\dots &\dfrac{\partial^2 K}{\partial x_1\partial y_D}(\x,\y)\\
\vdots & &\vdots\\
\dfrac{\partial^2 K}{\partial x_D\partial y_1}(\x,\y)&\dots &\dfrac{\partial^2 K}{\partial x_D\partial y_D}(\x,\y)
\end{bmatrix},\\
K^{(\boldsymbol{\xi},0)}(\x,\y)&:=
\sum_{a=1}^D\xi_a\frac{\partial K}{\partial x_a}(\x,\y),\\
K^{(\boldsymbol{\xi},\nabla)}(\x,\y)&:=\begin{bmatrix}
\sum\limits_{a=1}^D\xi_a\dfrac{\partial^2 K}{\partial x_a\partial y_1}(\x,\y),\dots,\sum\limits_{a=1}^D\xi_a\dfrac{\partial^2 K}{\partial x_a\partial y_D}(\x,\y)
\end{bmatrix}.
\end{align*}

\begin{cor}[0D-Uvarov dipole perturbation]
		Given couples  of vectors (positions and strength of the dipoles) $S_1=\{\x_i,\boldsymbol\xi_i\}_{i=1}^q$ and a corresponding discrete additive dipolar  perturbation of the form
		\begin{align*}
		\hat u=u+\sum_{i=1}^q 
		\nabla_{\boldsymbol \xi_i}\delta(\x-\x_i),
		\end{align*}
		the new MVOPR and quasi-tau matrices are given by the following quasi-determinantal expressions
		\small	\begin{align*}
	\hspace*{-10cm}	\hat P_{[n]}(\x)&=\Theta_*\left[
		\begin{array}{cc|c|cc|c}
		1+K^{(\boldsymbol{\xi}_1,0)}_{n-1}(\x_1,\x_1)&K^{(\boldsymbol{\xi}_1,\nabla)}_{n-1}(\x_1,\x_1)& \multirow{2}[4]*{ \dots} & K^{(\boldsymbol{\xi}_1,0)}_{n-1}(\x_1,\x_q)&K^{(\boldsymbol{\xi}_1,\nabla)}_{n-1}(\x_1,\x_q)&K^{(\boldsymbol{\xi}_1,0)}_{n-1}(\x_1,\x)\\
		\boldsymbol \xi_1 	K_{n-1}(\x_1,\x_1)&I_D+\boldsymbol \xi_1 K^{(0,\nabla)}_{n-1}(\x_1,\x_1)&& 	\boldsymbol \xi_1 	K_{n-1}(\x_1,\x_q)&\boldsymbol \xi_1 K^{(0,\nabla)}_{n-1}(\x_1,\x_q)&	\boldsymbol \xi_1 	K_{n-1}(\x_1,\x)\\\hline
		\multicolumn{2}{c|}{\vdots }&&\multicolumn{2}{c|}{\vdots }&\vdots\\\hline
		K^{(\boldsymbol{\xi}_q,0)}_{n-1}(\x_q,\x_1)&K^{(\boldsymbol{\xi}_q,\nabla)}_{n-1}(\x_q,\x_1)& \multirow{2}*{ \dots} & 1+K^{(\boldsymbol{\xi}_q,0)}_{n-1}(\x_q,\x_q)&K^{(\boldsymbol{\xi}_q,\nabla)}_{n-1}(\x_q,\x_q)&K^{(\boldsymbol{\xi}_q,0)}_{n-1}(\x_q,\x)\\
		\boldsymbol \xi_q 	K_{n-1}(\x_q,\x_1)&\boldsymbol \xi_q K^{(0,\nabla)}_{n-1}(\x_q,\x_1)&& 	\boldsymbol \xi_q 	K_{n-1}(\x_q,\x_q)&I_D+\boldsymbol \xi_q K^{(0,\nabla)}_{n-1}(\x_q,\x_q)&	\boldsymbol \xi_q 	K_{n-1}(\x_q,\x)\\\hline	
		P_{[n]}(\x_1) &	\nabla P_{[n]}(\x_1)&\dots& P_{[n]}(\x_q)&\nabla P_{[n]}(\x_q)&P_{[n]}(\x)
		\end{array}
		\right],\\
		\hat H_{[n]}&=\Theta_*\left[
		\begin{array}{cc|c|cc|c}
		1+K^{(\boldsymbol{\xi}_1,0)}_{n-1}(\x_1,\x_1)&K^{(\boldsymbol{\xi}_1,\nabla)}_{n-1}(\x_1,\x_1)& \multirow{2}[4]*{ \dots} & K^{(\boldsymbol{\xi}_1,0)}_{n-1}(\x_1,\x_q)&K^{(\boldsymbol{\xi}_1,\nabla)}_{n-1}(\x_1,\x_q)&-\nabla_{\boldsymbol{\xi_1}} P_{[n]}(\x)\\
		\boldsymbol \xi_1 	K_{n-1}(\x_1,\x_1)&I_D+\boldsymbol \xi_1 K^{(0,\nabla)}_{n-1}(\x_1,\x_1)&& 	\boldsymbol \xi_1 	K_{n-1}(\x_1,\x_q)&\boldsymbol \xi_1 K^{(0,\nabla)}_{n-1}(\x_1,\x_q)&-\boldsymbol{\xi}_1 P_{[n]}(\x_1)\\\hline
		\multicolumn{2}{c|}{\vdots }&&\multicolumn{2}{c|}{\vdots }&\vdots\\\hline
		K^{(\boldsymbol{\xi}_q,0)}_{n-1}(\x_q,\x_1)&K^{(\boldsymbol{\xi}_q,\nabla)}_{n-1}(\x_q,\x_1)& \multirow{2}*{ \dots} & 1+K^{(\boldsymbol{\xi}_q,0)}_{n-1}(\x_q,\x_q)&K^{(\boldsymbol{\xi}_q,\nabla)}_{n-1}(\x_q,\x_q)&-\nabla_{\boldsymbol{\xi_q}} P_{[n]}(\x),\\
		\boldsymbol \xi_q 	K_{n-1}(\x_q,\x_1)&\boldsymbol \xi_q K^{(0,\nabla)}_{n-1}(\x_q,\x_1)&& 	\boldsymbol \xi_q 	K_{n-1}(\x_q,\x_q)&I_D+\boldsymbol \xi_1 K^{(0,\nabla)}_{n-1}(\x_q,\x_q)&-\boldsymbol{\xi}_q P_{[n]}(\x_q)\\\hline	
		P_{[n]}(\x_1) &	\nabla P_{[n]}(\x_1)&\dots& P_{[n]}(\x_q)&\nabla P_{[n]}(\x_q)&H_{[n]}		\end{array}
		\right].
		\end{align*}
\end{cor}

\begin{proof}
	In this case we take $S=S_1=\big\{\x_i,\boldsymbol \beta_i\big\}_{i=1}^q $ with $\boldsymbol \beta_i=\q^{(1)}_D$, for $i\in\{1,\dots, q\}$,
	and choose
	\begin{align*}
	\xi_{i,\q^{(0)}_1}&=0, & \boldsymbol{\xi}_i&=\begin{bmatrix}
	\xi_{i,\q^{(1)}_1}\\	\vdots\\\xi_{i,\q^{(1)}_D}
	\end{bmatrix}\in\C^D.
	\end{align*}
we have $N_S=q(D+1)$ and
	\begin{align*}
	\mathcal J_{P_{[n]}}(S_1)&=\begin{bmatrix}
	P_{[n]}(\x_1),	\nabla P_{[n]}(\x_1),\dots,P_{[n]}(\x_q),\nabla P_{[n]}(\x_q)
	\end{bmatrix},\\
	\xi^{(i)}&=\begin{bmatrix}
	0 &(\boldsymbol{\xi}_i)^\top\\
	\boldsymbol \xi_i & 0_D
	\end{bmatrix}\in\mathbb C^{(D+1)\times (D+1)}, \\ \varXi&=\diag(\xi^{(1)},\dots,\xi^{(q)})\in \mathbb C^{q(D+1)\times q(D+1)},	\\
\mathcal J_{K_{n-1}(\cdot,\x)}(S_1)&=\begin{bmatrix}
K_{n-1}(\x_1,\x), \big(K^{(\nabla,0)}_{n-1}(\x_1,\x)\big)^\top, \dots, K_{n-1}(\x_q,\x),\big(K^{(\nabla,0)}_{n-1}(\x_q,\x)\big)^\top
\end{bmatrix},	\end{align*}
	\begin{align*}
	\mathcal K_{n-1}(S_1)&=\begin{multlined}[t][0.7\textwidth]
	\begin{bmatrix}
	P^{[n]}(\x_1),	\nabla P^{[n]}(\x_1),\dots,P^{[n]}(\x_q),\nabla P^{[n]}(\x_q)
	\end{bmatrix}^\top\big(H^{[n]}\big)^{-1}\\	\begin{bmatrix}
	P^{[n]}(\x_1),\nabla P^{[n]}(\x_1),\dots,P^{[n]}(\x_q),\nabla P^{[n]}(\x_q)
	\end{bmatrix}
	\end{multlined}\\&=
	\left[
	\begin{array}{cc|c|cc}
	K_{n-1}(\x_1,\x_1)&K^{(0,\nabla)}_{n-1}(\x_1,\x_1)& \multirow{2}{*}{ \dots} & K_{n-1}(\x_1,\x_q)&K^{(0,\nabla)}_{n-1}(\x_1,\x_q)\\
	K^{(\nabla,0)}_{n-1}(\x_1,\x_1)&K^{(\nabla,\nabla)}_{n-1}(\x_1,\x_1)&& 	K^{(\nabla,0)}_{n-1}(\x_1,\x_q)&K^{(\nabla,\nabla)}_{n-1}(\x_1,\x_q)\\\hline
\multicolumn{2}{c|}{\vdots }&&\multicolumn{2}{c}{\vdots }\\\hline
K_{n-1}(\x_q,\x_1)&K^{(0,\nabla)}_{n-1}(\x_q,\x_1)&  \multirow{2}{*}{ \dots} & K_{n-1}(\x_q,\x_q)&K^{(0,\nabla)}_{n-1}(\x_q,\x_q)\\
K^{(\nabla,0)}_{n-1}(\x_q,\x_1)&K^{(\nabla,\nabla)}_{n-1}(\x_q,\x_1)&& 	K^{(\nabla,0)}_{n-1}(\x_q,\x_q)&K^{(\nabla,\nabla)}_{n-1}(\x_q,\x_q)
	\end{array}
	\right].
	\end{align*}
%\textbf{Paso intermedio}To compute $\varXi\mathcal K_{n-1}(S_1)$ we first calculate 
%\begin{align*}
%\begin{bmatrix}
%0 &(\boldsymbol{\xi}_i)^\top\\
%\boldsymbol \xi_i & 0_D
%\end{bmatrix}\begin{bmatrix}
%	K_{n-1}(\x_i,\x_j)&K^{(0,\nabla)}_{n-1}(\x_i,\x_j)\\
%	\big(K^{(\nabla,0)}_{n-1}(\x_i,\x_j)\big)^\top&K^{(\nabla,\nabla)}_{n-1}(\x_i,\x_j)
%\end{bmatrix}=\begin{bmatrix}
%K^{(\boldsymbol{\xi}_i,0)}_{n-1}(\x_i,\x_j)&K^{(\boldsymbol{\xi}_i,\nabla)}_{n-1}(\x_i,\x_j)\\
%\boldsymbol \xi_i 	K_{n-1}(\x_i,\x_j)&\boldsymbol \xi_i K^{(0,\nabla)}_{n-1}(\x_i,\x_j)
%\end{bmatrix}
%\end{align*}
Therefore, we compute	
	\begin{multline*}
	I_{N_{S_1}}+\varXi\mathcal K_{n-1}(S_1)\\=
		\left[
		\begin{array}{cc|c|cc}
1+K^{(\boldsymbol{\xi}_1,0)}_{n-1}(\x_1,\x_1)&K^{(\boldsymbol{\xi}_1,\nabla)}_{n-1}(\x_1,\x_1)& \multirow{2}[4]*{ \dots} & K^{(\boldsymbol{\xi}_1,0)}_{n-1}(\x_1,\x_q)&K^{(\boldsymbol{\xi}_1,\nabla)}_{n-1}(\x_1,\x_q)\\
	\boldsymbol \xi_1 	K_{n-1}(\x_1,\x_1)&I_D+\boldsymbol \xi_1 K^{(0,\nabla)}_{n-1}(\x_1,\x_1)&& 	\boldsymbol \xi_1 	K_{n-1}(\x_1,\x_q)&\boldsymbol \xi_1 K^{(0,\nabla)}_{n-1}(\x_1,\x_q)\\\hline
		\multicolumn{2}{c|}{\vdots }&&\multicolumn{2}{c}{\vdots }\\\hline
			K^{(\boldsymbol{\xi}_q,0)}_{n-1}(\x_q,\x_1)&K^{(\boldsymbol{\xi}_q,\nabla)}_{n-1}(\x_q,\x_1)& \multirow{2}*{ \dots} & 1+K^{(\boldsymbol{\xi}_q,0)}_{n-1}(\x_q,\x_q)&K^{(\boldsymbol{\xi}_q,\nabla)}_{n-1}(\x_q,\x_q)\\
			\boldsymbol \xi_q 	K_{n-1}(\x_q,\x_1)&\boldsymbol \xi_q K^{(0,\nabla)}_{n-1}(\x_q,\x_1)&& 	\boldsymbol \xi_q 	K_{n-1}(\x_q,\x_q)&I_D+\boldsymbol \xi_1 K^{(0,\nabla)}_{n-1}(\x_q,\x_q)
		\end{array}
		\right],
	\end{multline*}
	and the result follows.
%	Consequently

%	begin{bmatrix}
%	), , \dots, K_{n-1}(\x_1,\x),K^{(0,\nabla)}_{n-1}(\x_1,\x)
%	\end{bmatrix
	
%
\end{proof}

%
%Then, from \eqref{eq:additive} we get
%	\begin{align*}
%	\hat P_{[n]}(\x)
%	&=P_{[n]}(\x)	-\sum_{\substack{\q\in(\Z_+)^D\\j\in\Z_+}} 	\frac{(-1)^{|\q|}}{\q!}\xi_{\q,j}\frac{\partial^{|\q|}\big(\hat P_{[n]}(\y)K_{n-1}(\y,\x)\big)}{\partial \y^\q}\Big|_{\y=\x_{\q,j }}\\
%	&=P_{[n]}(\x)	-\sum_{\substack{\q\in(\Z_+)^D\\j\in\Z_+}}\sum_{|\boldsymbol{\beta}|\leq |\q|} 	\frac{(-1)^{|\q|}}{\q!}\xi_{\q,j}
%	\frac{\partial^{|\q-\boldsymbol{\beta}|}\big(K_{n-1}(\y,\x)\big)}{\partial \y^{\q-\boldsymbol{\beta}}}\Big|_{\y=\x_{\q,j} }
%		\frac{\partial^{|\boldsymbol{\beta}|}\big(\hat P_{[n]}(\y)\big)}{\partial \y^{\boldsymbol{\beta}}}\Big|_{\y=\x_{\q,j} }
%	\end{align*}		

\subsection{1D-Uvarov perturbations and Fredholm integral equations}

We have discussed 0-dimensional  additive perturbations of $D$-dimensional generalized functions in full generality. 
However, we reckon that this is a very limited analysis, as in this multivariate context much more general perturbations do exist, as is illustrated by   \eqref{eq:v_general}.
We now discuss a very particular example, adding a 1D massive string. For this aim we assume that we have a parametrized curve, i.e. a smooth map from the interval $I\subset\R$ to $\R^D$:
\begin{align*}
\boldsymbol\gamma:I\to\R^D,
\end{align*}
as well as  a weight function $w:I\to\C$. Then, the linear functional $v$ is
\begin{align*}
\langle v, P\rangle =\int_I P(\boldsymbol\gamma (t)) w(t)\d t.
\end{align*}
Recalling 	\eqref{eq:additive} we can write
\begin{align}\label{eq:Pnadditive}
\hat P_{[n]}(\x)=P_{[n]}(\x)-\int_I \hat P_{[n]}\big(\boldsymbol\gamma (s)\big) K_{n-1}(\boldsymbol{\gamma}(s),\x)w(s)\d s.
\end{align}
Now, let us remark one of the basic ideas in the proof of Theorem \ref{teo:discrete uvarov}. First, one uses \eqref{eq:additive} and then evaluates on the support of the distribution. In that case, we evaluated again at the points where the delta functions and its derivatives where supported. In this case, we should evaluate it again at the curve $\boldsymbol \gamma$.  
\begin{definition}
	We introduce some notation
\begin{align*}
\hat \pi_{[n]}(t)&:=\hat P_{[n]}(\boldsymbol\gamma (t)), & \pi_{[n]}(t)&:= P_{[n]}(\boldsymbol\gamma (t)), &
 \kappa_{n-1}(t,s)&:= K_{n-1}\big(\boldsymbol\gamma (t),\boldsymbol\gamma(s)\big)w(s).
\end{align*}
\end{definition}
Then, \eqref{eq:Pnadditive} implies the following integral Fredholm equation
\begin{align}\label{eq:fredholm}
\hat \pi_{[n]}(t)=\pi_{[n]}(t)-\int_I \hat \pi_{[n]}(s)\kappa_{n-1}(s,t)\d s.
\end{align}
This integral equation, having as integral kernel $\kappa_{n-1}(t,s)$ a separable one, can be solved explicitly. In fact, % it can be proven that
\begin{pro}
	The solution of the separable Fredholm equation \eqref{eq:fredholm} can be expressed as a last quasi-determinant as follows
	\begin{align*}
\hat{\pi}_{[n]}(t)=\Theta_*\left[\begin{array}{ccc|c}
H_{[0]}+\int_I  \pi_{[0]}(s)\big(\pi_{[0]}(s)\big)^\top w(s)\d s & \dots & \int_I  \pi_{[0]}(s)\big(\pi_{[n-1]}(s)\big)^\top w(s)\d s&\pi_{[0]}(t)\\
\int_I  \pi_{[1]}(s)\big(\pi_{[0]}(s)\big)^\top w(s)\d s & \dots & \int_I  \pi_{[1]}(s)\big(\pi_{[n-1]}(s)\big)^\top w(s)\d s&\pi_{[1]}(t)\\
\vdots & &\vdots&\vdots\\
\int_I  \pi_{[n-1]}(s)\big(\pi_{[0]}(s)\big)^\top w(s)\d s & \dots & H_{[n-1]}+\int_I  \pi_{[n-1]}(s)\big(\pi_{[n-1]}(s)\big)^\top w(s)\d s&\pi_{[n-1]}(t)\\[10pt]
\hline
\int_I  \pi_{[n]}(s)\big(\pi_{[0]}(s)\big)^\top w(t)\d s & \dots & \int_I  \pi_{[n]}(s)\big(\pi_{[n-1]}(s)\big)^\top w(s)\d s&\pi_{[n]}(t)
\end{array}\right].
\end{align*}
\end{pro}
\begin{proof}
The separability of the kernel means that the kernel $\kappa_{n-1}(t,s)$ can be written 
\begin{align*}
\kappa_{n-1}(s,t)=\sum_{m=0}^{n-1}\big(\pi_{[m]}(s)\big)^\top \big(H_{[m]}\big)^{-1} \pi_{[m]}(t).
\end{align*}
Now,  with the notation
\begin{align}\label{eq:C}
C_{[n],[m]}:=\int_I \hat\pi_{[n]}(s)\big(\pi_{[m]}(s)\big)^\top w(s)\d s
\end{align}
we can write the Fredholm equation 
\begin{align*}
\hat\pi_{[n]}(t)=\pi_{[n]}(t)-\sum_{m=0}^{n-1}C_{[n],[m]}\big(H_{[m]}\big)^{-1}\pi_{[m]}(t),
\end{align*}
that can be introduced in \eqref{eq:C} to get 
\begin{align*}
C_{[n,[m]}&=A_{[n],[m]}-\sum_{l=0}^{n-1}C_{[n],[l]}\big(H_{[l]}\big)^{-1} A_{[l],[m]}, & A_{[n],[m]}&:=\int_I\pi_{[n]}(s)\big(\pi_{[m]}(s)\big)^\top w(s)\d s.
\end{align*}
This linear system for the $C$'s can be written
\begin{align*}
\begin{bmatrix}
C_{[n],[0]},\dots,C_{[n],[n-1]}
\end{bmatrix}=\begin{bmatrix}
A_{[n],[0]},\dots,_{[n],[n-1]}
\end{bmatrix}-\begin{bmatrix}
C_{[n],[0]},\dots,C_{[n],[n-1]}\
\end{bmatrix}
\big(H^{[n]}\big)^{-1}
\begin{bmatrix}
A_{[0],[0]} &\dots& A_{[0],[n-1]}\\
\vdots & & \vdots\\
A_{[n-1],[0]} &\dots& A_{[n-1],[n-1]}
\end{bmatrix},
\end{align*}
and the result follows.
\end{proof}

\begin{pro}
	Given the solution $\hat{\pi}_{[n]}(t)$ to the Fredholm equation \eqref{eq:fredholm} we find the perturbed MVOPR and squared norms can be expressed
\begin{align*}
\hat P_{[n]}(\x)&=
P_{[n]}(\x)-\int_I \hat \pi_{[n]} (t)K_{n-1}(\boldsymbol{\gamma}(t),\x)w(t)\d t,\\
\hat H_{[n]}&=H_{[n]}+\int_I \hat\pi_{[n]}(t)\big(\pi_{[n]}(t)\big)^\top w(t)\d t.
\end{align*}

\end{pro}

%begin{align}\label
%\hat P_{[n]}(\x)
%&=P_{[n]}(\x)	-\left\langle v, \hat P_{[n]}(\y)K_{n-1}(\y,\x)\right\rangle,\\
%\hat H_{[n]}				&=	H_{[n]}+\left\langle v, \hat P_{[n]}(\x)\big(P_{[n]}(\x)\big)^\top\right\rangle. \label{eq:additive2}
%\end{align}

% \printbibliography
	
\end{document}